\documentclass[leqno]{amsart}
\usepackage[utf8]{inputenc}
\usepackage[T1]{fontenc}

\usepackage{boilerplate-paramalg}

\usepackage[margin=1in]{geometry}

\usepackage{lmodern}
\usepackage{wrapfig}
\usepackage{csquotes}

\usepackage[shortlabels]{enumitem}

\usepackage{tikz-cd}
\usepackage{mathrsfs}
\usepackage{sseq}
\usepackage{mathdots}
\usepackage{thm-restate}
\usepackage{amsmath}
\usepackage{amsxtra}
\usepackage{amscd}
\usepackage{amsthm}
\usepackage{amsfonts}
\usepackage{amssymb}
\usepackage{mathtools}
\usepackage[scr=rsfs]{mathalfa}
\usepackage{eucal}
\usepackage{bbm}
\usepackage[all,cmtip]{xy}

\usepackage{latexsym}
\usepackage{mdwlist}

\usepackage{xifthen}

\title{Parametrized higher category theory II: Universal constructions}

\author{Jay Shah}
\address{Fachbereich Mathematik und Informatik, WWU Münster, 48149 M\"{u}nster, Germany}
\email{jayhshah@gmail.com}

\begin{document}

\tikzcdset{arrow style=tikz, diagrams={>=stealth}}

\begin{abstract} 
We develop parametrized generalizations of a number of fundamental concepts in the theory of $\infty$-categories, including factorization systems, free fibrations, exponentiable fibrations, relative colimits and relative Kan extensions, filtered and sifted diagrams, and the universal constructions $\Ind$ and $\PShv^{\Sigma}$.
\end{abstract}

\date{\today}
\maketitle

\tableofcontents

\section{Introduction}

In this paper, we continue our development of the foundations of \emph{parametrized} (i.e., \emph{indexed}) higher category theory from \cite{Exp2}. Let $\cT$ be an $\infty$-category.

\begin{dfn}
A \emph{$\cT$-$\infty$-category} is a cocartesian fibration $\cC \to \cT^{\op}$. Given two $\cT$-$\infty$-categories $\cC$ and $\cD$, a \emph{$\cT$-functor} $F: \cC \to \cD$ is a morphism of cocartesian fibrations, i.e., a functor over $\cT^{\op}$ that preserves cocartesian edges.
\end{dfn} 


\begin{wrn}
In \cite{Exp2}, we set $S = \cT^{\op}$ and instead spoke of $S$-$\infty$-categories as cocartesian fibrations $\cC \to S$. As this is purely an issue of nomenclature, we will not hesitate in referring to results from \cite{Exp2} with our opposite convention in force.
\end{wrn}

The basic idea of parametrized higher category theory is to develop a theory of $\infty$-categories internal to the $(\infty,2)$-category of $\cT$-$\infty$-categories. The most fundamental new complication that arises is that of a broader notion of \emph{point}; points should now be thought of as encompassing all the corepresentable left fibrations over $\cT^{\op}$. For example, taking $\cT = \OO_G$ to be the orbit category of a finite group, the theory of $G$-colimits essentially amalgamates the usual theory of colimits together with that of coproducts indexed by $G$-orbits.\footnote{See \cref{thmx:DecompositionMainResult} for a precise statement.} Our original motivation for this project lay in the necessity of having robust $\infty$-categorical foundations for equivariant homotopy theory -- see \cite{Exp0} and the introduction of \cite{Exp2} for more details on this. However, nothing in \cite{Exp2} or this paper is specific to that application. In principle, the foundational work that we undertake here should prove useful wherever classical indexed category theory has found application, or for base $\infty$-categories $\cT$ of algebro-geometric origin (e.g., in a motivic context). It will also be essential for our development of the theory of \emph{parametrized $\infty$-operads} in \cite{paramalg}, which underpins the work of Horev and his collaborators \cite{horev2019enuine,hahn2020quivariant} on equivariant factorization homology.

Recall that in \cite{Exp2} we accomplished the following primary objectives:

\begin{enumerate}
\item We introduced the concepts of $\cT$-(co)limits and $\cT$-Kan extensions.\footnote{We give a rapid review of these concepts in \cref{sec:old}.} We also proved the basic existence and uniqueness theorem for $\cT$-Kan extensions (cf. \cite[Thm.~10.3]{Exp2} and \cite[Thm.~10.5]{Exp2}).
\item Say that $\cT$ is \emph{orbital} if its finite coproduct completion $\FF_{\cT}$ admits pullbacks. Supposing that $\cT$ is orbital, we proved as \cite[Cor.~12.15]{Exp2} that a $\cT$-$\infty$-category $\cC$ is $\cT$-cocomplete \cite[Def.~5.13]{Exp2} if and only if $\cC$ admits all $\cT$-coproducts \cite[Def.~5.10]{Exp2}, fiberwise geometric realizations, and the restriction functors preserve geometric realizations. This was done by a $\cT$-colimit decomposition technique in the form of the parametrized Bousfield--Kan formula; cf. \cite[Thm.~12.6]{Exp2} and \cite[Thm.~12.13]{Exp2} coupled with the parametrized Quillen's Theorem A \cite[Thm.~6.7]{Exp2}.
\item We proved a parametrized Yoneda lemma \cite[Lem.~11.1]{Exp2} and subsequently established the universal property of the $\cT$-$\infty$-category of presheaves \cite[Thm.~11.5]{Exp2}.
\end{enumerate}

For more involved applications, we need to establish generalizations of all three of these results. Firstly, recall that Lurie in \cite[\S 4.3]{HTT} set up a theory of \emph{relative Kan extensions}. The idea is that given a commutative diagram
\[ \begin{tikzcd}
\cC \ar{r}{F} \ar[hookrightarrow]{d}[swap]{i} & \cE \ar{d}{\pi} \\
\cD \ar{r} \ar[dotted]{ru} & \cB
\end{tikzcd} \]
of $\infty$-categories, one can give a pointwise criterion for the existence of an initial filler $i_! F$. If $\pi$ is in addition a cocartesian fibration, then as a corollary one sees that $i_! F$ always exists if we suppose that for all objects $b \in \cB$, the fiber $\cE_b$ admits all colimits, and for all morphisms $f: b \to b'$, the pushforward functor $f_!: \cE_b \to \cE_{b'}$ preserves all colimits.

We will establish the theory of \emph{relative $\cT$-colimits} and \emph{relative $\cT$-left Kan extensions} and thereby obtain a generalization of Lurie's result in the parametrized setting in the form of \cref{thm-app:RelativeLKEexistence}.\footnote{Of course, one may dualize appropriately to obtain analogous results involving relative $\cT$-limits and relative $\cT$-right Kan extensions; cf. \cite[Cor.~5.25]{Exp2}.} Since the definitions of relative $\cT$-colimit and relative $\cT$-left Kan extension are technically involved (cf. \cref{dfn:relativeColimit} and \cref{def-app:relativeKanExtension}), at this point we will only state a simplified corollary of our main existence result that nonetheless covers the case of most relevance. To formulate the analogous existence criterion in the parametrized context, we need the notion of a \emph{parametrized fiber} of a $\cT$-functor:
 
\begin{dfn}[{\cite[Notn.~2.29]{Exp2}}] 
Let $\cB$ be a $\cT$-$\infty$-category, $b \in \cB_t$ an object, and let $\Ar^{\cocart}(\cB)$ be the full subcategory of $\Ar(\cB)$ on the cocartesian edges in $\cB$. We let $\underline{b} \coloneq \{b\} \times_{\cB, \ev_0} \Ar^{\cocart}(\cB)$. Note then that the functor $\underline{b} \to (\cT^{/t})^{\op} \cong \{ t \} \times_{\cT^{\op}} \Ar(\cT^{\op})$ induced by the structure map of $\cB$ is a trivial fibration \cite[Lem.~12.10]{Exp2} and $\ev_1: \underline{b} \to \cB$ is a $\cT$-functor covering $(\cT^{/t})^{\op} \to \cT^{\op}$.

Now suppose $\pi: \cE \to \cB$ is a $\cT$-functor. The \emph{parametrized fiber} of $\pi$ over $b$ is the $\cT^{/t}$-$\infty$-category
$$\cE_{\underline{b}} \coloneq \underline{b} \times_{\ev_1, \cB, \pi} \cE.$$
\end{dfn}   

\begin{thmx} \label{thmx:Kan}
Suppose we have a commutative diagram of $\cT$-$\infty$-categories
\[ \begin{tikzcd}
\cC \ar{r} \ar[hookrightarrow]{d}[swap]{i} & \cE \ar{d}{\pi} \\
\cD \ar{r} \ar[dotted]{ru} & \cB
\end{tikzcd} \]
in which $i$ is the inclusion of a full $\cT$-subcategory \cite[Def.~2.2]{Exp2} and $\pi$ is in addition a cocartesian fibration. Consider the restriction functor
\[ i^*: \Fun_{/\cB, \cT}(\cD, \cE) \to \Fun_{/\cB, \cT}(\cC, \cE) \]
where $\Fun_{/\cB, \cT}(-, -)$ denotes the full subcategory of $\Fun_{/\cB}(-, -)$ spanned by the $\cT$-functors. Suppose that for all $b \in \cB_t$, the parametrized fiber $\cE_{\underline{b}}$ admits all $\cT^{/t}$-colimits, and for all $f: b \to b' \in \cB_t$, the induced pushforward $\cT^{/t}$-functor $f_!: \cE_{\underline{b}} \to \cE_{\underline{b'}}$ preserves all $\cT^{/t}$-colimits. Then $i^*$ admits a left adjoint $i_!$. Moreover, the unit transformation $\id \Rightarrow i^* i_!$ is an equivalence, so $i_!$ is fully faithful.
\end{thmx}

\begin{rem} As with the ordinary theory of Kan extensions, the full faithfulness assertion in \cref{thmx:Kan} is where the pointwise formula for $i_!$ comes into play. In particular, even if we assumed the relevant presentability hypotheses, it would not suffice to appeal to the adjoint functor theorem to verify this property.
\end{rem}

Secondly, we develop the theory of $\cT$-$\kappa$-small, $\cT$-filtered, and $\cT$-sifted $\cT$-$\infty$-categories. In order to speak of small and large simplicial sets and $\infty$-categories, we henceforth fix two strongly inaccessible cardinals $\delta_0 < \delta_1$.

\begin{convention}
For simplicity, we now also suppose throughout that the base $\infty$-category $\cT$ is small.
\end{convention}

Let $\Cat$ denote the (large) $\infty$-category of small $\infty$-categories and let $\Cat_{\cT} \coloneq \Cat^{\cocart}_{/\cT^{\op}} \simeq \Fun(\cT^{\op}, \Cat)$ be the $\infty$-category of $\cT$-small $\cT$-$\infty$-categories.\footnote{Since we suppose that $\cT$ is small, a $\cT$-$\infty$-category $\cC$ is $\cT$-small if and only if $\cC$ is small.}

\begin{dfn}[\cref{def:Tfinite}] \label{dfn:FiniteEtc}
Let $\Delta_{\cT} \subset \Cat_{\cT}$ be the full subcategory spanned by the objects
\[ \{ \Delta^n \times \Map_{\cT}(-,t) \}_{t\in \cT, n \geq 0}. \]
Then for every regular cardinal $\kappa$, we define the full subcategory $\Cat^{\ksmall}_{\cT} \subset \Cat_{\cT}$ to be the smallest full subcategory that contains $\Delta_{\cT}$ and is closed under all colimits indexed by $\kappa$-small simplicial sets. We say that a $\cT$-small $\cT$-$\infty$-category $\cC$ is \emph{$\cT$-$\kappa$-small} if it belongs to $\Cat^{\ksmall}_{\cT}$. If $\kappa = \omega$, we also say that $\cC$ is \emph{$\cT$-finite}.
\end{dfn}

\begin{rem} \label{rem:annoyingTerminology}
Adopting the terminology of \cref{dfn:FiniteEtc} entails speaking of a host of seemingly redundant expressions such ``$\cT$-finite $\cT$-$\infty$-category''. We avoid simply writing e.g. ``finite $\cT$-$\infty$-category'' because of the possible ambiguity as to whether, given some $\cT$-$\infty$-category $\cC$, ``finite'' refers to $\cC$ being finite as an $\infty$-category or as a $\cT$-$\infty$-category.
\end{rem}


We then have the following generalization of \cite[Cor.~12.15]{Exp2}, whose proof turns out to be far simpler than our earlier strategy of appealing to the parametrized Bousfield--Kan formula. We give the most useful formulation of this here; a slightly more general statement is recorded as \cref{prop:FiberwiseAndCoproductsGiveAll}.

\begin{thmx} \label{thmx:DecompositionMainResult}
Suppose that $\cT$ is orbital. Let $\cC$ be an $\cT$-$\infty$-category and $\kappa$ a regular cardinal. Then $\cC$ strongly admits\footnote{We recall this notion as \cref{def-app:StronglyAdmitsAndPreservesColimit}.} all $\cT$-$\kappa$-small $\cT$-colimits if and only if
\begin{enumerate}
\item For every $t \in \cT$, the fiber $\cC_t$ admits all $\kappa$-small colimits, and for every $\alpha: s \to t$, the restriction functor $\alpha^*: \cC_t \to \cC_s$ preserves $\kappa$-small colimits.
\item For every map $\alpha: U \to V$ of finite $\cT$-sets,\footnote{A finite $\cT$-set is defined to be an object of the finite coproduct completion $\FF_{\cT}$ of $\cT$.} the restriction functor $\alpha^*: \cC_V \to \cC_U$ admits a left adjoint $\alpha_!$.\footnote{For a finite $\cT$-set $U$ with orbit decomposition $U_1 \sqcup ... \sqcup U_n$, we write $\cC_U \coloneq \prod_{i=1}^n \cC_{U_i}$, and the contravariant functoriality in the finite $\cT$-set is inherited from that for orbits.}
\item $\cC$ satisfies the \emph{Beck-Chevalley condition}, i.e., for every pullback square
\[ \begin{tikzcd}
U' \ar{r}{\beta'} \ar{d}{\alpha'} & U \ar{d}{\alpha} \\
V' \ar{r}{\beta} & V
\end{tikzcd} \]
in $\FF_{\cT}$, the mate
\[ \alpha'_! \beta'^* \Rightarrow \beta^* \alpha_!: \cC_U \to \cC_{V'} \]
is an equivalence.
\end{enumerate}
\end{thmx}

\begin{rem}
Again supposing that $\cT$ is orbital, note that by \cite[Prop.~5.12]{Exp2} $\cC$ admits finite $\cT$-coproducts if and only if conditions (2) and (3) in \cref{thmx:DecompositionMainResult} hold. Moreover, the ordinary ($\infty$-categorical) Bousfield--Kan formula shows that an $\infty$-category is cocomplete if and only if it admits coproducts and geometric realizations (cf. \cite[Cor.~12.3]{Exp2}). Taking $\kappa$ to be our fixed inaccessible cardinal $\delta_0$, we then see that the hypotheses of \cref{thmx:DecompositionMainResult} are equivalent to those of \cite[Cor.~12.15]{Exp2}.
\end{rem}

\begin{rem}
In \cite{Exp4},  Nardin implicitly \emph{defines} a $\cT$-$\infty$-category to strongly admit $\cT$-finite $\cT$-(co)limits\footnote{Note that Nardin writes instead ``finite $\cT$-(co)limits'' for this notion and he also doesn't use the adjective ``strongly''; see \cref{rem:annoyingTerminology}.} if conditions (1) through (3) in \cref{thmx:DecompositionMainResult} are satisfied. Moreover, using his formulation, a $\cT$-stable $\cT$-$\infty$-category \cite[Def.~7.1]{Exp4} by definition strongly admits all $\cT$-finite $\cT$-colimits and $\cT$-finite $\cT$-limits. One practical consequence of \cref{thmx:DecompositionMainResult} is that $\cT$-stable $\cT$-$\infty$-categories then strongly admit $\cT$-(co)limits indexed by an \emph{a priori} larger class of $\cT$-diagrams; for instance, when $\cT = \OO_G$ this includes those $G$-spaces that admit the structure of a finite $G$-CW complex.
\end{rem}

Moving onto the theory of $\cT$-$\kappa$-filtered and $\cT$-sifted $\cT$-$\infty$-categories, we may make the following definitions as the evident parametrized generalizations of \cite[Def.~5.3.1.7]{HTT} and \cite[Def.~5.5.8.1]{HTT}.

\begin{dfn}[\cref{def:ParamFiltered}]
Let $\cJ$ be a $\cT$-$\infty$-category and let $\kappa$ be a regular cardinal. We say that $\cJ$ is \emph{$\cT$-$\kappa$-filtered} if for all $t \in \cT$ and $\cT^{/t}$-$\kappa$-small $\cK$, every $\cT^{/t}$-functor $p: \cK \to \cJ_{\underline{t}}$ admits an extension to a $\cT^{/t}$-functor $\overline{p}: \cK^{\underline{\rhd}} \to \cJ_{\underline{t}}$.\footnote{We recall the parametrized cone as \cref{def-app:JoinAndSlice}. Here, for the $\cT^{/t}$-$\infty$-category $\cK$, $\cK^{\underline{\rhd}}$ is notation for $\cK \star_{(\cT^{/t})^{\op}} (\cT^{/t})^{\op}$.}
\end{dfn}

\begin{ntn}
For a finite $\cT$-set $U$ with orbit decomposition $U_1 \sqcup ... \sqcup U_n$, we write
$$\underline{U} \coloneq \coprod_{i=1}^n (\cT^{\op})^{U_i/}$$
for the $\cT$-$\infty$-category given by the coproduct of corepresentable left fibrations; this straightens to the presheaf $\Map_{\FF_{\cT}}(-,U)|_{\cT^{\op}}$.
\end{ntn}

Recall from \cite{Exp2} that we write $\underline{\Fun}_{\cT}(-,-)$ for the internal hom for $\cT$-$\infty$-categories (defined at the level of marked simplicial sets as \cite[Def.~3.2]{Exp2}); for every $t \in \cT$ we have that $\underline{\Fun}_{\cT}(\cC,\cD)_t \simeq \Fun_{\cT^{/t}}(\cC_{\underline{t}}, \cD_{\underline{t}})$, and for every $\alpha: s \to t$, the restriction functor $\alpha^*$ is given by restricting $\cT^{/t}$-functors to $\cT^{/s}$-functors.

\begin{dfn}[\cref{def:ParamSifted}]
Let $\cJ$ be a $\cT$-$\infty$-category. Then $\cJ$ is \emph{$\cT$-sifted} if for all $t \in \cT$ and finite $\cT^{/t}$-sets $U$, the diagonal $\cT^{/t}$-functor $\delta: \cJ_{\underline{t}} \to \underline{\Fun}_{\cT^{/t}}(\underline{U}, \cJ_{\underline{t}})$ is $\cT^{/t}$-cofinal in the sense of \cite[Def.~6.8]{Exp2}, i.e. $\delta$ is fiberwise cofinal.
\end{dfn}

Our main theorems about these concepts should be read as confirming the following expectation: $\cT$-filtered and $\cT$-sifted $\cT$-colimits are computed as ordinary filtered and ordinary sifted colimits in the fibers. To say this precisely, we need another definition.

\begin{dfn}[\cref{def:cofinalconstant}]
Let $\cJ$ be a $\cT$-$\infty$-category. We say that $\cJ$ is \emph{cofinal-constant} (cc) if for all morphisms $\alpha: s \to t$ in $\cT$, the restriction functor $\alpha^*: \cJ_t \to \cJ_s$ is cofinal.
\end{dfn}

\begin{rem}
Let $\cJ$ be a cofinal-constant $\cT$-$\infty$-category and $p: \cJ \to \cC$ a $\cT$-functor. Moreover, suppose that $\cT$ has a terminal object $t$. Then by our hypothesis on $\cJ$ and \cite[Thm.~6.7]{Exp2}, the $\cT$-functor $\chi: \cJ_t \times \cT^{\op} \to \cJ$ uniquely determined by the inclusion $\cJ_t \subset \cJ$ is $\cT$-cofinal. Consequently, we obtain an equivalence
$$\colim^{\cT}_{\cJ} p \simeq \colim^{\cT}_{\cJ_t \times \cT^{\op}} p \circ \chi$$
provided that either $\cT$-colimit exists.
\end{rem}

Now let $\Spc$ denote the (large) $\infty$-category of small spaces and let $\underline{\Spc}_{\cT}$ be the $\cT$-$\infty$-category of small $\cT$-spaces \cite[Exm.~3.12]{Exp2}.

\begin{thmx}[{\cref{thm:EquivalentConditionsFiltered} and \cref{prop:FilteredCofinalityCriterion}}] \label{thmx:filtered} Suppose that $\cT$ is orbital. Let $\cJ$ be a $\cT$-$\infty$-category and let $\kappa$ be a regular cardinal. The following conditions are equivalent:
\begin{enumerate}
\item $\cJ$ is $\cT$-$\kappa$-filtered.
\item For all $t \in \cT$, $\cJ_t$ is $\kappa$-filtered, and $\cJ$ is cofinal-constant.
\item The $\cT$-colimit $\cT$-functor
\[ \underline{\colim}^{\cT}_{\cJ}: \underline{\Fun}_{\cT}(\cJ, \underline{\Spc}_{\cT}) \to \underline{\Spc}_{\cT} \]
strongly preserves $\cT$-$\kappa$-small $\cT$-limits.
\item For all $t \in \cT$ and $\cT^{/t}$-$\kappa$-small $\cK$, the diagonal $\cT^{/t}$-functor
\[ \delta: \cJ_{\underline{t}} \to \underline{\Fun}_{\cT^{/t}}(\cK, \cJ_{\underline{t}}) \]
is $\cT^{/t}$-cofinal.
\end{enumerate}
\end{thmx}

\begin{thmx}[{\cref{thm:EquivalentConditionsSifted}}] \label{thmx:sifted} Suppose that $\cT$ is orbital and let $\cJ$ be a $\cT$-$\infty$-category. The following conditions are equivalent:
\begin{enumerate}
\item $\cJ$ is $\cT$-sifted.
\item For all $t \in \cT$, $\cJ_t$ is sifted, and $\cJ$ is cofinal-constant.
\item The $\cT$-colimit $\cT$-functor
\[ \underline{\colim}^{\cT}_{\cJ}: \underline{\Fun}_{\cT}(\cJ, \underline{\Spc}_{\cT}) \to \underline{\Spc}_{\cT} \]
preserves finite $\cT$-products.
\end{enumerate}
\end{thmx}

Thirdly, building upon our earlier discussion of $\cT$-presheaves, we introduce the universal constructions $\underline{\Ind}^{\kappa}_{\cT}(\cC)$ and $\underline{\PShv}^{\Sigma}_\cT(\cC)$ that freely adjoin $\cT$-$\kappa$-filtered $\cT$-colimits and $\cT$-sifted $\cT$-colimits to $\cC$, respectively (\cref{dfn:Ind}). These are essentially defined to be the minimal full $\cT$-subcategories of $\underline{\PShv}_{\cT}(\cC) \coloneq \underline{\Fun}_{\cT}(\cC^{\vop}, \underline{\Spc}_{\cT})$ closed under the relevant $\cT$-colimits. However, in view of condition (2) in \cref{thmx:filtered} and \cref{thmx:sifted}, it turns out that $\underline{\Ind}^{\kappa}_{\cT}(-)$ and $\underline{\PShv}^{\Sigma}_\cT(-)$ are obtained by fiberwise application of $\Ind^{\kappa}$ and $\PShv^{\Sigma}$ (cf. \cref{var:FiberwisePresheaves}).

Our main result identifies these constructions in terms of $\cT$-presheaves that strongly preserve certain $\cT$-limits if $\cC$ admits sufficiently many $\cT$-colimits.

\begin{notation}
Let $\cD$ and $\cE$ be $\cT$-$\infty$-categories, and suppose in the following that $\cD, \cE$ strongly admit the relevant $\cT$-(co)limits. We introduce notation for certain full $\cT$-subcategories of $\underline{\Fun}_{\cT}(\cD,\cE)$, which may be specified by indicating over each $t \in \cT$ what $\cT^{/t}$-functors $\cD_{\underline{t}} \to \cE_{\underline{t}}$ span the fiber:
\begin{enumerate}
\item $\underline{\Fun}_{\cT}^{L}(\cD, \cE)$: take those $\cT^{/t}$-functors that strongly preserve all (small) $\cT^{/t}$-colimits.
\item $\underline{\Fun}^{\times}_{\cT}(\cD, \cE)$: take those $\cT^{/t}$-functors that preserve finite $\cT^{/t}$-products.\footnote{Note that there is no distinction between strongly preserving and preserving finite $\cT^{/t}$-products, and likewise for finite $\cT^{/t}$-coproducts.}
\item $\underline{\Fun}^{\sqcup}_{\cT}(\cD, \cE)$: take those $\cT^{/t}$-functors that preserve finite $\cT^{/t}$-coproducts.
\item $\underline{\Fun}_{\cT}^{\operatorname{\kappa-lex}}(\cD, \cE)$: take those $\cT^{/t}$-functors that strongly preserve $\cT$-$\kappa$-small $\cT^{/t}$-limits.
\item $\underline{\Fun}_{\cT}^{\operatorname{\kappa-rex}}(\cD, \cE)$: take those $\cT^{/t}$-functors that strongly preserve $\cT$-$\kappa$-small $\cT^{/t}$-colimits.
\end{enumerate}
\end{notation}

We only state the most important points here and refer the reader to the main body of the paper for the more comprehensive theorem.

\begin{thmx}[{\cref{thm:SiftedCocompletionAsParamProductPreservingPresheaves}}]
Suppose that $\cT$ is orbital and let $\cC$ be a $\cT$-$\infty$-category.
\begin{enumerate}
\item Suppose that $\cC$ admits finite $\cT$-coproducts. We then have an equality
$$\underline{\PShv}^{\Sigma}_{\cT}(\cC) = \underline{\Fun}^{\times}_{\cT}(\cC^{\vop}, \underline{\Spc}_{\cT}).$$
Moreover, $\underline{\PShv}^{\Sigma}_{\cT}(\cC)$ is $\cT$-cocomplete, and given any $\cT$-cocomplete $\cT$-$\infty$-category $\cD$, restriction along the $\cT$-Yoneda embedding $j^{\Sigma}_{\cT}: \into{\cC}{\underline{\PShv}^{\Sigma}_{\cT}(\cC)}$ implements an equivalence
\[ \underline{\Fun}^L_{\cT}(\underline{\PShv}^{\Sigma}_{\cT}(\cC), \cD) \xto{\simeq} \underline{\Fun}^{\sqcup}_{\cT}(\cC, \cD) \]
with inverse given by $\cT$-left Kan extension.
\item Suppose that $\cC$ strongly admits $\cT$-$\kappa$-small $\cT$-colimits. We then have an equality
$$ \underline{\Ind}^{\kappa}_{\cT}(\cC) = \underline{\Fun}^{\operatorname{\kappa-lex}}_{\cT}(\cC^{\vop}, \underline{\Spc}_{\cT}).$$
Moreover, $\underline{\Ind}^{\kappa}_{\cT}(\cC)$ is $\cT$-cocomplete, and given any $\cT$-cocomplete $\cT$-$\infty$-category $\cD$, restriction along the $\cT$-Yoneda embedding $j^{\kappa}_{\cT}: \into{\cC}{\underline{\Ind}^{\kappa}_{\cT}(\cC)}$ implements an equivalence
\[ \underline{\Fun}^L_{\cT}(\underline{\Ind}^{\kappa}_{\cT}(\cC), \cD) \xto{\simeq} \underline{\Fun}^{\operatorname{\kappa-rex}}_{\cT}(\cC, \cD) \]
with inverse given by $\cT$-left Kan extension.
\end{enumerate}
\end{thmx}

Lastly, we also lay parametrized foundations for two other important concepts in the theory of $\infty$-categories: \emph{factorization systems} and \emph{exponentiable} (i.e., \emph{flat}) \emph{fibrations}.\footnote{Some authors (e.g., Ayala and Francis \cite{ayala2017ibrations}) reserve the term \emph{exponentiable} for the homotopy invariant definition, but we will elide this distinction in our narrative here.} We defer the statements of these results to their respective sections \ref{sec:fact} and \ref{sec:pairing}. Most notably, we use the theory of $\cT$-factorization systems to establish the universal property of the free $\cT$-cocartesian fibration (\cref{exm:freeCocartesianFibration}), while we use the theory of $\cT$-flat fibrations and the associated \emph{$\cT$-pairing construction} (\cref{conthm:PairingConstructionRecalled}) to study $\cT$-(co)limits in a $\cT$-$\infty$-category of sections (\cref{thm:ParametrizedLimitsAndColimitsInSectionCategories}).


\begin{rem}
In the case $\cT = \ast$, our main \cref{thm:GeneralFreeFibration} on parametrized factorization systems applies to give a common generalization of the proof of the universal property of the usual free cocartesian fibration (\cite[Thm.~4.5]{GHN}) with that of the universal property of the $\cO$-monoidal envelope for an $\infty$-operad $\cO$ (\cite[Prop.~2.2.4.9]{HA}).\footnote{This line of reasoning is well-known to experts and has also appeared in the literature as \cite[Prop.~B.1]{ayala2017actorization}; we thank Rune Haugseng for the pointer.} In \cite{paramalg}, we will apply \cref{thm:GeneralFreeFibration} to establish the theory of $\cO$-monoidal envelopes for a $\cT$-$\infty$-operad $\cO$.
\end{rem}

\begin{rem}
Our main interest in \cref{thm:ParametrizedLimitsAndColimitsInSectionCategories} lies in using it in \cite{paramalg} to study $\cT$-(co)limits in a $\cT$-$\infty$-category of $\cO$-algebras for a $\cT$-$\infty$-operad $\cO$. Also see \cite[Prop.~7.6]{BACHMANN2021} for a similar type of statement in the context of normed $E_{\infty}$-algebras in motivic homotopy theory.
\end{rem}

In the appendix, we take the opportunity to give the correct\footnote{If $\cO^{\otimes}$ is the commutative $\infty$-operad, then it turns out that our earlier definition of symmetric promonoidal given in \cite{BarwickGlasmanShah} was insufficiently general; see \cref{exm:promonoidalFlat}. We thank Yonatan Harpaz for alerting us to this issue.} definition of an exponentiable fibration of $\infty$-operads and then the construction of \emph{$\cO$-promonoidal Day convolution} with respect to a base $\infty$-operad $\cO^{\otimes}$ (\cref{conthm:DayConvolution}). This generalizes Lurie's construction in \cite[\S 2.2.6]{HA}, which supposes that the source $\infty$-operad in question is $\cO$-monoidal. We saw fit to include this material here because the main lemma behind it (\cref{lem:FlatCategoricalFibrationArrowLemma}) is also used to establish the theory of $\cT$-flat fibrations.

\begin{rem}
Vladimir Hinich has informed us that our treatment of $\cO$-promonoidal Day convolution is a slightly reorganized version of his discussion in \cite[\S 2.8]{HINICH2020107129}. In particular, our \cref{conthm:DayConvolution} is essentially his \cite[Prop.~2.8.3]{HINICH2020107129}, and \cref{lem:FlatCategoricalFibrationArrowLemma} when specialized to the context of $\infty$-operads is his \cite[Lem.~2.8.4]{HINICH2020107129}.
\end{rem}

\subsection*{Notation and terminology}

We collect a few miscellaneous pieces of notation and terminology from \cite{Exp2} that we have not introduced yet in our discussion.

\begin{convention}
Let $X, Y \to Z$ be maps of simplicial sets. Unless otherwise indicated, when we write $X \times_Z \Ar(Z) \times_Z Y$ we mean $X \times_{Z, \ev_0} \Ar(Z) \times_{\ev_1, Z} Y$ (i.e., evaluation at the source goes to the left and evaluation at the target goes to the right).
\end{convention}

We will need to use the theory of marked simplicial sets in various places in this paper; see \cite[\S 2]{Exp2} for a review.

\begin{ntn}
\begin{enumerate}[leftmargin=*]
\item Given a simplicial set $X$, we let $X^{\flat}$ be the minimal marking on $X$ and $X^{\sharp}$ the maximal marking on $X$. 
\item If $p: X \to S$ is a cocartesian fibration, then we let $\leftnat{X}$ denote $X$ with its $p$-cocartesian edges marked.
\end{enumerate}
\end{ntn}

\begin{notation}
We will generally write $\cT^{\op}$ as $\ast_{\cT}$ when we wish to think of it as the terminal $\cT$-$\infty$-category.
\end{notation}

\begin{dfn}
Let $\cC$ be a $\cT$-$\infty$-category. We define the \emph{$\cT$-$\infty$-category of arrows} in $\cC$ to be
$$\Ar_{\cT}(\cC) \coloneq \cT^{\op} \times_{\Ar(\cT^{\op})} \Ar(\cC)$$
where the map $\cT^{\op} \to \Ar(\cT^{\op})$ is the identity section.
\end{dfn}

\begin{rec}[{\cite[Def.~4.1]{Exp2}}]
Let $S$ be a simplicial set, let $\iota: \partial \Delta^1 \times S \subset \Delta^1 \times S$ be the inclusion functor, and consider the right adjoint
\[ \iota_*: \sSet_{/\partial \Delta^1 \times S} \to \sSet_{/\Delta^1 \times S} \]
to pullback along $\iota$. Then for maps $p,q: X, Y \to S$ of simplicial sets, we define the \emph{$S$-join} $X \star_S Y$ to be $\iota_*(X,Y)$; this recovers the ordinary join if $S = \ast$. Note that if we let $\chi: S \times \Delta^1 \to S \star S$ be the map adjoint to $\id_{S \times \partial \Delta^1}$ (using the universal property of the ordinary join), then we have a canonical isomorphism
\[ X \star_S Y \cong S \times \Delta^1 \underset{\chi, S \star S, p \star q}{\times} X \star Y, \]
so the $S$-join is the \emph{relative join} in the sense of Lurie \cite[\href{https://kerodon.net/tag/0241}{Tag 0241}]{kerodon}. In keeping with the terminology of \cite{Exp2}, however, we will prefer to generically call this the parametrized join.

Now if $S = \cT^{\op}$ and we have $\cT$-$\infty$-categories $\cC$ and $\cD$, the $\cT^{\op}$-join $\cC \star_{\cT^{\op}} \cD$ is again a $\cT$-$\infty$-category, and in fact the structure map to $\cT^{\op} \times \Delta^1$ is a $\cT$-functor (cf. \cite[Prop.~4.3]{Exp2}).
\end{rec}

\begin{rec}[{\cite[Def.~8.3]{Exp2}}]
Let $\cC$ and $\cD$ be $\cT$-$\infty$-categories and let $\adjunct{F}{\cC}{\cD}{G}$ be a relative adjunction with respect to $\cT^{\op}$ \cite[Def.~7.3.2.2]{HA}. Then we say that $F \dashv G$ is a \emph{$\cT$-adjunction} if $F$ and $G$ are both $\cT$-functors.
\end{rec}

\begin{rec}[{\cite[Def.~7.1]{Exp2}}]
Let $p: \cX \to \cB$ be a $\cT$-functor. We say that $p$ is a \emph{$\cT$-fibration} if $p$ is a categorical fibration. In this case, $p$ is \emph{$\cT$-cocartesian}, resp. \emph{$\cT$-cartesian} if
\begin{enumerate}
\item For every object $t \in \cT$, $p_t: \cX_t \to \cB_t$ is a cocartesian, resp. cartesian fibration.
\item For every morphism $\alpha: s \to t$, the restriction functor $\alpha^\ast: \cX_{t} \to \cX_{s}$ carries $p_t$-cocartesian, resp. $p_t$-cartesian edges to $p_s$-cocartesian, resp. $p_s$-cartesian edges.
\end{enumerate}
If $p: \cX \to \cB$ and $q: \cY \to \cB$ are two $\cT$-cocartesian fibrations, we say that a $\cT$-functor $F: \cX \to \cY$ over $\cB$ is a \emph{morphism of $\cT$-cocartesian fibrations} if $F$ preserves fiberwise (with respect to $\cT$) cocartesian edges. Similarly, we have the analogous definition of a morphism of $\cT$-cartesian fibrations.

Finally, note that $p$ is $\cT$-cocartesian if and only if $p$ is a cocartesian fibration \cite[Rem.~7.4]{Exp2}.
\end{rec}

\subsection*{Acknowledgements}

I would like to thank Denis Nardin for helpful conversations on the subject matter of this paper. I would also like to acknowledge that Dylan Wilson has obtained similar results in unpublished work. The author was funded by the Deutsche Forschungsgemeinschaft (DFG, German Research Foundation) under Germany’s Excellence Strategy EXC 2044–390685587, Mathematics Münster: Dynamics–Geometry–Structure.

\section{Recollections on parametrized limits and colimits}

\label{sec:old}

In this section, we give a streamlined exposition of the concepts of parametrized (co)limits and Kan extensions introduced in \cite{Exp2}. This is done primarily to fix notation and make this paper more self-contained. For the reader already familiar with \cite{Exp2}, the only points to bear in mind are our more concise notation for parametrized cones (\cref{def-app:JoinAndSlice}) and the notion of strongly admitting and preserving $\CMcal{K}$-indexed $\cT$-colimits with respect to certain collections $\CMcal{K}$ of parametrized diagrams (\cref{def-app:StronglyAdmitsAndPreservesColimit}).

\begin{notation}[{\cite[Notn.~3.5]{Exp2}}] Let $p: \cK \to \cC$ be a $\cT$-functor. We then let
\[ \sigma_p: \ast_{\cT} \to \underline{\Fun}_{\cT}(\cK, \cC) \]
denote the cocartesian section given by adjointing the map $\Ar(\cT^\op)^\sharp \times_{\cT^\op} \leftnat{\cK} \xto{\pr} \leftnat{\cK} \xto{p} \leftnat{\cC}$. This is an explicit choice of $\cT$-functor corresponding to $p$ under the equivalence
\[ \Fun_{\cT}(\ast_{\cT}, \underline{\Fun}_{\cT}(\cK, \cC)) \simeq \Fun_{\cT}(\cK,\cC) \]
of \cite[Prop.~3.4]{Exp2}.
\end{notation}



\begin{definition}[Cones and slices] \label{def-app:JoinAndSlice}
Let $\cC$ be a $\cT$-$\infty$-category. We let
$$\cC^{\underline{\rhd}} \coloneq \cC \star_{\cT^\op} \cT^\op \:, \qquad \cC^{\underline{\lhd}} \coloneq \cT^\op \star_{\cT^\op} \cC $$
denote the \emph{$\cT$-right} and \emph{$\cT$-left cones} on $\cC$. We also write $v: \ast_{\cT} \subset \cC^{\underline{\rhd}}$ or $\cC^{\underline{\lhd}}$ for the inclusion of the cone $\cT$-point.

For a $\cT$-functor $p: \cK \to \cC$, we then let
$$ \cC^{(p,\cT)/} \coloneq \ast_{\cT} \times_{\sigma_p, \underline{\Fun}_{\cT}(\cK, \cC)} \underline{\Fun}_{\cT}(\cK^{\underline{\rhd}}, \cC) \:, \qquad \cC^{/(p,\cT)} \coloneq \ast_{\cT} \times_{\sigma_p,\underline{\Fun}_{\cT}(\cK, \cC)} \underline{\Fun}_{\cT}(\cK^{\underline{\lhd}}, \cC) $$
denote the \emph{slice} $\cT$-$\infty$-categories.
\end{definition}


We will also need in a few places the following smaller model for slicing over and under a $\cT$-object.

\begin{definition} \label{def-app:SmallerSlice}
Let $\cC$ be a $\cT$-$\infty$-category. For any object $x \in \cC_t$, we write
\[ \cC^{/\underline{x}} \coloneq \Ar_{\cT}(\cC) \times_{\cC} \underline{x} \:, \qquad \cC^{\underline{x}/} \coloneq \underline{x} \times_{\cC} \Ar_{\cT}(\cC) \]
and regard these as $\cT^{/t}$-$\infty$-categories via composition of the projection to $\underline{x}$ with the trivial fibration $\underline{x} \xto{\simeq} (\cT^{/t})^\op$ (\cite[Lem.~12.10]{Exp2}).
\end{definition}

\begin{observation}[{\cite[Prop.~4.30]{Exp2}}] \label{obs-app:SmallerSliceEqv}
In \cref{def-app:SmallerSlice}, if we write $i_x: \underline{x} \to \cC_{\underline{t}}$ for the $\cT^{/t}$-functor defined by $x$, then we have canonical equivalences
\[ \cC^{/\underline{x}} \simeq (\cC_{\underline{t}})^{/(i_x, \cT^{/t})} \:, \qquad \cC^{\underline{x}/} \simeq (\cC_{\underline{t}})^{(i_x, \cT^{/t})/} \]
of $\cT^{/t}$-$\infty$-categories over $\cC_{\underline{t}}$. Similarly, for a cocartesian section $\sigma: \ast_{\cT} \to \cC$, we have canonical equivalences 
\[ \Ar_{\cT}(\cC) \times_{\cC, \sigma} \ast_{\cT} \simeq \cC^{/(\sigma, \cT)} \:, \qquad \ast_{\cT} \times_{\sigma,\cC} \Ar_{\cT}(\cC) \simeq \cC^{(\sigma, \cT)/} \]
of $\cT$-$\infty$-categories over $\cC$.
\end{observation}

We now proceed to our discussion on parametrized colimits; the case of parametrized limits is dual in view of \cite[Cor.~5.25]{Exp2} and hence will not be explicitly considered.

\begin{definition}[{\cite[Def.~5.1-2]{Exp2}}] \label{def-app:ParamColimit}
Let $\cC$ be a $\cT$-$\infty$-category. A $\cT$-functor $\sigma: \ast_{\cT} \to \cC$ is a \emph{$\cT$-initial object} if and only if $\sigma(t) \in \cC_t$ is an initial object for all $t \in \cT$. A $\cT$-functor $\overline{p}: \cK^{\underline{\rhd}} \to \cC$ is then a \emph{$\cT$-colimit diagram} if and only if the $\cT$-functor
\[ (\id,\sigma_{\overline{p}}) : \ast_{\cT} \to \cC^{(p,\cT)/} = \ast_{\cT} \times_{\sigma_p, \underline{\Fun}_{\cT}(\cK, \cC)} \underline{\Fun}_{\cT}(\cK^{\underline{\rhd}}, \cC) \]
is a $\cT$-initial object. Lastly, we say that a $\cT$-functor $p: \cK \to \cC$ \emph{admits a $\cT$-colimit} if $p$ admits an extension to a $\cT$-colimit diagram $\overline{p}$, and we then write $\colim^\cT_{\cK} p = \overline{p}|_{v}$. If $\cT$ moreover has a terminal object $t$, we will also identify the cocartesian section $\colim^\cT_{\cK} p$ with its value at $t$.\footnote{If $t \in \cT^{\op}$ is an initial object, then cocartesian sections are uniquely specified by their value at $t$.}
\end{definition}

\begin{notation}
Let $p: \cK \to \cC$ be a $\cT$-functor and let $\delta: \cC \to \underline{\Fun}_{\cT}(\cK, \cC)$ be the constant $\cT$-functor. We then write
\[ \begin{tikzcd}
\underline{\colim}^{\cT}_{\cK}: \underline{\Fun}_{\cT}(\cK, \cC) \ar[dotted]{r} & \cC
\end{tikzcd} \]
for the partially-defined\footnote{For a $\cT$-functor $R: \cC \to \cD$, the domain of its partial $\cT$-left adjoint $L$ is the largest full $\cT$-subcategory $\cD_0 \subset \cD$ for which $L_t$ is a partial left adjoint to $R_t$ for all $t \in \cT$ and $f^\ast L_t \xto{\simeq} L_s f^\ast$ for all $(f: s \to t) \in \cT$.} $\cT$-left adjoint of $\delta$.
\end{notation}

The next observation is the trivial case of \cite[Cor.~9.16]{Exp2} where we let $\cD = \cT^\op$ there.

\begin{observation}
For a $\cT$-functor $p: \cK \to \cC$, $\underline{\colim}^{\cT}_{\cK}$ is defined on an object $p: \cK_{\underline{t}} \to \cC_{\underline{t}}$ in the fiber $\underline{\Fun}_{\cT}(\cK, \cC)_t \simeq \Fun_{\cT^{/t}}(\cK_{\underline{t}}, \cC_{\underline{t}})$ if and only if $p$ admits a $\cT^{/t}$-colimit, in which case $\underline{\colim}^{\cT}_{\cK} p \simeq (\colim^{\cT^{/t}}_{\cK_{\underline{t}}} p)(t)$. In particular, if for each $t \in \cT$ the parametrized fiber $\cC_{\underline{t}}$ admits all $\cK_{\underline{t}}$-indexed $\cT^{/t}$-colimits, then $\underline{\colim}^{\cT}_{\cK}$ is defined on its entire domain.

Passing to cocartesian sections, we then see that
\[ \begin{tikzcd}
\colim^{\cT}_{\cK}: \Fun_{\cT}(\cK, \cC) \ar[dotted]{r} & \Fun_{\cT}(\ast_{\cT},\cC)
\end{tikzcd} \]
is a partial left adjoint to the functor given by precomposing with the structure map of $\cK$, and $\colim^{\cT}_{\cK}$ is defined on its entire domain if $\underline{\colim}^{\cT}_{\cK}$ is (but possibly not conversely).
\end{observation}

This observation already highlights the need to systematically distinguish between $\cT$-colimits in $\cC$ and $\cT^{/t}$-colimits in the parametrized fibers $\cC_{\underline{t}}$. We do this as follows:

\begin{definition} \label{def-app:StronglyAdmitsAndPreservesColimit}
Let $\cC$ be a $\cT$-$\infty$-category.
\begin{enumerate}
\item $\cC$ \emph{strongly admits} all $\cT$-colimits, i.e., is \emph{$\cT$-cocomplete} \cite[Def.~5.13]{Exp2}, if for each $t \in \cT$, $\cC_{\underline{t}}$ admits all $\cT^{/t}$-colimits.
\item If $\cC$ and $\cD$ are $\cT$-cocomplete $\cT$-$\infty$-categories, then a $\cT$-functor $F:\cC \to \cD$ \emph{strongly preserves} all $\cT$-colimits if for each $t \in \cT$, $F_{\underline{t}}: \cC_{\underline{t}} \to \cD_{\underline{t}}$ preserves all $\cT^{/t}$-colimits \cite[Def.~11.2]{Exp2}.
\end{enumerate} 
More generally, if we have a collection $\CMcal{K} = \{ \CMcal{K}_{t} : t \in \cT \}$ where $\CMcal{K}_t$ is a class of small $\cT^{/t}$-$\infty$-categories such that for each morphism $f: s \to t$ in $\cT$, $f^\ast(\CMcal{K}_t) \subset \CMcal{K}_s$, then we have analogous notions of strongly admitting and preserving $\CMcal{K}$-indexed $\cT$-colimits. We will typically leave the collection $\{ \CMcal{K}_t \}$ implicit when referring to $\CMcal{K}$. Abusing notation, we will also let $\CMcal{K}$ refer to the class of $\cT$-$\infty$-categories $\cK$ such that $\cK_{\underline{t}} \in \CMcal{K}_t$ for all $t \in \cT$.
\end{definition}

\begin{remark}
In this paper, all $\cT$-colimits will be indexed by small $\cT$-$\infty$-categories, and we will typically suppress the adjective `small' in this context (as was already done in \cref{def-app:StronglyAdmitsAndPreservesColimit}).
\end{remark}

We next review the theory of $\cT$-left Kan extensions along fully faithful $\cT$-functors. We first need an auxiliary construction. 

\begin{rem} \label{rem:ConeMappingOut}
Let $\cC$ be a $\cT$-$\infty$-category. By definition, the $\cT$-right cone $\cC^{\underline{\rhd}}$ has a universal mapping property with respect to maps going \emph{in}. In the following we will also need a universal mapping property of $\cC^{\underline{\rhd}}$ for maps going \emph{out}. Namely, by \cite[Lem.~4.5]{Exp2} we have a homotopy pushout square of $\cT$-$\infty$-categories
\[ \begin{tikzcd}
\cC \times \{ 1\} \ar[hookrightarrow]{r} \ar{d}{p} & \cC \times \Delta^1 \ar{d}{f} \\ 
\ast_{\cT} \ar[hookrightarrow]{r}{v} & \cC^{\underline{\rhd}}
\end{tikzcd} \]
where $f$ is defined as the adjoint to $(\id_{\cC}, p)$.

Now suppose that $\sigma: \ast_{\cT} \to \cC$ is a $\cT$-final object. Then we may construct a homotopy $h: \cC \times \Delta^1 \to \cC$ from $\id_{\cC}$ to $\sigma$, which yields a $\cT$-functor
\[ h': \cC^{\underline{\rhd}} \to \cC \]
such that $h'|_{\cC} = \id_{\cC}$ and $h'|_{v} = \sigma$. Moreover, if one considers the bifibration (cf. \cite[Lem.~4.8]{Exp2})
\[ (f , g): \Fun_{\cT}(\cC^{\underline{\rhd}}, \cC) \to \Fun_{\cT}(\cC, \cC) \times \Fun_{\cT}(\ast_{\cT}, \cC) \]
then $h'$ is obtained by taking a $f$-cartesian lift with target $[\cC^{\underline{\rhd}} \xto{\overline{p}} \ast_{\cT} \xto{\sigma} \cC]$ in $\Fun_{\cT}(\cC^{\underline{\rhd}}, \cC)$ over the edge $\id_{\cC} \to \sigma \circ p$ in $\Fun_{\cT}(\cC, \cC)$ specified by $h$.
\end{rem}

\begin{construction} \label{con-app:LKE}
Let $\cD$ be a $\cT$-$\infty$-category and let $x \in \cD_t$. We then construct a $\cT^{/t}$-functor
\[ \theta_x: (\cD^{/\underline{x}})^{\underline{\rhd}} \to \cD_{\underline{t}} \]
as follows (where the parametrized right cone is formed with respect to the base $\cT^{/t}$). First, we adjoint the projection $\cD^{/\underline{x}} \to \Ar_{\cT}(\cD)$ to obtain a $\cT^{/t}$-functor $h_x: \cD^{/\underline{x}} \times \Delta^1 \to \cD_{\underline{t}}$. We then let $\theta_x$ be the composite of $h_x$ and the $\cT^{/t}$-functor
$$(h' , \pi) : (\cD^{/\underline{x}})^{\underline{\rhd}} \to \cD^{/\underline{x}} \times \Delta^1,$$
where $\pi$ is the structure map to $\Delta^1$ of the parametrized join and $h'$ is as in \cref{rem:ConeMappingOut} (note that any choice of cocartesian section $j_x: \ast_{\cT^{/t}} \to \cD^{/\underline{x}}$ determined by $\id_x$ is a $\cT^{/t}$-final object).

Now suppose given a $\cT$-functor $G: \cD \to \cE$ and a full $\cT$-subcategory $\cC \subset \cD$. We let $F \coloneq G|_{\cC}$, and for any $x \in \cD_t$ we let $\cC^{/\underline{x}} \coloneq \cC \times_{\cD} \cD^{/\underline{x}}$. We then write
\[ G^x: (\cC^{/\underline{x}})^{\underline{\rhd}} \to (\cD^{/\underline{x}})^{\underline{\rhd}} \xto{\theta_x} \cD_{\underline{t}} \xto{G_{\underline{t}}} \cE_{\underline{t}} \]
for the composite $\cT^{/t}$-functor. Note that $G^x|_{\cC^{/\underline{x}}}$ factors as $\cC^{/\underline{x}} \to \cC_{\underline{t}} \xto{F_{\underline{t}}} \cE_{\underline{t}}$; we write $F^x$ for this $\cT^{/t}$-functor.
\end{construction}

The following is a simplification of \cite[Def.~10.1]{Exp2} in which we have chosen the datum of the natural transformation $\eta$ present there to be the identity, which allows us to dispense with the auxiliary construction of $G'': (\cC \times_{\cD} \Ar_{\cT}(\cD)) \star_{\cD} \cD \to \cE$ in that definition.

\begin{definition}
Let $\cD$ be a $\cT$-$\infty$-category and $\cC \subset \cD$ a full $\cT$-subcategory. We say that a $\cT$-functor $G: \cD \to \cE$ is a \emph{$\cT$-left Kan extension} of its restriction $F = G|_{\cC}$ if for all $x \in \cD_t$, the $\cT^{/t}$-functor $G^x$ of \cref{con-app:LKE} is a $\cT^{/t}$-colimit diagram.

A $\cT$-functor $F: \cC \to \cE$ then \emph{admits a $\cT$-left Kan extension to $\cD$} if there exists such a $G$, and we say that $\{ F^x \}_{x \in \cD}$ constitutes the set of \emph{relevant diagrams} for the extension problem.
\end{definition}

The author proved the following existence and uniqueness theorem for $\cT$-left Kan extensions as \cite[Thm.~10.3]{Exp2}, \cite[Thm.~10.5]{Exp2}, and \cite[Prop.~10.6]{Exp2}.

\begin{theorem} \label{thm-app:LKEexistence}
Let $\cD$ be a $\cT$-$\infty$-category and $\cC \subset \cD$ a full $\cT$-subcategory (with inclusion $\cT$-functor $\phi$).
\begin{enumerate}
\item A $\cT$-functor $F : \cC \to \cE$ admits a $\cT$-left Kan extension $G$ over $\cD$ if and only if all the relevant diagrams for $F$ admit parametrized colimits. Moreover, $G$ is then uniquely specified up to contractible choice.
\item The partial $\cT$-left adjoint $\phi_!$ to the restriction $\cT$-functor $\phi^\ast: \underline{\Fun}_{\cT}(\cD,\cE) \to \underline{\Fun}_{\cT}(\cC, \cE)$ is defined on all $F: \cC_{\underline{t}} \to \cE_{\underline{t}}$ that admit a $\cT^{/t}$-left Kan extension $G$, in which case $\phi_! F \simeq G$. In particular, $\phi_!$ is defined on its entire domain if for every $x \in \cD_{t}$, the parametrized fiber $\cE_{\underline{t}}$ admits all $\cC^{/\underline{x}}$-indexed $\cT^{/t}$-colimits.
\end{enumerate}
\end{theorem}

In fact, constructing $\cT$-left Kan extensions along fully faithful $\cT$-functors suffices to handle the general case:

\begin{remark} \label{rem-app:RelatingLKEnotions}
Let $\phi: \cC \to \cD$ be a $\cT$-functor and let $\pi: \cM \to \Delta^1 \times \cT^\op$ be a cocartesian fibration classified by $\phi$ (so $\cM$ is a $\cT$-$\infty$-category and $\cC \simeq \cM_0 \subset \cM$ is the inclusion of a full $\cT$-subcategory). Suppose that we have a $\cT$-functor $\overline{G}: \cM \to \cE$ that is a $\cT$-left Kan extension of its restriction $F = \overline{G}|_{\cC}$. Let $G = \overline{G}|_{\cD}$. We then may construct a natural transformation $\eta: F \Rightarrow \phi^\ast G$ such that $\eta$ exhibits $G$ as the $\cT$-left Kan extension of $F$ along $\phi$ in the sense of \cite[Def.~10.1]{Exp2}. Indeed, consider the trivial fibration\footnote{Cf. \cite[Lem.~2.23]{Exp2} applied to $\pi$ and restrict arrows to be $\cT$-fiberwise.}
$$p = (\ev_0, \pi): \Ar^{\cocart}_{\cT}(\cM) \to \cM \times_{\Delta^1} \Ar(\Delta^1)$$
and let $\sigma$ be a section that restricts to the identity on $\cM$, given by a choice of dotted lift in the diagram
\[ \begin{tikzcd}
\cM \ar[hookrightarrow]{r}{\iota} \ar[hookrightarrow]{d}{\iota} & \Ar^{\cocart}_{\cT}(\cM) \ar[twoheadrightarrow]{d}{p}[swap]{\simeq} \\
\cM \times_{\Delta^1} \Ar(\Delta^1) \ar{r}{=} \ar[dotted]{ru}{\sigma} & \cM \times_{\Delta^1} \Ar(\Delta^1)
\end{tikzcd} \]
where $\iota$ generically denotes the identity section. Then let
$$\eta: \cC \times \Delta^1 = \cC \times \{[0=0] \ra [0 \ra 1] \} \subset \cM \times_{\Delta^1} \Ar(\Delta^1) \xto{\sigma} \Ar^{\cocart}_{\cT}(\cM) \xto{\ev_1} \cM \xto{\overline{G}} \cE $$
and note that $\eta|_{\cC \times \{ 0\}} = F$ and $\eta|_{\cC \times \{ 1\}} = G \circ \phi$. The assertion that $G \simeq \phi_! F$ then follows by examining the pointwise formula defining a $\cT$-left Kan extension.
\end{remark}

\section{Parametrized factorization systems}

\label{sec:fact}


Our goal in this section is to prove a theorem about parametrized factorization systems (\cref{thm:GeneralFreeFibration}) that will allow us to prove the universal property of the free $\cT$-cocartesian fibration (\cref{exm:freeCocartesianFibration}) and subsequently that of the $\cO$-monoidal envelope for a $\cT$-$\infty$-operad $\cO$ in \cite{paramalg}.

\begin{dfn} \label{def:parametrizedFactSystem}
Let $\cC$ be a $\cT$-$\infty$-category. Then a \emph{$\cT$-factorization system} on $\cC$ is the data of a factorization system ($\sL_t$, $\sR_t$) on the fiber $\cC_t$ for every $t \in \cT$, subject to the condition that for every morphism $\alpha: s \to t$ in $\cT$, the restriction functor $\alpha^*: \cC_t \to \cC_s$ sends $(\sL_t, \sR_t)$ into ($\sL_s, \sR_s$).
\end{dfn}

\begin{rem} In \cref{def:parametrizedFactSystem}, we could instead formulate the condition of compatibility of the fiberwise factorization systems with restriction in the following way. Let $p$ denote the structure map of $\cC$ and consider the collection of commutative squares in $\cC$
\[ \begin{tikzcd}[row sep=2em, column sep=2em]
x \ar{r}{\alpha_x} \ar{d}[swap]{f} & x' \ar{d}{f'} \\
y \ar{r}{\alpha_y} & y'
\end{tikzcd} \]
such that $f$ resp. $f'$ lies in the fiber $\cC_t$ resp. $\cC_{t'}$, $p(\alpha_x) = p(\alpha_y)$, and $\alpha_x, \alpha_y$ are $p$-cocartesian edges. Then we must have that if $f$ is in $\sL_t$ resp. $\sR_t$, then $f'$ is in $\sL_{t'}$ resp. $\sR_{t'}$.
\end{rem}

\begin{definition} \label{def:TotalFactorizationSystem}
Let $\cC$ be a $\cT$-$\infty$-category with structure map $p$. Given a $\cT$-factorization system $(\sL_t,\sR_t)_{t \in \cT}$ on $\cC$, let $\sL$ be the collection of edges $e: x \rightarrow y$ in $\cC$ such that for any factorization $x \xrightarrow{e'} x' \xrightarrow{f} y$ of $e$ by a $p$-cocartesian edge $e'$ and a fiberwise edge $f$, $f$ is in $\sL_{p(y)}$. Let $\sR$ be the closure of the union of the $\sR_t$ under equivalences in $\cC$.
\end{definition}

We have the following variant of \cite[Prop.~2.1.2.5]{HA}.

\begin{lem} \label{lem:TotalFactSystem} $(\sL,\sR)$ is a factorization system on $\cC$.
\end{lem}
\begin{proof} We check the three conditions of a factorization system in turn.
\begin{enumerate}
\item Using the stability of the classes $\{\sL_t \}$ and $\{\sR_t\}$ along with the $p$-cocartesian edges under retracts, we see that $\sL$ and $\sR$ are closed under retracts.
\item Given an edge $e: x \rightarrow y$ in $\cC$, factor $e$ as $x \xrightarrow{e'} x' \xrightarrow{f} y$ for $e'$ $p$-cocartesian and $f$ in the fiber $\cC_{p(y)}$. Using the factorization system $(\sL_{p(y)},\sR_{p(y)})$ on $\cC_{p(y)}$, factor $f$ as $x' \xrightarrow{f'} x'' \xrightarrow{f''} y$ where $f' \in \sL_{p(y)}$ and $f'' \in \sR_{p(y)}$. Then $x \xrightarrow{f' \circ e'} x'' \xrightarrow{f''} y$ is our desired factorization of $e$.
\item Suppose we have a commutative square
\[ \begin{tikzcd}[row sep=2em, column sep=2em]
w \ar{r} \ar{d}{f} & y \ar{d}{g} \\
x \ar{r} & z
\end{tikzcd} \]
with $f \in \sL$ and $g \in \sR$; we want to produce an essentially unique filler $x \to y$. Without loss of generality, we may suppose $p(y) = p(z)$ and $g \in \sR_{p(y)}$. Choosing $p$-cocartesian edges we may factor the square as
\[ \begin{tikzcd}[row sep=2em, column sep=2em]
w \ar{r} \ar{rd}[swap]{f} & w' \ar{r} \ar{d}{f'} & w'' \ar{r} \ar{d}{f''} & y \ar{d}{g} \\
& x \ar{r}{\alpha} & x' \ar{r} \ar[dotted]{ru}[swap]{h} & z
\end{tikzcd} \]
where the edges which `add a prime' are $p$-cocartesian, and vertical edges along with the rightmost square lie in a fiber. By definition, $f' \in \sL_{p(x)}$, and since the $(\sL_t,\sR_t)_{t \in \cT}$ constitute a $\cT$-factorization system on $\cC$, $f'' \in \sL_{p(y)}$. Then we have an essentially unique filler $h$, and $h \circ \alpha: x \rightarrow y$ is our desired filler.
\end{enumerate}
\end{proof}

\begin{proposition} \label{prop:InertCartesianFibration} Suppose $\cD$ is a $\cT$-$\infty$-category and $(\sL_t, \sR_t)_{t \in \cT}$ is a $\cT$-factorization system on $\cD$. Let $(\sL,\sR)$ be the induced factorization system on $\cD$ of \cref{def:TotalFactorizationSystem}.

\begin{enumerate}
\item Let $\Ar^{L}_{\cT}(\cD)$ resp. $\Ar^L(\cD)$ denote the full subcategory of $\Ar_{\cT}(\cD)$ resp. $\Ar(\cD)$ on the morphisms in $\sL$. Then the source map
\[ \ev_0: \Ar^{L}_{\cT}(\cD) \to \cD \]
is a $\cT$-cartesian fibration of $\cT$-$\infty$-categories, and the source map
\[ \ev_0: \Ar^L(\cD) \to \cD \]
is a cartesian fibration (where here the domain is not generally a $\cT$-$\infty$-category).

\item Let $\Ar^R_{\cT}(\cD)$ denote the full subcategory of $\Ar_{\cT}(\cD)$ on the morphisms in $\sR$. Suppose that $p: \cC \to \cD$ is a $\cT$-fibration which admits $p$-cocartesian lifts over all edges in $\sL$. Then the target map
\[ \ev_1: \cC \times_{\cD}  \Ar^R_{\cT}(\cD) \to \cD \]
is a $\cT$-cocartesian fibration.
\end{enumerate}
\end{proposition}
\begin{proof} (1): For the first assertion, since we have the factorization system $(\sL_t, \sR_t)$ on the fibers $\cD_t$ for all $t \in \cT$, $\ev_0$ is fiberwise a cartesian fibration, with an edge in $\Ar^{L_t}(\cD_t) = \Ar^L_{\cT}(\cD)_t$
\[ \begin{tikzcd}[row sep=2em, column sep=2em]
x_0 \ar{r}{\alpha} \ar{d}{f} & y_0 \ar{d}{\beta} \\
x_1 \ar{r}{g} & y_1
\end{tikzcd} \]
$(\ev_0)_t$-cartesian if and only if $g$ is in $\sR_t$. Since the factorization systems in the fibers are compatible with restriction, it follows that $\ev_0$ is in addition $\cT$-cartesian.

For the second assertion, repeat the argument with the factorization system $(\sL,\sR)$ on $\cD$ itself.

\noindent (2): The argument is dual to (1), except that now the $\ev_1$-cocartesian edges are given by
\[ (\alpha, (f,g)) : (c_0, x_0 \rightarrow y_0) \to (c_1, x_1 \rightarrow y_1) \]
with $f \in \sL$ and $\alpha$ a $p$-cocartesian edge.
\end{proof}

\begin{theorem} \label{thm:GeneralFreeFibration} Suppose we are in the setup of \cref{prop:InertCartesianFibration}(2) so that we have a $\cT$-fibration $p: \cC \to \cD$ that admits $p$-cocartesian lifts over all edges in $\sL$.
\begin{enumerate} \item For every cocartesian fibration $q: \cE \to \cD$, restriction along the inclusion $i: \cC \to \cC \times_{\cD} \Ar^R_{\cT}(\cD)$ yields a trivial fibration
\[ i^\ast: \Fun_{/\cD}^{\cocart}(\cC \times_{\cD} \Ar^R_{\cT}(\cD), \cE) \to \Fun_{/\cD}^{L}(\cC, \cE) \]
where we define
\begin{align*}
\Fun_{/\cD}^{\cocart}(\cC \times_{\cD} \Ar^R_{\cT}(\cD), \cE) & \coloneq \Fun_{/\cD}(\leftnat{(\cC \times_{\cD} \Ar^R_{\cT}(\cD))},\leftnat{\cE}), \\
\Fun_{/\cD}^{L}(\cC, \cE) & \coloneq \Fun_{/\cD}((\cC,M),\leftnat{\cE}),
\end{align*}
and the marked edges $M$ in $\cC$ are the $p$-cocartesian edges of $\cC$ over $\sL$.\footnote{Note that objects of $\Fun^L_{/\cD}(\cC, \cE)$ are necessarily $\cT$-functors.} In other words,
$$i: (\cC,M) \to \leftnat{(\cC \times_{\cD} \Ar^R_{\cT}(\cD))}$$
is a cocartesian equivalence in $\sSet^+_{/\cD}$. 

\item Let $M'$ denote the $\ev_1$-cocartesian edges in $\cC \times_{\cD} \Ar^R_{\cT}(\cD)$ over $\sL$ and define
\[ \Fun_{/\cD}^{L}(\cC \times_{\cD} \Ar^R_{\cT}(\cD), \cE) \coloneq \Fun_{/\cD}((\cC \times_{\cD} \Ar^R_{\cT}(\cD),M'),\leftnat{\cE}).\]
Then we have an adjunction
\[ \adjunct{i_!}{\Fun_{/\cD}^L(\cC,\cE)}{\Fun_{/\cD}^{L}(\cC \times_{\cD} \Ar^R_{\cT}(\cD), \cE)}{i^\ast} \]
where $i_!$ is the fully faithful inclusion of the full subcategory $\Fun_{/\cD}^{\cocart}(\cC \times_{\cD} \Ar^R_{\cT}(\cD), \cE)$ under the equivalence of (1).
\end{enumerate}
\end{theorem}
\begin{proof} (1): Given a monomorphism $A \to B$ of simplicial sets, we need to solve the lifting problem
\[ \begin{tikzcd}[row sep=2em, column sep=2em]
A^\flat \times \leftnat{(\cC \times_{\cD} \Ar^R_{\cT}(\cD))} \bigcup_{A^\flat \times (\cC,M)} B^\flat \times (\cC,M) \ar{r} \ar{d} & \leftnat{\cE} \ar{d}{q} \\
B^\flat \times \leftnat{(\cC \times_{\cD} \Ar^R_{\cT}(\cD))} \ar{r} \ar[dotted]{ru} & \cD^\sharp.
\end{tikzcd} \]
Let us suppress markings for clarity. We can factor this square as
\[ \begin{tikzcd}[row sep=2em, column sep=2em]
A \times (\cC \times_{\cD} \Ar^R_{\cT}(\cD)) \bigcup_{A \times \cC} B \times \cC \ar{r} \ar{d} & \cE \ar{r}{\iota} & \Ar^\cocart_{\cT}(\cE) \ar[->>]{d}{\simeq} \ar{r}{\ev_1} & \cE \ar{d}{q} \\
B \times (\cC \times_{\cD} \Ar^R_{\cT}(\cD)) \ar{rr}{\lambda} \ar[dotted]{rru} & &  \cE \times_{\cD} \Ar_{\cT}(\cD) \ar{r}{\ev_1} & \cD
\end{tikzcd} \]
where the map $\lambda$ is given by
\begin{align*} B \times (\cC \times_{\cD} \Ar^R_{\cT}(\cD)) &\to B \times \cC \to \cE, \\
B \times (\cC \times_{\cD} \Ar^R_{\cT}(\cD)) &\to B \times \Ar^R_{\cT}(\cD) \to \Ar_{\cT}(\cD),
\end{align*}
and the map $\Ar^\cocart_{\cT}(\cE) \to \cE \times_{\cD} \Ar_{\cT}(\cD)$ is a pullback of the known trivial fibration $\Ar^\cocart(\cE) \to {\cE} \times_{\cD} \Ar(\cD)$ of \cite[Lem.~2.23]{Exp2} by the identity section $\cT \to \Ar(\cT)$, hence is a trivial fibration. Therefore, the dotted arrow exists, and postcomposing by $\ev_1$ yields the desired lift.

(2): We need to show that $\Fun^\cocart_{/\cD}(\cC \times_{\cD} \Ar^R_{\cT}(\cD),\cE) \subset \Fun^L_{/\cD}(\cC \times_{\cD} \Ar^R_{\cT}(\cD),\cE)$ is a coreflective subcategory. For this, it suffices to show that for every object $F \in \Fun^L_{/\cD}(\cC \times_{\cD} \Ar^R_{\cT}(\cD),\cE)$, there exists a colocalization $\epsilon_F: G \to F$ relative to $\Fun^\cocart_{/\cD}(\cC \times_{\cD} \Ar^R_{\cT}(\cD),\cE)$ in the sense of \cite[Def.~5.2.7.6]{HTT} (after taking opposites there). We will construct this explicitly as follows. First define a homotopy
\[ H: \Delta^1 \times \cC \times_{\cD} \Ar^R_{\cT}(\cD) \to \cE \times_{\cD} \Ar_{\cT}(\cD) \]
between the functors
\begin{align*} H_0 = (F \circ i,\subset) &: \cC \times_{\cD} \Ar^R_{\cT}(\cD) \to \cE \times_{\cD} \Ar_{\cT}(\cD), \; & (c, x \ra y) \mapsto (F(c, x=x), x \ra y), \\
H_1 = (\id,\iota_{\cD} q) \circ F &: \cC \times_{\cD} \Ar^R_{\cT}(\cD) \to \cE \to \cE \times_{\cD} \Ar_{\cT}(\cD), \; & (c, x \ra y) \mapsto (F(c, x \ra y), y=y)
\end{align*}
in the following way: let $\min, \max: \Delta^1 \times \Delta^1 \to \Delta^1$ be the min and max maps, form the functors
\begin{align*}
F' = \Ar_{\cT}(F) \circ (\iota, \text{min}^{\ast}) &: \cC \times_{\cD} \Ar^{R}_{\cT}(\cD) \to \Ar_{\cT}(\cC \times_{\cD} \Ar^{R}_{\cT}(\cD)) \to \Ar_{\cT}(\cE), \\
G' = \text{max}^{\ast} \circ \pr_2 &: \cC \times_{\cD} \Ar^{R}_{\cT}(\cD) \to \Ar^{R}_{\cT}(\cD) \to \Ar_{\cT}(\Ar_{\cT}(\cD)),
\end{align*}
and let $H$ be the adjoint of the resulting map 
\[ (F', G'): \cC \times_{\cD} \Ar^R_{\cT}(\cD) \to \Ar_{\cT}(\cE \times_{\cD} \Ar_{\cT}(\cD)). \]
We then place $H$ into the commutative diagram
\[ \begin{tikzcd}[row sep=2em, column sep=2em]
\{1\} \times \cC \times_{\cD} \Ar^R_{\cT}(\cD) \ar{r}{F} \ar{d} & \cE \ar{r}{\iota} & \Ar^\cocart_{\cT}(\cE) \ar[->>]{d}{\simeq} \ar{r}{\ev_1} & \cE \ar{d}{q} \\
\Delta^1 \times \cC \times_{\cD} \Ar^R_{\cT}(\cD) \ar{rr}{H} \ar[dotted]{rru}{\epsilon'_F} & & \cE \times_{\cD} \Ar_{\cT}(\cD) \ar{r}{\ev_1} & \cD.
\end{tikzcd} \]
Let $\epsilon'_F$ be any filler, and define $\epsilon_F = \ev_1 \circ \epsilon'_F$. Note that $\epsilon_F(0) \simeq i_! i^\ast F$ and $\epsilon_F(1) = F$. We now make the following simple observations, whose verification we leave to the reader:
\begin{enumerate}
\item For every natural transformation $\theta: F \to G$, the square
\[ \begin{tikzcd}[row sep=2em, column sep=2em]
i_! i^\ast F \ar{r}{i_! i^\ast \theta} \ar{d}{\epsilon_F} & i_! i^\ast G \ar{d}{\epsilon_G} \\
F \ar{r}{\theta} & G
\end{tikzcd} \]
is homotopy commutative.
\item $i_! i^\ast \epsilon_F$ is an equivalence.
\item $\epsilon_{i_! i^\ast F}$ is an equivalence.
\end{enumerate}
Examining the part of the proof of \cite[Prop.~5.2.7.4]{HTT} that establishes the implication 5.2.7.4(3) $\Rightarrow$ 5.2.7.4(1), we conclude that $\epsilon_F$ is indeed a colocalization, so we are done by \cite[Prop.~5.2.7.8]{HTT}.
\end{proof}

\begin{rem}
Replacing the $\cT$-factorization system $(\sL_t, \sR_t)_{t \in \cT}$ by the factorization system $(\sL,\sR)$ on $\cD$ (\cref{lem:TotalFactSystem}), note that since edges in $\sR$ map down to equivalences in $\cT^{\op}$, we have that $\Ar^R_{\cT}(\cD) \xto{\simeq} \Ar^{R}(\cD)$ where by the latter $\infty$-category we mean the full subcategory of $\Ar(\cD)$ on those edges in $\sR$. \cref{thm:GeneralFreeFibration} could thus be formulated entirely in `non-parametrized' terms; this is related to the fact that a $\cT$-functor $q: \cE \to \cD$ is a $\cT$-cocartesian fibration if and only if it is a cocartesian fibration \cite[Rem.~7.4]{Exp2}. In this form, Ayala--Mazel-Gee--Rozenblyum have also articulated \cref{thm:GeneralFreeFibration}(1) model-independently in terms of an adjunction of $\infty$-categories \cite[Prop.~B.1]{ayala2017actorization}.
\end{rem}

We end this section by giving two important applications of \cref{thm:GeneralFreeFibration}.

\begin{exm} \label{exm:freeCocartesianFibration}
Let $(\sL_t,\sR_t)_{t \in \cT}$ be the $\cT$-factorization system given by letting $\sL_t$ be the equivalences and $\sR_t$ be all morphisms for every $t \in \cT$. Then $\Ar^R_{\cT}(\cD) = \Ar_{\cT}(\cD)$, and $\cC \times_{\cD} \Ar_{\cT}(\cD)$ is the free $\cT$-cocartesian fibration on $\cD$ (\cite[Def.~7.6]{Exp2}). By \cref{thm:GeneralFreeFibration}(1), we see that $i: \cC \to \cC \times_{\cD} \Ar_{\cT}(\cD)$ has the expected universal property: for every $\cT$-cocartesian fibration $\cE \to \cD$,
\[ i^\ast: \Fun^{\cocart}_{/\cD,\cT}(\cC \times_{\cD} \Ar_{\cT}(\cD), \cE) \to \Fun_{/\cD,\cT}(\cC,\cE) \]
is an equivalence. This promotes to an adjunction
\[ \adjunct{\text{Fr}^\cocart}{(\CatT)_{/\cD}}{(\CatT)^{\cocart}_{/\cD} \simeq \Cat_{\cD}}{\mathrm{U}}. \]
When $\cT = \ast$, this recovers \cite[Thm.~4.5]{GHN}.

By \cref{thm:GeneralFreeFibration}(2), we also have an adjunction
\[ \adjunct{i_!}{\Fun_{/\cD,\cT}(\cC,\cE)}{\Fun_{/\cD,\cT}(\cC \times_{\cD} \Ar_{\cT}(\cD), \cE)}{i^\ast} \]
in which $i_!$ is fully faithful.
\end{exm}

\begin{exm}
Suppose $\cT = \ast$ and consider the inert-active factorization system on an $\infty$-operad $\cO^{\otimes}$. Let $p: \cC^{\otimes} \to \cO^{\otimes}$ be a fibration of $\infty$-operads. Then $\Env_{\cO}(\cC)^{\otimes} \coloneq \cC^{\otimes} \times_{\cO^{\otimes}} \Ar^{\act}(\cO^{\otimes})$ is the \emph{$\cO$-monoidal envelope} of $\cC^{\otimes}$ \cite[Constr.~2.2.4.1]{HA}, and by \cref{thm:GeneralFreeFibration}(1) for any $\cO$-monoidal $\infty$-category $\cD^{\otimes}$ we have that
\[ \Fun^{\otimes}_{\cO}(\Env_{\cO}(\cC), \cD) \xto{\simeq} \Alg_{\cC/\cO}(\cD). \]
This recovers \cite[Prop.~2.2.4.9]{HA}.
\end{exm}

\section{Parametrized pairing construction}

\label{sec:pairing}

In this section, we first introduce the concept of a \emph{$\cT$-flat fibration} $p: \cX \to \cB$, which will amount to a condition on $p$ that ensures that the pullback functor
\[ p^*: (\Cat_{\cT})^{/\cB} \to (\Cat_{\cT})^{/\cX} \]
admits a right adjoint $p_*$ (\cref{rem:FromUniversalPropertyToAdjointExistence}). Given another $\cT$-fibration $q: \cY \to \cB$, we then recall from \cite[Constr.~9.1]{Exp2} the \emph{$\cB$-relative $\cT$-pairing construction} $\widetilde{\Fun}_{\cB,\cT}(\cX,\cY)$ (\cref{conthm:PairingConstructionRecalled}) as a certain $\cT$-fibration over $\cB$. In our discussion in \cite[\S 9]{Exp2}, we only established the properties of the $\cT$-pairing construction needed for our application to proving the existence theorem for $\cT$-left Kan extensions. We enter into a more systematic discussion here by first proving its base-change property (\cref{prop:BaseChangePairingConstruction}) and then its universal property internal to $\cT$-$\infty$-categories (\cref{thm:PairingConstructionUniversalProperty}), from which it follows that 
\[ \widetilde{\Fun}_{\cB,\cT}(\cX,\cY) \simeq p_* p^* (\cY \xto{q} \cB) \]
at the level of the underlying $\infty$-category $(\Cat_{\cT})^{/\cB}$ of $\sSet^+_{/\leftnat{\cB}}$.\footnote{We equip $\sSet^+_{/\leftnat{\cB}}$ with the slice model structure with respect to the cocartesian model structure on $\sSet^+_{/(\cT^{\op})^{\sharp}}$. Since $\leftnat{\cB} \to (\cT^{\op})^{\sharp}$ is fibrant, we may indeed identify the underlying $\infty$-category as $(\Cat_{\cT})^{/\cB}$.} We then apply the $\cT$-pairing construction to study $\cT$-(co)limits in a $\cT$-$\infty$-category of sections (\cref{thm:ParametrizedLimitsAndColimitsInSectionCategories}). This material will be used in \cite{paramalg} to understand $\cT$-(co)limits in $\cT$-$\infty$-categories of $\cO$-algebras for a $\cT$-$\infty$-operad $\cO$.

\begin{definition} \label{def:ParametrizedFlatFibration}
Let $p: \cX \to \cB$ be a $\cT$-fibration. We say that $p$ is a \emph{$\cT$-flat fibration} if for every $t \in \cT$, the pullback $p_t: \cX_t \to \cB_t$ is flat.
\end{definition}


In what follows, for a $\cT$-$\infty$-category $\cB$ we let $\Ar^{\cocart}(\cB) \subset \Ar(\cB)$ denote the full subcategory on arrows that are cocartesian edges with respect to the structure map to $\cT^\op$.

\begin{conthm}[{\cite[Def.~9.1]{Exp2}}] \label{conthm:PairingConstructionRecalled}
Let $p: \cX \to \cB$ be a $\cT$-flat fibration and consider the span of marked simplicial sets
\[ \begin{tikzcd}
\leftnat{\cB} & (\Ar^{\cocart}(\cB) \times_{\ev_1,\cB,p} \cX, \sE) \ar{l}[swap]{\ev_0} \ar{r}{\pr_{\cX}} & \leftnat{\cX}
\end{tikzcd} \]
in which $\cB$ and $\cX$ are given the cocartesian markings (with respect to the structure maps to $\cT^\op$), and an edge $e$ in $\Ar^{\cocart}(\cB) \times_{\cB} \cX$ is marked if and only if $\ev_0(e)$ is marked and $\pr_{\cX}(e)$ is marked. The functor
\[ (\ev_0)_{\ast} (\pr_{\cX})^{\ast}: \sSet^+_{/\leftnat{\cX}} \to \sSet^+_{/\leftnat{\cB}} \]
is then right Quillen with respect to the slice model structures induced from the cocartesian model structure on $\sSet^+_{/\cT^\op}$. For a $\cT$-fibration $q: \cY \to \cB$, we then define the \emph{$\cT$-pairing} of $(\cX,p)$ and $(\cY,q)$ to be the $\cT$-fibration over $\cB$ given by
\[ \widetilde{\Fun}_{\cB,\cT}(\cX,\cY) \coloneq (\ev_0)_{\ast} (\pr_{\cX})^{\ast} q^\ast(\leftnat{\cY}),\]
where the marked edges are precisely the cocartesian edges with respect to the structure map to $\cT^\op$.
\end{conthm}
\begin{proof} The assertion that $(\ev_0)_{\ast} (\pr_{\cX})^{\ast}$ is right Quillen was proved under the assumption that $p$ is a $\cT$-cocartesian or $\cT$-cartesian fibration in \cite[Thm.~9.3(1)]{Exp2}. However, that assumption was only used in the proof to show that $\ev_0$ is flat. Using our weaker assumption that $p$ is $\cT$-flat, this follows instead from \cref{lem:FlatCategoricalFibrationArrowLemma} applied to the factorization system $(\sL,\sR)$ on $\cB$ with $\sL$ given by the cocartesian edges and $\sR$ given by those edges lying over equivalences in $\cT^\op$ \cite[Ex.~5.2.8.15]{HTT}.
\end{proof}

\begin{rec} \label{rec:PairingConstructionCocartesianFunctoriality}
If $p: \cX \to \cB$ is a $\cT$-cartesian fibration and $q: \cY \to \cB$ is a $\cT$-cocartesian fibration, then we showed in \cite[Thm.~9.3]{Exp2} that $r: \widetilde{\Fun}_{\cB,\cT}(\cX,\cY) \to \cB$ is a $\cT$-cocartesian fibration. Moreover, we may produce $\widetilde{\Fun}_{\cB,\cT}(\cX,\cY)$ as a marked simplicial set with the $r$-cocartesian edges marked in the following way: let $\sE' \subset (\Ar^{\cocart}(\cB) \times_{\cB} \cX)_1$ be the minimal collection of edges closed under composition that contains the class $\sE$ in \cref{conthm:PairingConstructionRecalled} and the $\ev_0$-cartesian edges in $\Ar^{\cocart}(\cB) \times_{\cB} \cX$, which are those edges
\[ \left( \begin{tikzcd}
b_0 \ar{r} \ar{d} & b_1 \ar{d} \\ 
c_0 \ar{r}{f} & c_1
\end{tikzcd}, \: x_0 \xto{g} x_1 \right) \]
such that $f$ is sent to an equivalence in $\cT^\op$ and $g$ is a $p$-cartesian edge.\footnote{See \cite[Lem.~7.5]{Exp2} for a justification as to why we can take $p$-cartesian edges here as opposed to fiberwise cartesian.} Then the span of marked simplicial sets
\[ \begin{tikzcd}
\cB^{\sharp} & (\Ar^{\cocart}(\cB) \times_{\cB} \cX, \sE') \ar{l}[swap]{\ev_0} \ar{r}{\ev_1} & \cB^{\sharp}
\end{tikzcd} \]
defines via $(\ev_0)_{\ast} (\ev_1)^\ast (\cY,q\text{-cocart})$ the same underlying simplicial set $\widetilde{\Fun}_{\cB,\cT}(\cX,\cY)$ as before, but with the $r$-cocartesian edges marked. Unwinding the definitions, we thus see that $\widetilde{\Fun}_{\cB,\cT}(\cX,\cY)$ enjoys the following additional functoriality with respect to morphisms in $\cB$: for every fiberwise morphism $f: b \to b' \in \cB_{t}$, we have a pushforward functor
\begin{align*} f_!: \Fun_{\cT^{/t}}(\cX_{\underline{b}}, \cY_{\underline{b}}) \to \Fun_{\cT^{/t}}(\cX_{\underline{b'}}, \cY_{\underline{b'}}), \;  F \mapsto f_! \circ F \circ f^{\ast}
\end{align*}
where $f^{\ast}: \cX_{\underline{b'}} \to \cX_{\underline{b}}$ and $f_!: \cY_{\underline{b}} \to \cY_{\underline{b'}}$ are the $\cT^{/t}$-functors encoded by $p$ and $q$.
\end{rec}

We next establish the compatibility of the $\cT$-pairing construction with base-change. First, we need a lemma.

\begin{lemma} \label{lem:BaseChangeCocartesianArrows}
Let $f: \cA \to \cB$ be a $\cT$-functor. The functor $$\psi: \Ar^{\cocart}(\cA) \to \cA \times_{f,\cB,\ev_0} \Ar^{\cocart}(\cB)$$ induced by $f$ is a homotopy equivalence of cartesian fibrations over $\cA$ (with respect to $\ev_0$ on the source and projection to $\cA$ on the target).
\end{lemma}
\begin{proof}
By \cite[Lem.~9.2(1)]{Exp2}, $\ev_0: \Ar^{\cocart}(\cA) \to \cA$ is a cartesian fibration and an edge $e$ is $\ev_0$-cartesian if and only if the projection of $\ev_1(e)$ to $\cT^\op$ is an equivalence. It follows that $\psi$ preserves cartesian edges, so to show $\psi$ is a homotopy equivalence it suffices to check that for every $a \in \cA$, the map on fibers
$$\psi_a: \underline{a} = \{a\} \times_{\cA, \ev_0} \Ar^{\cocart}(\cA) \to \underline{f(a)} = \{ f(a) \} \times_{\cB, \ev_0} \Ar^{\cocart}(\cB)$$
is an equivalence of $\infty$-categories. But if $a$ lies over $t \in T^\op$, then the induced projections $\underline{a} \to (\cT^{/t})^\op$ and $\underline{f(a)} \to (\cT^{/t})^\op$ are equivalences (cf. \cite[Notn.~2.28]{Exp2}), so $\psi_a$ is an equivalence.
\end{proof}

\begin{proposition} \label{prop:BaseChangePairingConstruction}
Let $f: \cA \to \cB$ be a $\cT$-functor, let $\cX \to \cB$ be a $\cT$-flat fibration, and let $\cY \to \cB$ be a $\cT$-fibration. We have a canonical and natural equivalence of $\cT$-fibrations over $\cA$
\[ \widetilde{\Fun}_{\cA,\cT}(\cX \times_{\cB} \cA ,\cY \times_{\cB} \cA) \simeq \widetilde{\Fun}_{\cB,\cT}(\cX,\cY) \times_{\cB} \cA. \]
\end{proposition}
\begin{proof}
Consider the morphism of spans
\[ \begin{tikzcd}[column sep=6ex]
& \Ar^{\cocart}(\cA) \times_{\cA} (\cA \times_{\cB} \cX) \ar{ld}[swap]{\ev_0} \ar{rd}{\pr_{\cX}} \ar{d}{\phi} & \\
\cA & \cA \times_{\cB} \Ar^{\cocart}(\cB) \times_{\cB} \cX \ar{l}{\pr_{\cA}} \ar{r}[swap]{\pr_{\cX}} & \cX
\end{tikzcd} \]
where $\phi$ is induced by $f$. Noting that $\cT$-flat fibrations are stable under pullback, we see that $\phi$ induces a comparison functor (after marking as necessary)
\[ \Phi: \widetilde{\Fun}_{\cB,\cT}(\cX,\cY) \times_{\cB} \cA \to \widetilde{\Fun}_{\cA,\cT}(\cX \times_{\cB} \cA ,\cY \times_{\cB} \cA). \]
By \cref{lem:BaseChangeCocartesianArrows}, $\phi$ is a homotopy equivalence. Moreover, the homotopy inverse respects the projection to $\cX$, so by the proof of \cite[Lem.~2.27]{Exp2}, $\Phi$ is an equivalence of $\cT$-fibrations over $\cA$. 
\end{proof}

We can use the base-change property of the $\cT$-pairing construction explicated in \cref{prop:BaseChangePairingConstruction} to give a more transparent identification of its parametrized fibers {\cite[Prop.~9.7]{Exp2}}.

\begin{cor} \label{cor:PairingFibers}
Let $\cX \to \cB$ be a $\cT$-flat fibration and $\cY \to \cB$ a $\cT$-fibration. For every $b \in \cB$ over $t \in \cT^\op$, we have an equivalence of $\cT^{/t}$-$\infty$-categories
\[ \widetilde{\Fun}_{\cB,\cT}(\cX,\cY)_{\underline{b}} \simeq \underline{\Fun}_{\cT^{/t}}(\cX_{\underline{b}}, \cY_{\underline{b}}) \]
\end{cor}
\begin{proof}
We may invoke \cref{prop:BaseChangePairingConstruction} to replace $\cB$ with $\cB_{\underline{t}}$, and invoke \cref{prop:BaseChangePairingConstruction} again with $\cA = \underline{b} \to \cB$ to reduce to the case $\cB = \cT$, for which $\widetilde{\Fun}_{\cT,\cT}(-,-) \cong \Fun_{\cT}(-,-)$ as marked simplicial sets.
\end{proof}

We next proceed to articulate the universal property of the $\cT$-pairing construction (\cref{thm:PairingConstructionUniversalProperty}) as a partially-defined internal hom for $\cT$-fibrations over a fixed base $\cT$-$\infty$-category.

\begin{notation}
Let $\cB$ be a $\cT$-$\infty$-category, let $p: \cX \to \cB$, $q: \cY \to \cB$ be $\cT$-fibrations over $\cB$, and let $q_{\circ}: \underline{\Fun}_{\cT}(\cX, \cY) \to \underline{\Fun}_{\cT}(\cX, \cB)$ be the $\cT$-functor given by postcomposition by $q$. We then let
\[ \underline{\Fun}_{/\cB,\cT}(\cX,\cY) \coloneq \ast_{\cT} \times_{\sigma_p,\underline{\Fun}_{\cT}(\cX, \cB),q_{\circ}} \underline{\Fun}_{\cT}(\cX, \cY) \]
denote the $\cT$-$\infty$-category of $\cT$-functors $\cX \to \cY$ over $\cB$.
\end{notation}

\begin{lemma} \label{lem:RelativeFunctorsViaSpanConstr}
Let $p: \cX \to \cB$ be a $\cT$-fibration and consider the span of marked simplicial sets
\[ \begin{tikzcd}[column sep=6ex]
(\cT^\op)^\sharp & \Ar(\cT^\op)^\sharp \times_{\ev_1, (\cT^\op)^\sharp} \leftnat{\cX} \ar{l}[swap]{\ev_0} \ar{r}{p \circ \pr_{\cX}} & \leftnat{\cB}.
\end{tikzcd} \]
For any $\cT$-fibration $q: \cY \to \cB$, we then have an equivalence of $\cT$-$\infty$-categories
\[ (\ev_0)_{\ast} (p \circ \pr_{\cX})^\ast (\leftnat{Y}) \simeq \underline{\Fun}_{/\cB,\cT}(\cX,\cY) \]
which is an isomorphism at the level of marked simplicial sets.
\end{lemma}
\begin{proof}
By definition, given a map of marked simplicial sets $K \to (\cT^\op)^\sharp$, we have natural bijections
\begin{align*}
& \Hom_{/(\cT^\op)^\sharp}(K, (\ev_0)_{\ast} (p \circ \pr_{\cX})^\ast (\leftnat{\cY})) \cong \Hom_{/\leftnat{\cB}}(K \times_{\cT^\op} \Ar(\cT^\op)^\sharp \times_{\cT^\op} \leftnat{\cX}, \leftnat{\cY}) \\
& \cong \Hom_{/(\cT^\op)^\sharp}(K \times_{\cT^\op} \Ar(\cT^\op)^\sharp \times_{\cT^\op} \leftnat{\cX}, \leftnat{\cY})
\times_{ \Hom_{/(\cT^\op)^{\sharp}}(K \times_{\cT^\op} \Ar(\cT^\op)^{\sharp} \times_{\cT^\op} \leftnat{\cX}, \leftnat{\cB}) }
\{\phi \circ \pr\}
\end{align*}
yielding an isomorphism of marked simplicial sets over $\cT^\op$
\[ (\ev_0)_{\ast} (p \circ \pr_{\cX})^\ast (\leftnat{Y}) \cong \leftnat{\underline{\Fun}_{/\cB,\cT}(\cX,\cY)}. \]
\end{proof}

\begin{theorem} \label{thm:PairingConstructionUniversalProperty}
We have a canonical equivalence of $\cT$-$\infty$-categories
\[ \underline{\Fun}_{/\cB,\cT}(\cC,\widetilde{\Fun}_{\cB,\cT}(\cX,\cY)) \simeq \underline{\Fun}_{/\cB,\cT}(\cC \times_{\cB} \cX,\cY) \]
natural in $\cT$-flat fibrations $p: \cX \to \cB$ and $\cT$-fibrations $\cC, \cY \to \cB$.
\end{theorem}
\begin{proof}
By \cref{prop:BaseChangePairingConstruction}, we may assume $\cC = \cB$ without loss of generality. We then adopt essentially the same strategy that we used to show \cite[Eqn.~9.12.1]{Exp2} (which is the equivalence when $\cY = \cB \times_{\cT^\op} \cE$) by considering the diagram of marked simplicial sets
\[ \begin{tikzcd}
\Ar(\cT^\op)^\sharp \times_{\cT^\op} \leftnat{\cX} \ar{r}{(j,\id)} \ar{d}{(\id,p)} & (\Ar^{\cocart}(\cB) \times_{\cB} \cX, \sE) \ar{r}{p \circ \pr_{\cX}} \ar{d}{\ev_0} & \leftnat{\cB} \\
\Ar(\cT^\op)^\sharp \times_{\cT^\op} \leftnat{\cB} \ar{d}{\ev_0} \ar{r}{\pr_{\cB}} & \leftnat{\cB} \\
(\cT^\op)^\sharp
\end{tikzcd} \]
where $j$ is the composite $\Ar(\cT^\op) \times_{\cT^\op} \cX \xto{\pr_{\cX}} \cX \xto{p} \cB \xto{\iota} \Ar^{\cocart}(\cB)$, $\iota$ being the identity section. Let
\[ i: \Ar(\cT^\op)^\sharp \times_{\cT^\op} \leftnat{\cX} \to (\Ar(\cT^\op)^\sharp \times_{\cT^\op} \leftnat{\cB}) \times_{\leftnat{\cB}} (\Ar^{\cocart}(\cB) \times_{\cB} \cX, \sE) \]
denote the induced map to the pullback. By \cite[Lem.~9.12]{Exp2}, $i$ is a homotopy equivalence with respect to the projection to $\cX$. By \cite[Lem.~2.27]{Exp2} and \cref{lem:RelativeFunctorsViaSpanConstr}, we obtain an equivalence of $\cT$-$\infty$-categories
\[ \underline{\Fun}_{/\cB,\cT}(\cB,\widetilde{\Fun}_{\cB,\cT}(\cX,\cY)) \xto{\simeq} \underline{\Fun}_{/\cB,\cT}(\cX,\cY). \]
\end{proof}

\begin{rem} \label{rem:FromUniversalPropertyToAdjointExistence}
In the statement of \cref{thm:PairingConstructionUniversalProperty}, if we replace the span
\[ \begin{tikzcd}
\leftnat{\cB} & (\Ar^{\cocart}(\cB) \times_{\cB} \cX, \sE) \ar{l}[swap]{\ev_0} \ar{r}{p \circ \pr_{\cX}} & \leftnat{\cB}
\end{tikzcd}  \]
with
\[ \begin{tikzcd}
\leftnat{\cB} & (\Ar^{\cocart}(\cB) \times_{\cB} \cX, \sE) \ar{l}[swap]{\ev_0} \ar{r}{\pr_{\cX}} & \leftnat{\cX}
\end{tikzcd} \]
then the same argument as in the proof of \cref{thm:PairingConstructionUniversalProperty} shows that for all $\cT$-fibrations $\cD \to \cX$, we have a canonical and natural equivalence\footnote{Here, $(\ev_0)_* (\pr_{\cX})^*: \sSet^+_{/\leftnat{\cX}} \to \sSet^+_{/\leftnat{\cB}}$ and $(\ev_0)_* (\pr_{\cX})^* \cD \coloneq (\ev_0)_* (\pr_{\cX})^* (\leftnat{\cD})$ regarded as a $\cT$-fibration over $\cB$.}
\[ \underline{\Fun}_{/\cB, \cT}(\cC, (\ev_0)_* (\pr_{\cX})^* \cD) \xto{\simeq} \underline{\Fun}_{/\cX, \cT}(\cC \times_{\cB} \cX, \cD). \]
Passing first to cocartesian sections and then to mapping spaces, this shows that at the level of underlying $\infty$-categories, $(\ev_0)_* (\pr_{\cX})^*: (\Cat_{\cT})^{/\cX} \to (\Cat_{\cT})^{/\cB}$ computes the right adjoint to $p^*$. This justifies the terminology of ``$\cT$-flat fibration'' since these are indeed exponentiable in the parametrized sense. Moreover, we then see that
\[ \widetilde{\Fun}_{\cB,\cT}(\cX,-) \simeq p_* p^*(-) \]
as endofunctors of $(\Cat_{\cT})^{/\cB}$.
\end{rem}

\begin{rem} \label{rem:FunctorPairingComparison}
Suppose $\cX, \cY \to \cB$ are $\cT$-cocartesian fibrations and let $\underline{\Fun}_{\cB}(\cX,\cY)$ denote the internal hom construction of \cite[\S 3]{Exp2},\footnote{We mildly abuse our notational conventions by writing $\underline{\Fun}_{\cB}(-,-)$ instead of $\underline{\Fun}_{\cB^\op}(-,-)$ as we would do if $\cB = \ast_{\cT} = \cT^{\op}$.} so that $\underline{\Fun}_{\cB}(\cX,\cY) \to \cB$ is a $\cT$-cocartesian fibration. Consider the morphism of spans
\[ \begin{tikzcd}
& (\Ar^{\cocart}(\cB) \times_{\cB} \cX, \sE) \ar{rd} \ar{ld} \ar{d}{(i,\id)} & \\
\leftnat{\cB} & \Ar(\cB)^{\sharp} \times_{\cB} \leftnat{\cX} \ar{l} \ar{r} & \leftnat{\cB}
\end{tikzcd} \]
in which $i$ is the inclusion $\Ar^{\cocart}(\cB) \subset \Ar(\cB)$. Then this morphism induces a $\cT$-functor over $\cB$
\[ \rho: \underline{\Fun}_{\cB}(\cX,\cY) \to \widetilde{\Fun}_{\cB,\cT}(\cX,\cY). \]
that upon passage to cocartesian sections in the source, regarded as $\cT$-sections in the target, induces the inclusion
\[ \Fun_{\cB}(\cX, \cY) \to \Fun_{/\cB,\cT}(\cX,\cY). \]
\end{rem}

Beware that even if $\cX$ is in addition $\cT$-cartesian so that $\widetilde{\Fun}_{\cB,\cT}(\cX,\cY)$ is $\cT$-cocartesian, $\rho$ will not preserve cocartesian edges in general; indeed, one observes that the above functor $(i,\id)$ does not carry the marking $\sE'$ of \cref{rec:PairingConstructionCocartesianFunctoriality} to marked edges in the target. Nonetheless, we have the following proposition.

\begin{proposition} \label{prop:PairingConstructionReducesToFunctorInternalHomWhenSourceConstant}
In the situation of \cref{rem:FunctorPairingComparison}, if $\cX \simeq \cB \times_{\cT^\op} \cK$ for some $\cT$-$\infty$-category $\cK$, then the comparison $\cT$-functor $\rho$ implements an equivalence
\[ \rho: \underline{\Fun}_{\cB}(\cB \times_{\cT^\op} \cK,\cY) \xto{\simeq} \widetilde{\Fun}_{\cB,\cT}(\cB \times_{\cT^\op} \cK,\cY) \]
of $\cT$-cocartesian fibrations over $\cB$. Consequently, for every $\cT$-cocartesian fibration $\cC \to \cB$, the equivalence of \cref{thm:PairingConstructionUniversalProperty} restricts to
\[ \underline{\Fun}_{\cB}(\cC, \widetilde{\Fun}_{\cB,\cT}(\cB \times_{\cT^\op} \cK,\cY)) \simeq \underline{\Fun}_{\cB}(\cC \times_{\cT^\op} \cK, \cY). \]
\end{proposition}
\begin{proof}
The consequence will follow immediately from the universal property of $\underline{\Fun}_{\cB}(-,-)$ once we establish the first claim. For this, first observe that since $\cX \simeq \cB \times_{\cT^\op} \cK$, $\cX$ is both a $\cT$-cocartesian and $\cT$-cartesian fibration over $\cB$ such that the fiberwise cartesian edges, fiberwise cocartesian edges, and fiberwise equivalences all coincide in $\cX$. In particular, the functor $(i,\id)$ in \cref{rem:FunctorPairingComparison} carries the class $\sE'$ of \cref{rec:PairingConstructionCocartesianFunctoriality} into the marked edges of $\Ar(\cB)^{\sharp} \times_{\cB} \leftnat{\cX}$, so $\rho$ is a morphism of $\cT$-cocartesian fibrations. It therefore suffices to check the claimed equivalence fiberwise.

Given $b \in \cB$ over $t \in \cT^\op$, we may replace $\cB$ by $\cB^{b/}$ and $\cT$ by $(\cT^{/t})^\op$, so it further suffices to check that $\rho$ induces an equivalence upon passage to cocartesian sections. But this is the map
\[ \Fun_{\cB}(\cX, \cY) \xto{\simeq} \Fun^{\cT\text{-}\cocart}_{/\cB,\cT}(\cX,\cY). \]
Indeed, for any $\cT$-cartesian fibration $\cX$ the cocartesian sections of $\widetilde{\Fun}_{\cB,\cT}(\cX,\cY)$ is the full subcategory of $\Fun_{/\cB,\cT}(\cX,\cY)$ spanned by those $\cT$-functors $F: \cX \to \cY$ over $\cB$ that carry fiberwise cartesian edges to fiberwise cocartesian edges, but in our case $F$ equivalently preserves cocartesian edges.
\end{proof}

\subsection{Application: parametrized (co)limits in section \texorpdfstring{$\cT$-$\infty$}{T-infinity-}-categories}

We next use the $\cT$-pairing construction to analyze $\cT$-limits and $\cT$-colimits in a $\cT$-$\infty$-category of sections. Actually, we work in somewhat greater generality: given a $\cT$-cocartesian fibration $\cX \to \cB$ and a $\cT$-fibration $\cC \to \cB$, we will study $\cT$-(co)limits in $\underline{\Fun}_{/\cB,\cT}(\cC, \cX)$ (\cref{thm:ParametrizedLimitsAndColimitsInSectionCategories}). First, we introduce some terminology concerning relative adjunctions, extending \cite[Def.~8.3]{Exp2}.

\begin{definition} \label{def:RelativeParametrizedAdjunction}
Let $\cX, \cY \to \cB$ be $\cT$-fibrations and let
\[ \adjunct{F}{\cX}{\cY}{G} \]
be a relative adjunction with respect to the structure maps to $\cB$ (in the sense of \cite[Def.~7.3.2.2]{HA}). We then say that $F \dashv G$ is a \emph{$\cB$-relative $\cT$-adjunction} if $F$ and $G$ are $\cT$-functors.
\end{definition}

\begin{remark}[Stability under base-change] \label{rem:BaseChangeRelativeParametrizedAdjunction}
Suppose
$$\adjunct{F}{\cX}{\cY}{G}$$
is a $\cB$-relative $\cT$-adjunction and $\phi: \cC \to \cB$ is a $\cT$-functor. By \cite[Prop.~7.3.2.5]{HA}, the pullback
$$\adjunct{F_{\cC}}{\cX \times_{\cB} \cC}{\cY \times_{\cB} \cC}{G_{\cC}}$$
is then a $\cC$-relative $\cT$-adjunction.
\end{remark}

The following lemma illustrates the basic asymmetry between $\cB$-relative left and right $\cT$-adjoints that should already be familiar from the theory of relative adjunctions (compare \cite[Prop.~7.3.2.6]{HA} versus \cite[Prop.~7.3.2.11]{HA}.)

\begin{lemma} \label{lem:RelativeAdjointExistence}
Suppose $\cX, \cY \to \cB$ are $\cT$-cocartesian fibrations. 
\begin{enumerate}
\item Let $F: \cX \to \cY$ be a morphism of $\cT$-cocartesian fibrations over $\cB$. Then $F$ admits a $\cB$-relative right $\cT$-adjoint $R: \cY \to \cX$ if and only if for all $b \in \cB_t$, the parametrized fiber $F_{\underline{b}}: \cX_{\underline{b}} \to \cY_{\underline{b}}$ admits a right $\cT^{/t}$-adjoint $R_{\underline{b}}$.
\item Let $F: \cX \to \cY$ be a $\cT$-functor over $\cB$. Then $F$ admits a $\cB$-relative left $\cT$-adjoint $L: \cY \to \cX$ if and only if for all $b \in \cB_t$, the parametrized fiber $F_{\underline{b}}: \cX_{\underline{b}} \to \cY_{\underline{b}}$ admits a left $\cT^{/t}$-adjoint $L_{\underline{b}}$, and for all fiberwise morphisms $f: b \to b'$ in $\cB_t$, the natural transformation
\[ \begin{tikzcd}
\cY_b \ar{r}{L_b} \ar{d}[swap]{f_!} & \cX_b \ar{d}{f_!} \\
\cY_{b'} \ar{r}[swap]{L_{b'}} \ar[Rightarrow, shorten <=10pt, shorten >=10pt]{ru} & \cX_{b'}
\end{tikzcd} \quad \text{adjoint to} \quad
\begin{tikzcd}
\cY_b \ar{d}[swap]{f_!} \ar[Rightarrow, shorten <=10pt, shorten >=10pt]{rd} & \cX_b \ar{l}[swap]{F_b} \ar{d}{f_!} \\
\cY_{b'} & \cX_{b'} \ar{l}{F_{b'}}
\end{tikzcd} \]
is an equivalence. Moreover, in this case $L$ is a morphism of $\cT$-cocartesian fibrations over $\cB$.
\end{enumerate}
\end{lemma}
\begin{proof} (1): By the opposite of \cite[Prop.~7.3.2.6]{HA}, $F$ admits a $\cB$-relative right adjoint $R$ if and only if for all $b \in \cB$, $F_b$ admits a right adjoint $R_b$. But we are then reduced to showing that $R$ is in addition a $\cT$-functor if and only if $R_{\underline{b}}$ is in addition a $\cT^{/t}$-functor for all $b \in \cB$, which is clear.

\noindent (2): Since every morphism in $\cB$ factors as the composite of a cocartesian edge and a fiberwise morphism, the claim follows directly from \cite[Prop.~7.3.2.11]{HA}.
\end{proof}

We may now state the main result of this subsection. Note that the cases of parametrized limits and colimits involve different hypotheses (as should be familiar from the theory of limits and colimits in $\infty$-categories of algebras).

\begin{theorem} \label{thm:ParametrizedLimitsAndColimitsInSectionCategories}
Let $p: \cX \to \cB$ be a $\cT$-cocartesian fibration, let $\cK$ be a $\cT$-$\infty$-category, and let $r:\cC \to \cB$ be any $\cT$-fibration.
\begin{enumerate}[leftmargin=*]
\item There exists a $\cB$-relative constant $\cK$-indexed diagram $\cT$-functor
\[ \delta_{p}: \cX \to \widetilde{\Fun}_{\cB,\cT}(\cB \times_{\cT^\op} \cK, \cX), \]
which is a morphism of $\cT$-cocartesian fibrations over $\cB$, such that the constant $\cK$-indexed diagram $\cT$-functor
\[ \delta_{r,p}: \underline{\Fun}_{/\cB, \cT}(\cC, \cX) \to \underline{\Fun}_{\cT}(\cK, \underline{\Fun}_{/\cB, \cT}(\cC, \cX)) \]
is given by $\underline{\Fun}_{/\cB, \cT}(\cC,\delta_p)$.
\item Suppose that for every $b \in \cB_t$, $\cX_{\underline{b}}$ admits all $\cK_{\underline{t}}$-indexed $\cT^{/t}$-limits. Then $\delta_p$ admits a $\cB$-relative $\cT$-right adjoint
\[ \mathrm{lim}^{\cB,\cT}: \widetilde{\Fun}_{\cB,\cT}(\cB \times_{\cT^\op} \cK, \cX) \to \cX \]
which for all $b \in \cB_t$ restricts to the $\cT^{/t}$-limit $\cT^{/t}$-functor
$$\mathrm{lim}^{\cT^{/t}}: \underline{\Fun}_{\cT^{/t}}(\cK_{\underline{t}}, \cX_{\underline{b}}) \to \cX_{\underline{b}}.$$
Consequently, $\delta_{r,p}$ admits a $\cT$-right adjoint $\lim^{\cT}$ given by $\underline{\Fun}_{/\cB, \cT}(\cC, \mathrm{lim}^{\cB,\cT})$.
\item A $\cT^{/t}$-functor
\[ \overline{f}: \cK_{\underline{t}}^{\underline{\lhd}} \to \underline{\Fun}_{/\cB,\cT}(\cC, \cX)_{\underline{t}} \]
is a $\cT^{/t}$-limit diagram, resp. $f$ admits a $\cT^{/t}$-limit, if for all $\alpha: s \to t$ in $\cT$ and $c \in \cC_{s}$ over $b \in \cB_{s}$, the composite $\cT^{/s}$-functor
\[ \overline{f_c}: \cK_{\underline{s}}^{\underline{\lhd}} \xto{f_{\underline{\alpha}}} \underline{\Fun}_{/\cB,\cT}(\cC, \cX)_{\underline{s}} \xto{\ev_{\underline{c}}} \underline{\Fun}_{/\cB_{\underline{s}},\cT^{/s}}(\underline{c}, \cX_{\underline{s}}) \simeq \cX_{\underline{b}} \]
is a $\cT^{/s}$-limit diagram, resp. $f_c$ admits a $\cT^{/s}$-limit.
\item Suppose that for every $b \in \cB_t$, $\cX_{\underline{b}}$ admits all $\cK_{\underline{t}}$-indexed $\cT^{/t}$-colimits, and for every morphism $f:b \to b'$ in $\cB_t$, the pushforward $\cT^{/t}$-functor $f_!: \cX_{\underline{b}} \to \cX_{\underline{b'}}$ preserves $\cK_{\underline{t}}$-indexed $\cT^{/t}$-colimits. Then $\delta_p$ admits a $\cB$-relative $\cT$-left adjoint
\[ \colim^{\cB,\cT}: \widetilde{\Fun}_{\cB,\cT}(\cB \times_{\cT^\op} \cK, \cX) \to \cX \]
which for all $b \in \cB_t$ restricts to the $\cT^{/t}$-colimit $\cT^{/t}$-functor
\[ \colim^{\cT^{/t}}: \underline{\Fun}_{\cT^{/t}}(\cK_{\underline{t}}, \cX_{\underline{b}}) \to \cX_{\underline{b}}. \]
Consequently, $\delta_{r,p}$ admits a $\cT$-left adjoint $\colim^{\cT}$ given by $\underline{\Fun}_{/\cB, \cT}(\cC, \colim^{\cB,\cT})$.
\item A $\cT^{/t}$-functor
\[ \overline{f}: \cK_{\underline{t}}^{\underline{\rhd}} \to \underline{\Fun}_{/\cB,\cT}(\cC, \cX)_{\underline{t}} \]
is a $\cT^{/t}$-colimit diagram, resp. $f$ admits a $\cT^{/t}$-colimit, if for all $\alpha: s \to t$ in $\cT$ and $c \in \cC_{s}$ over $b \in \cB_{s}$, the composite $\cT^{/s}$-functor
\[ \overline{f_c}: \cK_{\underline{s}}^{\underline{\rhd}} \xto{f_{\underline{\alpha}}} \underline{\Fun}_{/\cB,\cT}(\cC, \cX)_{\underline{s}} \xto{\ev_{\underline{c}}} \underline{\Fun}_{/\cB_{\underline{s}},\cT^{/s}}(\underline{c}, \cX_{\underline{s}}) \simeq \cX_{\underline{b}} \]
is a $\cT^{/s}$-colimit diagram (resp. $f_c$ admits a $\cT^{/s}$-colimit), and for all $c \to c' \in \cC_{s}$ over $g: b \to b' \in \cB_s$, $g_! \circ \overline{f_c}$ is a $\cT^{/s}$-colimit diagram (resp. $g_!$ preserves the $\cT^{/s}$-colimit of $f_c$).
\end{enumerate}
\end{theorem}
\begin{proof}
(1): Using \cref{prop:PairingConstructionReducesToFunctorInternalHomWhenSourceConstant}, we may define $\delta_p$ as adjoint to the projection $\cX \times_{\cT^\op} \cK \to \cX$, since this is a morphism of $\cT$-cocartesian fibrations over $\cB$, and by construction this has the indicated property. 

\noindent (2): Under our assumption, the existence of $\lim^{\cB,\cT}$ follows immediately from \cref{lem:RelativeAdjointExistence}(1). For the consequence, note that we have an equivalence of $\cT$-$\infty$-categories
\begin{align*}
\underline{\Fun}_{\cT}(\cK, \underline{\Fun}_{/\cB, \cT}(\cC, \cX)) & \simeq \underline{\Fun}_{/\cB,\cT}(\cC \times_{\cT^\op} \cK, \cX) \\
& \simeq \underline{\Fun}_{/\cB,\cT}(\cC, \widetilde{\Fun}_{\cB,\cT}(\cB \times_{\cT^\op} \cK, \cX))
\end{align*}
where the first equivalence holds by the universal property of $\underline{\Fun}_{\cT}(-,-)$ and the definition of $\underline{\Fun}_{/\cB, \cT}(\cC, \cX)$ as a pullback, and the second equivalence holds by \cref{thm:PairingConstructionUniversalProperty}. It thus suffices to show that $\underline{\Fun}_{/\cB,\cT}(\cC,-)$ covariantly transforms $\cB$-relative $\cT$-adjunctions into $\cT$-adjunctions. For this, by \cref{rem:BaseChangeRelativeParametrizedAdjunction} we may suppose that $\cC = \cB$ without loss of generality, in which case the assertion is \cite[Cor.~8.5]{Exp2}.

\noindent (3): This follows similarly to (2), but where we now consider how $\underline{\Fun}_{/\cB, \cT}(\cC,-)$ transforms $\lim^{\cB,\cT}$ as a partially defined $\cB$-relative right $\cT$-adjoint (under the given hypotheses on $\overline{f_c}$) to $\lim^{\cT}$ as a partially defined right $\cT$-adjoint.

\noindent (4) and (5): These are proven as for (2) and (3) but using \cref{lem:RelativeAdjointExistence}(2) instead.
\end{proof}

\begin{corollary} \label{cor:SectionCategoryAdmitsClassesOfLimits}
Let $\cX \to \cB$ be a $\cT$-cocartesian fibration and let $\cC \to \cB$ be a $\cT$-fibration. Let $\CMcal{K} = \{ \CMcal{K}_t : t \in \cT \} $ be a collection of classes $\CMcal{K}_t$ of small $\cT^{/t}$-$\infty$-categories closed with respect to base-change in $\cT$.
\begin{enumerate}
\item Suppose that for all $b \in \cB_{t}$, $\cX_{\underline{b}}$ admits all $\CMcal{K}_t$-indexed $\cT^{/t}$-limits. Then $\underline{\Fun}_{/\cB,\cT}(\cC, \cX)$ strongly admits all $\CMcal{K}$-indexed $\cT$-limits.
\item Suppose that for all $b \in \cB_{t}$, $\cX_{\underline{b}}$ admits all $\CMcal{K}_t$-indexed $\cT^{/t}$-colimits, and for all $g: b \to b' \in \cB_{t}$, the pushforward $\cT^{/t}$-functor $g_!: \cX_{\underline{b}} \to \cX_{\underline{b'}}$ preserves all $\CMcal{K}_t$-indexed $\cT^{/t}$-colimits. Then $\underline{\Fun}_{/\cB,\cT}(\cC, \cX)$ strongly admits all $\CMcal{K}$-indexed $\cT$-colimits.
\end{enumerate}
\end{corollary}
\begin{proof} We show how to deduce (1) from \cref{thm:ParametrizedLimitsAndColimitsInSectionCategories}, the proof of (2) being similar. For this, the only additional point to note is that for any $\cK \in \CMcal{K}_t$, by base-change of the given data to lie over $\cT^{/t}$ we may apply \cref{thm:ParametrizedLimitsAndColimitsInSectionCategories} under our hypotheses to show that $\underline{\Fun}_{/\cB,\cT}(\cC, \cX)_{\underline{t}}$ admits all $\cK$-indexed $\cT^{/t}$-limits.
\end{proof}

\section{Relative parametrized colimits}


In this section and the next we work towards the proof of \cref{thmx:Kan}.

\begin{dfn} \label{dfn:relativeColimit} Suppose we have a commutative diagram of $\cT$-$\infty$-categories
\[ \begin{tikzcd}[row sep=2em, column sep=2em]
\cK \ar{r}{p} \ar[hookrightarrow]{d}{i} & \cC \ar{d}{\pi} \\
\cK^{\underline{\rhd}} \ar{r}{\overline{q}} \ar{ru}{\overline{p}} & \cB
\end{tikzcd} \]
in which $\pi$ is a $\cT$-fibration. Let $q = \pi p = \overline{q} i$. We say that $\overline{p}$ is a \emph{weak $\pi$-$\cT$-colimit diagram} if the $\cT$-functor
\[ \ast_{\cT} \to \ast_{\cT} \times_{\sigma_{\overline{q}}, \cB^{(q,\cT)/}} \cC^{(p,\cT)/}  \]
induced by $\sigma_{\overline{p}}$ is a $\cT$-initial object.

We say that $\overline{p}$ is a \emph{$\pi$-$\cT$-colimit diagram} if the $\cT$-functor
\[ \ast_{\cT} \to \cB^{(\overline{q},\cT)/} \times_{\cB^{(q,\cT)/}} \cC^{(p,\cT)/} \]
induced by $\sigma_{\overline{p}}$ is a $\cT$-initial object. (Here, the projection to the first factor is induced by $\sigma_{\overline{q}'}$ for $\overline{q}'$ given by the composite
\begin{tikzcd}
\cK \star_{\cT^\op} (\Delta^1 \times \cT^\op) \ar{r}{\id \star \mathrm{const}} & \cK^{\underline{\rhd}} \ar{r}{\overline{q}} & \cB.)
\end{tikzcd}

\end{dfn}

\begin{example}
For $\cT = \Delta^0$ and $\cK = \Delta^0$, weak $\pi$-colimit diagrams are locally $\pi$-cocartesian edges, whereas $\pi$-colimit diagrams are $\pi$-cocartesian edges.
\end{example}

\begin{rem} For \cref{dfn:relativeColimit}, $\overline{p}$ is a $\pi$-$\cT$-colimit diagram if and only if the $\cT$-functor
\[ \cC^{(\overline{p},\cT)/} \to \cB^{(\overline{q},\cT)/} \times_{\cB^{(q,\cT)/}} \cC^{(p,\cT)/}  \]
is an equivalence of $\cT$-$\infty$-categories, or equivalently the commutative square
\[ \begin{tikzcd}[row sep=2em, column sep=2em]
\cC^{(\overline{p},\cT)/} \ar{r} \ar{d} & \cC^{(p,\cT)/} \ar{d} \\
\cB^{(\overline{q},\cT)/} \ar{r} & \cB^{(q,\cT)/}
\end{tikzcd} \]
is a homotopy pullback square (using \cref{lm:sliceCategoryFibrationProperties}(1)). This is ultimately because for any $\cT$-category $\cE$, a $\cT$-functor $\sigma: \ast_{\cT} \to \cE$ is a $\cT$-initial object if and only if $\cE^{(\sigma,\cT)/} \to \cE$ is an equivalence.
\end{rem}

We now collect a few lemmas that will feature in the proof of our main result (\cref{prp-app:RelativeColimitExistence}) on the existence of $\pi$-$\cT$-colimits. We first state a parametrized analogue of \cite[Prop.~4.2.1.6]{HTT}.

\begin{lem} \label{lm:sliceCategoryFibrationProperties} Suppose we have a commutative diagram of $\cT$-$\infty$-categories
\[ \begin{tikzcd}[row sep=2em, column sep=2em]
\cK \ar{r}{p} \ar[hookrightarrow]{d}{i} & \cC \ar{d}{\pi} \\
\cL \ar{r}{\overline{q}} \ar{ru}{\overline{p}} & \cB.
\end{tikzcd} \]
in which $i$ is a monomorphism. Let $q = \pi \circ p = \overline{q} \circ i$ and let
\[ \psi: \cC^{(\overline{p},\cT)/} \to \cC^{(p,\cT)/} \times_{\cB^{(q,\cT)/}} \cB^{(\overline{q},\cT)/} \]
denote the induced $\cT$-functor.
\begin{enumerate}
 \item If $\pi$ is a categorical fibration, then
 \[ \phi: \underline{\Fun}_{\cT}(\cL, \cC) \to \underline{\Fun}_{\cT}(\cK, \cC) \times_{\underline{\Fun}_{\cT}(\cK, \cB)} \underline{\Fun}_{\cT}(\cL, \cB) \]
 is a categorical fibration, and $\psi$ is a left fibration.

\item Suppose that $\pi$ is a $\cT$-cocartesian fibration and let $M_C$ denote the $\pi$-cocartesian edges in $\cC$. Suppose that we have fiberwise markings $\{ (M_K)_t \}$ on $\cK$ and $\{ (M_L)_t \}$ on $\cL$ that are stable under base-change (i.e., for all $f: s \to t \in \cT$, $f^\ast (M_K)_t \subset (M_K)_s$). Let $M_K$ be the minimal subset of the edges on $\cK$ closed under composition that contains the cocartesian edges and $\{ (M_K)_t \}$, and similarly define $M_L$. Let $\underline{\Fun}_{\cT}( (\cL, M_L), (\cC, M_C))$, etc. be the full $\cT$-subcategories spanned by those $\cT^{/t}$-functors that preserve the additional markings, and let
\[ \phi': \underline{\Fun}_{\cT}((\cL, M_L), (\cC, M_C)) \to \underline{\Fun}_{\cT}((\cK, M_K), (\cC, M_C)) \times_{\underline{\Fun}_{\cT}(\cK, \cB)} \underline{\Fun}_{\cT}(\cL, \cB) \]
be the restriction of $\phi$. Then if $i: (\cK, M_K) \to (\cL, M_L)$ is a cocartesian equivalence in $\sSet^+_{/\cT^\op}$, $\phi'$ is a trivial fibration. Moreover, if $\overline{p}$ then sends $M_L$ into $M_C$, $\psi$ is a trivial fibration.
\end{enumerate}
\end{lem}
\begin{proof} (1): It suffices to show that $\phi$ is a fibration in $\sSet^+_{/\cT^\op}$ where we mark the cocartesian edges. This follows as in the proof of \cite[Lem.~3.5(1)]{Exp2}. Considering this categorical fibration for both $i$ and $i^{\underline{\rhd}}$ then shows $\psi$ is a categorical fibration by a base-change argument. Since a functor between left fibrations over a common base that is a categorical fibration is necessarily a left fibration, to then show that $\psi$ is a left fibration it suffices to show that $\cC^{(f,\cT)/} \to \cC$ is a left fibration for any $\cT$-functor $f: \cJ \to \cC$. But this is a consequence of the $\cT$-functor
$$\underline{\Fun}_{\cT}(\cJ^{\underline{\rhd}}, \cC) \to \underline{\Fun}_{\cT}(\cJ, \cC) \times_{\cT^\op} \cC$$
being a $\cT$-bifibration (cf. \cite[Ex.~7.10]{Exp2} and \cite[Rem.~2.4.7.4]{HTT} for the stronger conclusion that $\cC^{(f,\cT)/} \to \cC$ is a left fibration and not just a $\cT$-cocartesian fibration).

(2): Since $\phi'$ is a categorical fibration by (1), it suffices to show that $\phi'_t$ is an equivalence for all $t \in \cT$. After replacing $\cT^{/t}$ with $\cT$, we thus reduce to checking that
\[ \widehat{\phi}': \Fun_{\cT}((\cL, M_L), (\cC, M_C)) \to \Fun_{\cT}((\cK, M_K), (\cC, M_C)) \times_{\Fun_{\cT}(\cK, \cB)} \Fun_{\cT}(\cL, \cB) \]
is a trivial fibration. Let $A \to B$ be any cofibration of simplicial sets. The relevant lifting problem then transposes to
\[ \begin{tikzcd}
A^{\flat} \times (\cL, M_K) \bigcup_{A^{\flat} \times (\cK, M_K)} B^{\flat} \times (\cK, M_K) \ar{r} \ar{d} & (\cC, M_C) \ar{d} \\
B^{\flat} \times (\cL, M_L) \ar{r} \ar[dotted]{ur} & \cB^{\sharp}.
\end{tikzcd} \]
Now because trivial cofibrations in the cocartesian model structure on $\sSet^+_{/\cT^\op}$ are stable under taking pushout-products with arbitrary cofibrations in $\sSet^+$ \cite[Cor.~3.1.4.3]{HTT}, there exists a dotted lift. The assertion about $\psi$ then follows as in (1) by a base-change argument.
\end{proof}

We have an elementary observation about lifting $\cT$-initial objects along a $\cT$-cocartesian fibration.

\begin{lem} \label{lem:InitialObjectsInCocartesianFibration} Let $\pi: \cC \to \cB$ be a $\cT$-cocartesian fibration and suppose that $\sigma: \cT^\op \to \cB$ is a $\cT$-initial object. Let $\widetilde{\sigma}: \cT^\op \to \cC$ be a $\cT$-functor lift of $\sigma$. Suppose that:
\begin{enumerate}
\item $\widetilde{\sigma}$ is a $\cT$-initial object in $\cT^\op \times_{\sigma, \cB} \cC$.
\item For all $t \in \cT$ and morphisms $f: \sigma(t) \to y \in \cB_t$, the pushforward $\cT^{/t}$-functor $f_!: \cC_{\underline{\sigma(t)}} \to \cC_{\underline{y}}$ preserves $\cT^{/t}$-initial objects.
\end{enumerate}
Then $\widetilde{\sigma}$ is a $\cT$-initial object.
\end{lem}
\begin{proof} We may check for all $t \in \cT$ that $\widetilde{\sigma}(t)$ is an initial object in $\cC_t$, using the known assertion when $\cT = \Delta^0$.
\end{proof}

We can then bootstrap from \cref{lem:InitialObjectsInCocartesianFibration} to understand relative $\cT$-colimits originating from $\cT$-colimits in the parametrized fibers.

\begin{lem} \label{lem:RelativeColimitsConcentratedInFiber} Suppose we have a commutative diagram of $\cT$-$\infty$-categories
\[ \begin{tikzcd}[row sep=2em, column sep=2em]
\cK \ar{d} \ar{r}{p_0} \ar[bend left]{rr}{p} & \cT^\op \times_{\cB} \cC \ar{r} \ar{d} & \cC \ar{d}{\pi} \\
\cK^{\underline{\rhd}} \ar{r} \ar{ru}{\overline{p_0}} \ar{rru}[swap, near end]{\overline{p}} \ar[bend right]{rr}{\overline{q}} & \cT^\op \ar{r}{\sigma} & \cB
\end{tikzcd} \]
in which $\pi$ is a $\cT$-cocartesian fibration and $\overline{p_0}$ is a $\cT$-colimit diagram. Then $\overline{p}$ is a weak $\pi$-$\cT$-colimit diagram.

Suppose moreover that for every $t \in \cT$ and morphism $f: \sigma(t) \to y \in \cB_t$, the pushforward $\cT^{/t}$-functor $f_!: \cC_{\underline{\sigma(t)}} \to \cC_{\underline{y}}$ preserves the $\cK_{\underline{t}}$-indexed $\cT^{/t}$-colimit diagram given by $(\overline{p_0})_{\underline{t}}$. Then $\overline{p}$ is a $\pi$-$\cT^{/t}$-colimit diagram.
\end{lem}
\begin{proof} By \cref{lem:PullbackOfSliceCategories}, we have an equivalence $\cC^{(p_0,\cT)/} \simeq \cT^\op \times_{\cB^{(q,\cT)/}} \cC^{(p,\cT)/}$. Thus by definition, if $\overline{p_0}$ is a $\cT$-colimit diagram, then $\overline{p}$ is a weak $\pi$-$\cC$-colimit diagram.

For the second claim, we note that the fiberwise cofinal $\cT$-functor $\cT^\op \subset \cK^{\underline{\rhd}}$ given by inclusion of the $\cT$-cone point induces an equivalence $\cB^{(\overline{q},\cT)/} \xto{\simeq} \cB^{(\sigma, \cT)/}$ by \cite[Thm.~6.7]{Exp2}. It thus suffices to check that the $\cT$-functor
\[ \begin{tikzcd}[row sep=2em, column sep=2em]
\cT^\op \ar{r}{\widetilde{\sigma}} \ar{rd}[swap]{\sigma} & \cB^{(\sigma,\cT)/} \times_{\cB^{(\pi p,\cT)/}} \cC^{(p,\cT)/} \ar{d}{\pi'} \\
& \cB^{(\sigma,\cT)/}
\end{tikzcd} \]
induced by $\overline{p}$ is a $\cT$-initial object (where we abuse notation and write $\sigma$ also for the $\cT$-initial object $\id_\sigma$ in $\cB^{(\sigma,\cT)/}$). By \cref{lm:sliceCategoryFibrationProperties}(1), $\pi'$ is a left fibration. Under the equivalence $\cB^{(\sigma,\cT)/} \simeq \cT^\op \times_{\sigma,\cB} \Ar(\cB)$ of \cref{obs-app:SmallerSliceEqv}, the objects of the base $\cB^{(\sigma,\cT)/}$ are equivalently given by pairs $(t \in \cT, f: \sigma(t) \to y)$. A cocartesian section of the parameterized fiber $\pi'_{\underline{f}}$ is determined up to equivalence by a commutative diagram of $\cT^{/t}$-$\infty$-categories
\[ \begin{tikzcd}[row sep=2em, column sep=2em]
\cK_{\underline{t}} \ar{d} \ar{r}{p'_0}  & \cC_{\underline{y}} \ar{d} \\
(\cK_{\underline{t}})^{\underline{\rhd}} \ar{r} \ar{ru}{\overline{p_0}'}  & (\cT^{/t})^\op
\end{tikzcd} \]
and is a $\cT^{/t}$-initial object if and only if $\overline{p_0}'$ is a $\cT^{/t}$-colimit diagram. It is now clear that under our hypotheses, \cref{lem:InitialObjectsInCocartesianFibration} applies to show that $\widetilde{\sigma}$ is a $\cT$-initial object.
\end{proof}

\begin{lem} \label{lem:PullbackOfSliceCategories} Suppose we have a homotopy pullback square of $\cT$-$\infty$-categories
\[ \begin{tikzcd}[row sep=2em, column sep=2em]
\cW \ar{r}{f} \ar{d}{g} & \cX \ar{d}{h} \\
\cY \ar{r}{k} & \cZ
\end{tikzcd} \]
and a $\cT$-functor $p: \cK \to \cW$. Then the commutative square of $\cT$-$\infty$-categories
\[ \begin{tikzcd}[row sep=2em, column sep=2em]
\cW^{(p,\cT)/} \ar{r} \ar{d} & \cX^{(f p,\cT)/} \ar{d} \\
\cY^{(g p,\cT)/} \ar{r} & \cZ^{(h f p, \cT)/}.
\end{tikzcd} \]
is a homotopy pullback square.
\end{lem}
\begin{proof} The proof is a straightforward diagram chase, starting from the known assertion that
\[ \underline{\Fun}_{\cT}(\cK,-): \CatT \to \CatT \]
preserves limits.
\end{proof}

Finally, we arrive at our main existence result for relative $\cT$-colimits.

\begin{prp} \label{prp-app:RelativeColimitExistence}
Suppose we have a commutative diagram of $\cT$-$\infty$-categories
\[ \begin{tikzcd}
\cK \ar{r}{p} \ar{d}[swap]{i} & \cC \ar{d}{\pi} \\
\cK^{\underline{\rhd}} \ar{r}{\overline{q}} \ar[dotted]{ru}{\overline{p}} & \cB
\end{tikzcd} \]
in which $\pi: \cC \to \cB$ is a $\cT$-cocartesian fibration. Let $\sigma = \overline{q}|_{\cT^\op}$. If for all $t \in \cT$, the parametrized fiber $\cC_{\underline{\sigma(t)}}$ admits $\cK_{\underline{t}}$-indexed $\cT^{/t}$-colimits, then there exists a filler $\overline{p}: \cK^{\underline{\rhd}} \to \cC$ which is a weak $\pi$-$\cT$-colimit diagram.

Moreover, suppose that for all morphisms $f: \sigma(t) \to y \in \cB_t$, the induced pushforward $\cT^{/t}$-functor $f_!: \cC_{\underline{\sigma(t)}} \to \cC_{\underline{y}}$ preserves $\cK_{\underline{t}}$-indexed $\cT^{/t}$-colimits. Then $\overline{p}$ is a $\pi$-$\cT$-colimit diagram.
\end{prp}
\begin{proof} We prove this by reducing to \cref{lem:RelativeColimitsConcentratedInFiber}. Let $M_C$ denote the $\pi$-cocartesian edges in $\cC$. First consider the diagram
\[ \begin{tikzcd}
\cK \times \{0\} \ar{rr}{p} \ar{d}{i_0} & &  \cC \ar{d}{\pi} \\
\cK \times \Delta^1 \ar{r}{f} \ar[dotted]{rru}{h} & \cK^{\underline{\rhd}} \ar{r}{\overline{q}} & \cB
\end{tikzcd} \]
where the map $f$ is adjoint to $(\cK = \cK, \cK \to \cT)$. Because $i_0: \leftnat{\cK} \times \{0\} \to \leftnat{\cK} \times (\Delta^1)^\sharp$ is left marked anodyne, the dotted map $h: \leftnat{\cK} \times (\Delta^1)^\sharp \to (\cC, M_C)$ exists. Consider the two commutative squares
\[ \begin{tikzcd}[row sep=2em, column sep=4em]
\cK \times \Delta^1 \ar{r}{h} \ar{d}{f} & \cC \ar{d}{\pi} \\
\cK^{\underline{\rhd}} \ar{r}{\overline{q}} & \cB
\end{tikzcd} \quad, \quad
\begin{tikzcd}[row sep=2em, column sep=4em]
\cK \ar{rr}{p' := h|_{\cK \times \{1\}}} \ar{d}{i} & &  \cC \ar{d}{\pi} \\
\cK^{\underline{\rhd}}  \ar{r} \ar[bend right]{rr}{\overline{q}'} & \cT^\op \ar{r}{\sigma} & \cB
\end{tikzcd} \]
We obtain a zig-zag
\[ \begin{tikzcd}
\cB^{(\overline{q},\cT)/} \times_{\cB^{(\pi p,\cT)/}} \cC^{(p,\cT)/} \ar{d} & \cB^{(\overline{q},\cT)/} \times_{\cB^{(\pi h,\cT)/}} \cC^{(h,\cT)/} \ar{r}{\phi} \ar{l}[swap]{\psi} \ar{d} & \cB^{(\overline{q}',\cT)/} \times_{\cB^{(\pi p',\cT)/}} \cC^{(p',\cT)/} \ar{d} \\
\cB^{(\overline{q},\cT)/} & \cB^{(\overline{q},\cT)/} \ar{r}{\chi} \ar{l}[swap]{=} & \cB^{(\overline{q}',\cT)/}
\end{tikzcd} \]
where all the maps are obvious (except possibly $\chi$, which is induced by precomposition by $\cK^{\underline{\rhd}} \to \cT^\op \to \cK^{\underline{\rhd}}$). We claim that $\psi$ and $\phi$ are equivalences. For $\psi$, by \cref{lm:sliceCategoryFibrationProperties}(2),
\[ \cC^{(h,\cT)/} \to \cB^{(\pi h,\cT)/} \times_{\cB^{(\pi p,\cT)/}} \cC^{(p,\cT)/} \]
is a trivial fibration. But $\psi$ is a pullback of this map, hence an equivalence. For $\phi$, note that the $\cT$-functors $\cK^{\underline{\rhd}} \to \cT^\op \to \cK^{\underline{\rhd}}$ and $\cK \times \{1\} \to \cK \times \Delta^1$ are both fiberwise cofinal. Hence by \cite[Thm.~6.7]{Exp2}, we deduce that $\phi$ is an equivalence.

Replacing $p$ and $\overline{q}$ by $p'$ and $\overline{q}'$, we find ourselves in the situation of \cref{lem:RelativeColimitsConcentratedInFiber}, which immediately applies given our hypotheses.
\end{proof}

\section{Relative parametrized left Kan extensions}


\begin{definition} \label{def-app:relativeKanExtension}
Suppose we have a commutative diagram of $\cT$-$\infty$-categories
\[ \begin{tikzcd}
\cC \ar{r}{F} \ar[hookrightarrow]{d}[swap]{i} & \cE \ar{d}{\pi} \\
\cD \ar{r} \ar{ru}{G} & \cB
\end{tikzcd} \]
in which $i$ is the inclusion of a full $\cT$-subcategory. Then we say that $G$ is a \emph{$\pi$-$\cT$-left Kan extension of $F$} if for every $x \in \cD_t$, the commutative diagram
\[ \begin{tikzcd}
\cC^{/\underline{x}} \ar{r}{F^x} \ar[hookrightarrow]{d} & \cE_{\underline{t}} \ar{d}{\pi_{\underline{t}}} \\
(\cC^{/\underline{x}})^{\underline{\rhd}} \ar{r} \ar{ru}{G^x} & \cB_{\underline{t}}
\end{tikzcd} \]
exhibits $G^x$ (defined in \cref{con-app:LKE}) as a $\pi_{\underline{t}}$-$\cT^{/t}$-colimit diagram. Here, the lower horizontal $\cT^{/t}$-functor is the composite
\[ (\cC^{/\underline{x}})^{\underline{\rhd}} \to (\cD^{/\underline{x}})^{\underline{\rhd}} \xto{\theta_x} \cD_{\underline{t}} \to \cB_{\underline{t}}. \]
We also say that $G$ is a \emph{weak $\pi$-$\cT$-left Kan extension of $F$} if in the pulled-back diagram
\[ \begin{tikzcd}
\cC \ar{r}{F'} \ar{d}[swap]{i} & \cE \times_{\cB} \cD \ar{d}{\pi'} \\
\cD \ar{r}{=} \ar{ru}{G'} & \cD,
\end{tikzcd}  \]
$G'$ is a $\pi'$-$\cT$-left Kan extension of $F'$.
\end{definition}


We may now prove our main existence result on relative $\cT$-left Kan extensions, from which \cref{thmx:Kan} is an immediate corollary.

\begin{theorem} \label{thm-app:RelativeLKEexistence}
Let $\pi: \cE \to \cB$ be a $\cT$-cocartesian fibration\footnote{Because of this assumption, our theorem is slightly weaker than \cite[Thm.~4.3.2.15]{HTT} in the case where $\cT = \Delta^0$.}, let $\rho: \cD \to \cB$ be a $\cT$-functor, and let $i: \cC \subset \cD$ be the inclusion of a full $\cT$-subcategory.
\begin{enumerate}
\item Let $F: \cC \to \cE$ be a $\cT$-functor over $\cB$ and suppose that for all $x \in \cD_t$, $F^x: \cC^{/\underline{x}} \to \cE_{\underline{t}}$ admits a $\pi_{\underline{t}}$-$\cT^{/t}$-colimit. Then $F$ admits an essentially unique $\pi$-$\cT$-left Kan extension $G: \cD \to \cE$.
\item The partial $\cT$-left adjoint $i_!$ to the restriction $\cT$-functor
\[ i^*: \underline{\Fun}_{/\cB, \cT}(\cD, \cE) \to \underline{\Fun}_{/\cB, \cT}(\cC, \cE) \]
is defined on all those $F: \cC_{\underline{t}} \to \cE_{\underline{t}}$ that admit a weak $\pi_{\underline{t}}$-$\cT^{/t}$-left Kan extension $G$, in which case $i_! F \simeq G$.
\end{enumerate}
\end{theorem}
\begin{proof}
The overarching strategy is the same as in the proof of \cref{thm-app:LKEexistence} given in \cite[\S 10]{Exp2}. The key idea is to factor $i$ through the free $\cT$-cocartesian fibration as
\[ \cC \xto{\iota} \cC \times_{\cD} \Ar_{\cT}(\cD) \xto{\ev_1} \cD. \]
(1): Choose a section $\xi$ of the trivial fibration $\Ar^{\cocart}_{\cT}(\cE) \to \cE \times_{\cB} \Ar_{\cT}(\cB)$ that restricts to the identity section on $\cE$ and let $F'$ be the composite
\[ \cC \times_{\cD} \Ar_{\cT}(\cD) \xto{F \times \Ar_{\cT}(\rho) } \cE \times_{\cB} \Ar_{\cT}(\cB) \xto{\xi} \Ar^{\cocart}_{\cT}(\cE) \xto{\ev_1} \cE. \]
We then have a commutative diagram
\[ \begin{tikzcd}
\cC \times_{\cD} \Ar_{\cT}(\cD) \ar{r}{F'} \ar{d}[swap]{j} & \cE \ar{d}{\pi} \\
(\cC \times_{\cD} \Ar_{\cT}(\cD)) \star_{\cD} \cD \ar{r}  & \cB
\end{tikzcd} \]
where $F = F'|_{\cC}$. Note that if $\xi: \cM \to \cD$ is any $\cT$-cocartesian fibration, then for any $x \in \cD_{t}$ we have a pullback square
\[ \begin{tikzcd}
\cM_{\underline{x}} \ar{r} \ar{d} & \cM^{/\underline{x}} \coloneq \cM \times_{\cM \star_{\cD} \cD} \Ar_{\cT}(\cM \star_{\cD} \cD) \times_{\cM \star_{\cD} \cD} \underline{x} \ar{d} \\
\underline{x} \ar{r}{\iota_x} & \cD^{/\underline{x}} \coloneq \Ar_{\cT}(\cD) \times_{\cD} \underline{x}
\end{tikzcd} \]
where the righthand vertical functor is a cocartesian fibration (induced by $\cM \star_{\cD} \cD \to \cD$). Since cocartesian fibrations are smooth \cite[Prop.~4.1.2.15]{HTT} and $\iota_x$ is fiberwise cofinal, it follows that $\cM_{\underline{x}} \to \cM^{/\underline{x}}$ is fiberwise cofinal (with respect to the base $(\cT^{/t})^\op$). In our situation, $\cM = \cC \times_{\cD} \Ar_{\cT}(\cD)$, $\cM_{\underline{x}} \cong \cC^{/\underline{x}}$, and $(F')^x$ restricts on $\cC^{/\underline{x}}$ to $F^x$. By the proof of \cref{prp-app:RelativeColimitExistence} together with \cite[Thm.~6.7]{Exp2}, we see that $F^x$ admits a $\pi$-$\cT$-colimit if and only if $(F')^x$ admits a $\pi$-$\cT$-colimit. We thereby reduce to the `$\cD$-parametrized' situation of a $\cT$-cocartesian fibration $\phi: \cM \to \cD$ and a commutative diagram
\[ \begin{tikzcd}
\cM \ar{r}{F} \ar{d}[swap]{j} & \cE \ar{d}{\pi} \\
\cM \star_{\cD} \cD \ar{r} \ar[dotted]{ru}{G} & \cB
\end{tikzcd} \]
in which $F$ sends $\phi$-cocartesian edges to $\pi$-cocartesian edges. Pulling back along $\rho$, we may also suppose that $\cD = \cB$, noting that the discrepancy between strong and weak $\pi$-$\cT$-left Kan extensions will lie only in the pointwise property of the eventual extension and does not feature in the constructive proof of existence.

We may solve the coherence problem of assembling the individual $\pi_{\underline{t}}$-$\cT^{/t}$-colimits together into a $\pi$-$\cT$-left Kan extension $G$ by a similar method to the proof of \cite[Thm.~9.15]{Exp2}. Consider \cite[Constr.~9.8]{Exp2} applied to $\phi$ and $F$; this yields
\[ \cW = \cE^{(\phi, F)/\cT} \coloneq \cD \times_{\widetilde{\Fun}_{\cD,\cT}(\cM, \cE)} \widetilde{\Fun}_{\cD,\cT}(\cM \star_{\cD} \cD, \cE), \]
such that for $x \in \cD_{t}$, if we let $F|_x: \cM_{\underline{x}} \to \cE_{\underline{x}}$ denote the $\cT^{/t}$-functor given by restriction, then the parametrized fiber $\cW_{\underline{x}}$ is equivalent to $(\cE_{\underline{x}})^{(F|_x, \cT^{/t})/}$ by \cref{cor:PairingFibers} (or \cite[Cor.~9.9]{Exp2}). Let $\cW' \subset \cW$ be the full $\cT$-subcategory spanned by the $\cT^{/t}$-colimit diagrams $(\cM_{\underline{x}})^{\underline{\rhd}} \to \cE_{\underline{x}}$, so that for all $x \in \cD$, the fiber $\cW'_{x}$ is that spanned by the initial objects in $\cW_x$, which are precisely weak $\pi_{\underline{t}}$-$\cT^{/t}$-colimit diagrams extending $F|_x$. Note that the use of \cite[Prop.~9.10]{Exp2} and \cite[Lem.~9.11]{Exp2} in the proof of \cite[Thm.~9.15]{Exp2} won't apply here since $\cE \to \cD$ isn't supposed to be a $\cT$-cartesian fibration. However, we may argue directly that $\cW' \to \cD$ is a trivial fibration by computing mapping spaces as follows:

\begin{itemize}
\item[($\ast$)] For any $\alpha: x \to y \in \cD_t$ and extensions $\overline{F|_{i}}$ of $F|_i$ over $(\cM_{\underline{i}})^{\underline{\rhd}}$, $i \in \{x,y\}$, we have that maps $\overline{F|_x} \to \overline{F|_y}$ in $\widetilde{\Fun}_{\cD,\cT}(\cM \star_{\cD} \cD, \cE)$ are defined by lax commutative squares of $\cT^{/t}$-functors
\[ \begin{tikzcd}
(\cM_{\underline{x}})^{\underline{\rhd}} \ar{r}{\overline{F|_x}} \ar{d}[swap]{\alpha_!} \ar[phantom]{rd}{\SWarrow}   & \cE_{\underline{x}} \ar{d}{\alpha_!} \\
(\cM_{\underline{y}})^{\underline{\rhd}} \ar{r}[swap]{\overline{F|_y}} & \cE_{\underline{y}},
\end{tikzcd} \]
and hence we have
\[ \Map_{\cW}(\overline{F|_x}, \overline{F|_y}) \simeq \Map(\alpha_! \overline{F|_x}, \overline{F|_y} \alpha_! ) \times_{\Map(\alpha_! F|_x, F|_y \alpha_!)} \{ \id\}. \]
\end{itemize}

By assumption, if $\overline{F|_x}$ is a $\cT^{/t}$-colimit diagram, then $\alpha_! \overline{F|_x}$ is as well. Therefore, $\Map_{\cW'}(\overline{F|_x}, \overline{F|_y})$ is contractible for all $\cT^{/t}$-colimit diagrams $\overline{F|_x}$ and $\overline{F|_y}$, and this suffices to show that $\cW' \to \cD$ is a trivial fibration since we already have compatibility of these initial objects with restriction in the base $\cT$. Furthermore, any section $\tau$ of this trivial fibration defines a relative left adjoint of $\cW \to \cD$ with respect to the base $\cD$.

Now by \cref{thm:PairingConstructionUniversalProperty}, applying $\underline{\Fun}_{/\cD,\cT}(\cD,-)$ to $\tau$ yields an extension $G: \cM \star_{\cD} \cD \to \cE$ of $F$ that is a $\cT$-initial object of
\[ \cT^\op \times_{\sigma_F, \underline{\Fun}_{/\cD,\cT}(\cM, \cE)} \underline{\Fun}_{/\cD,\cT}(\cM \star_{\cD} \cD, \cE). \]

Taking cocartesian sections, we then see that $G$ is an initial object in the space of such fillers and is in particular essentially unique.

(2): Let $\phi': (\cC \times_{\cD} \Ar_{\cT}(\cD)) \star_{\cD} \cD \to \cD$ be the structure map. Factor $i^\ast$ as
\[ \begin{tikzcd}
\underline{\Fun}_{/\cB,\cT}(\cD, \cE) \ar{r}{(\phi')^\ast} & \underline{\Fun}_{/\cB,\cT}((\cC \times_{\cD} \Ar_{\cT}(\cD)) \star_{\cD} \cD, \cE) \ar{r}{j^\ast} & \underline{\Fun}_{/\cB,\cT}(\cC \times_{\cD} \Ar_{\cT}(\cD), \cE) \ar{r}{\iota^\ast} & \underline{\Fun}_{/\cB,\cT}(\cC, \cE).
\end{tikzcd} \]
Then since $\phi'$ is $\cT$-left adjoint to $i_{\cD}$, $(\phi')^\ast$ has $\cT$-left adjoint $(i_{\cD})^\ast$. Also, by \cref{exm:freeCocartesianFibration} the procedure $F \mapsto F'$ of (1) defines a fully faithful $\cT$-left adjoint to $\iota^\ast$ with essential image spanned by those $\cT$-functors that send $\phi$-cocartesian edges to $\pi$-cocartesian edges. To conclude, we observe that in the proof of (1) we showed that the partial left $\cT$-adjoint $j_!$ is defined on $F'$ if it is obtained from $F$ satisfying the assumptions of (2).
\end{proof}


\section{More on the parametrized Yoneda embedding}

For a $\cT$-$\infty$-category $\cC$, let $\underline{\PShv}_\cT(\cC) \coloneqq \underline{\Fun}_\cT(\cC^{\vop}, \underline{\Spc}_\cT)$ be the $\cT$-$\infty$-category of $\cT$-presheaves and $j_\cT: \cC \to \underline{\PShv}_\cT(\cC)$ the $\cT$-Yoneda embedding \cite[\S 11]{Exp2}. In this section, we record a generalization (\cref{prop:ParamYonedaPreservesSlices}) of our earlier result that $j_\cT$ strongly preserves $\cT$-limits \cite[Cor.~11.10]{Exp2} as well as some basic facts concerning $\cT$-corepresentable $\cT$-left fibrations (\cref{lem:ParamRepresentableFibrations}) that mirror the discussion in \cite[\S 4.4.4]{HTT}. These results play a technical role in the remainder of the paper and so their proofs could be skipped on a first reading.

\begin{lemma} \label{lem:ParameterizedSlicesAsOrdinarySlices} Let $\cK$ and $\cC$ be $\cT$-$\infty$-categories. Then we have a homotopy pullback square
\[ \begin{tikzcd}
\underline{\Fun}_\cT(\cK^{\underline{\lhd}}, \cC) \ar{r} \ar{d} & \underline{\Fun}_\cT(\cK \times \Delta^1, \cC) \ar{d} \\
\cC \times_{\cT^{\op}} \underline{\Fun}_\cT(\cK, \cC) \ar{r} & \underline{\Fun}_\cT(\cK, \cC) \times_{\cT^{\op}} \underline{\Fun}_\cT(\cK, \cC).
\end{tikzcd} \]
Thus, for any $\cT$-functor $p: \cK \to \cC$ and $t \in \cT$, we have an equivalence
\[ (\cC^{/(p,\cT)})_t \simeq \cC_t \times_{\Fun_{\cT^{/t}}(\cK_{\underline{t}}, \cC_{\underline{t}})} \Fun_{\cT^{/t}}(\cK_{\underline{t}}, \cC_{\underline{t}})^{/\{ p_{\underline{t}}\}}  \]
of right fibrations over $\cC_t$.
\end{lemma}
\begin{proof} Consider the commutative square of $\cT$-$\infty$-categories
\[ \begin{tikzcd}
\cK \times \partial \Delta^1 \ar{r} \ar{d} & \cT^{\op} \bigsqcup \cK \ar{d} \\
\cK \times \Delta^1 \ar{r} & \cK^{\underline{\lhd}}
\end{tikzcd} \]
where the vertical maps are the inclusions and the horizontal maps are induced by the structure map $\cK \to \cT^{\op}$ and the identity on $\cK$. By application of \cite[Prop.~4.2.1.2]{HTT} fiberwise, this is a homotopy pushout square, and the first claim follows by transforming the pushout to a pullback under $\underline{\Fun}_\cT(-,\cC)$. For a $\cT$-functor $p: \cK \to \cC$, we thus obtain a commutative diagram of homotopy pullback squares
\[ \begin{tikzcd}
\cC^{/(p,\cT)} \ar{r} \ar{d} & \underline{\Fun}_\cT(\cK, \cC)^{/(\sigma_p,\cT)} \ar{r} \ar{d} & \underline{\Fun}_\cT(\cK \times \Delta^1, \cC) \ar{d} \\
\cC \ar{r}{\delta} & \underline{\Fun}_\cT(\cK, \cC) \ar{r}{(\id,\sigma_p)} & \underline{\Fun}_\cT(\cK, \cC) \times_{\cT^{\op}} \underline{\Fun}_\cT(\cK, \cC),
\end{tikzcd} \]
where $\sigma_p: \cT^\op \to \underline{\Fun}_\cT(\cK, \cC)$ selects $p$. To identify $(\cC^{/(p,\cT)})_t$, after replacing $p$ by $p_{\underline{t}}$ we may suppose that $\cT$ has a final object $\ast$. But then for any $\cT$-$\infty$-category $\cD$ and cocartesian section $\sigma: \cT^{\op} \to \cD$ that selects an object $x = \sigma(\ast) \in \cD_{\ast}$, we have that $(\cD^{/(\sigma,\cT)})_{\ast} \simeq (\cD_{\ast})^{/x}$ by \cite[Prop.~4.30]{Exp2}.
\end{proof}

The following lemma generalizes and supplies another proof of the fact that the Yoneda embedding preserves limits \cite[Prop.~5.1.3.2]{HTT}.

\begin{lemma} \label{lem:YonedaPreservesLimitsForSliceCategories} 
Let $p: \cK \to \cC$ be functor of small $\infty$-categories. Then the commutative square of $\infty$-categories
\[ \begin{tikzcd}
\cC^{/p} \ar{r} \ar{d} & \PShv(\cC)^{/jp} \ar{d} \\
\cC \ar{r}{j} & \PShv(\cC)
\end{tikzcd} \]
is a homotopy pullback square. Consequently, the functor $\varphi: \cC^{\op} \to \Spc$ classifying the right fibration $\cC^{/p} \to \cC$ is canonically equivalent to $\lim_\cK j p$.
\end{lemma}
\begin{proof} It suffices to show that the induced functor $\psi: \cC^{/p} \to \cC \times_{\PShv(\cC)} \PShv(\cC)^{/jp}$ is an equivalence of right fibrations over $\cC$ by checking that for all $x \in \cC$, $\psi_x$ is an equivalence. First note that by the $\cT = \Delta^0$ case of \cref{lem:ParameterizedSlicesAsOrdinarySlices}, we may identify $\psi_x$ with the map
$$\psi'_x: \Map_{\cC^\cK}(\delta x, p) \to \Map_{\PShv(\cC)^\cK}(\delta j(x), j p) \simeq  \Map_{\PShv(\cC)}(j(x), \lim_{\cK} jp) \simeq (\lim_{\cK} jp)(x)$$
induced by postcomposition by the Yoneda embedding. Using the end formula for mapping spaces in $\Fun(\cK,\cC)$ \cite[Prop.~2.3]{GlasmanHodge}, we have an equivalence
\[ \Map_{\cC^\cK}(\delta x, p) \simeq \int_{\cK} \Map_\cC(x, p(-)) \coloneqq \lim_{\Tw(\cK)} \Map_\cC(x, p(-)).  \]
where the limit is taken over the functor 
$$\Tw(\cK) \xrightarrow{\ev_1} \cK \xrightarrow{p} \cC \xrightarrow{\Map_\cC(x,-)} \Spc.$$
Under this identification, $\psi'_x$ is induced by restriction along $\ev_1$ (using the contravariant functoriality of limits in the diagram). The claim then follows from \cref{lem:TwistedArrowFinalFunctor}.

Since representable right fibrations are classified by the corresponding representable functor \cite[Prop.~4.4.4.5]{HTT}, we then have that $\varphi$ is equivalent to the composition
\[ \cC^{\op} \xrightarrow{j^{\op}} \PShv(\cC)^{\op} \xrightarrow{\Map_{\PShv(\cC)}(-,\lim_\cK j p)} \Spc. \]
An argument with the Yoneda lemma then shows $\varphi$ is in turn equivalent to $\lim_\cK j p$ (in more detail, see \cref{rem:YonedaLemmaExtension}). 
\end{proof}

\begin{lemma} \label{lem:TwistedArrowFinalFunctor} Let $\cK$ be an $\infty$-category. Then the source and target functors $$\ev_0, \ev_1: \Tw(\cK) \to \cK$$ are right cofinal.
\end{lemma}
\begin{proof} We verify the hypotheses of Joyal's cofinality theorem \cite[Thm.~4.1.3.1]{HTT} for $\ev_1$ (in its opposite formulation). Let $y \in \cK$ and consider the commutative diagram of homotopy pullbacks
\[ \begin{tikzcd}
(\cK_{/y})^{\op} \ar{r}{\iota'} \ar{d} & \Tw(\cK) \times_\cK \cK_{/y} \ar{r} \ar{d} & \Tw(\cK) \ar{d}{\ev_1} \\
\{ y \} \ar{r}{\iota} & \cK_{/y} \ar{r} & \cK
\end{tikzcd} \]
Since $\ev_1$ is a left fibration, $\ev_1$ is a smooth map \cite[Prop.~4.1.2.15]{HTT}. Therefore, the pullback $\iota'$ of the cofinal inclusion $\iota$ along $\ev_1$ is again cofinal, so in particular a weak homotopy equivalence. We deduce that $\Tw(\cK) \times_\cK \cK_{/y}$ is weakly contractible, which proves the claim. The proof for $\ev_0$ is similar.
\end{proof}

\begin{remark}\label{rem:YonedaLemmaExtension} Let $q \in \PShv(\cC)$ be a presheaf. By \cite[Lem.~5.1.5.2]{HTT}, we have an equivalence $q(c) \simeq \Map_{\PShv(\cC)}(j(c),q)$ of spaces for all $c \in \cC$. We can promote this to an equivalence of presheaves
$$\Map_{\PShv(\cC)}(j(-),q) \simeq q(-)$$
as follows. Let $\pi: \cC^{/q} = \cC \times_{\PShv(\cC)} \PShv(\cC)^{/q} \to \cC$ denote the projection. We then have the sequence of equivalences in $\PShv(\cC)$
\begin{align*}
q(-) & \simeq \colim_{\cC^{/q}} \Map_\cC(-,\pi) \simeq \colim_{\cC^{/q}} \Map_{\PShv(\cC)}(j(-),j \pi) \\
 & \simeq \Map_{\PShv(\cC)}(j(-), \colim_{\cC^{/q}} j \pi) \simeq \Map_{\PShv(\cC)}(j(-), q),
\end{align*}
where we use that $\id_{\PShv(\cC)}$ is the left Kan extension of $j$ along itself \cite[Lem.~5.1.5.3]{HTT} for the first and last equivalences, $j$ is fully faithful \cite[Prop.~5.1.3.1]{HTT} for the second equivalence, and $j(c)$ corepresents evaluation at $c$ \cite[Lem.~5.1.5.2]{HTT} (and is hence completely compact) to show the third equivalence, where the map in question is the canonical colimit interchange map.

Now suppose $\cC$ is a $\cT$-$\infty$-category and let $q \in \PShv_\cT(\cC) \simeq \PShv(\cC^{\rm v})$ be a $\cT$-presheaf. We again have that the $\cT$-Yoneda embedding $j_\cT: \cC \to \underline{\PShv}_\cT(\cC)$, given fiberwise by $\cC_t \subset \PShv(\cC_t) \subset \PShv_{\cT^{/t}}(\cC) \simeq \PShv(\cC_{\underline{t}}^{\rm v})$, is $\cT$-fully faithful, $\id_{\underline{\PShv}_\cT(\cC)}$ is the $\cT$-left Kan extension of $j_\cT$ along itself \cite[Lem.~11.1]{Exp2}, and for any $c \in \cC_t$, $j_\cT(c)$ is completely compact as an object in $\PShv_{\cT^{/t}}(\cC_{\underline{t}}) \simeq \PShv(\cC_{\underline{t}}^{\rm v})$, hence $\cT^{/t}$-completely compact in $\underline{\PShv}_{\cT^{/t}}(\cC_{\underline{t}})$ since $\cT$-colimits in $\underline{\Spc}_\cT$ are computed as ordinary colimits under the correspondence of \cite[Prop.~5.5]{Exp2}. Repeating the above argument then shows that we have an equivalence of $\cT$-presheaves
\[ q(-) \simeq \underline{\Map}_{\underline{\PShv}_\cT(\cC)}(j_\cT(-),q): \cC^{\vop} \to \Spc, \]
where on the right we view $q$ as a cocartesian section of $\underline{\PShv}_\cT(\cC)$.
\end{remark}

\begin{proposition} \label{prop:ParamYonedaPreservesSlices} Let $p: \cK \to \cC$ be a $\cT$-functor of small $\cT$-$\infty$-categories. Then the commutative square of $\cT$-$\infty$-categories
\[ \begin{tikzcd}
\cC^{/(p,\cT)} \ar{r} \ar{d} & \underline{\PShv}_\cT(\cC)^{/(j_\cT p,\cT)} \ar{d} \\
\cC \ar{r}{j_\cT} & \underline{\PShv}_\cT(\cC) 
\end{tikzcd} \]
is a homotopy pullback square. Consequently, the functor $\cC^{\vop} \to \Spc$ classifying the left fibration $$(\cC^{/(p,\cT)})^{\vop} \simeq (\cC^\vop)^{(p^{\vop},\cT)/} \to \cC^{\vop}$$ is canonically equivalent to $\lim^\cT_\cK j_\cT p$.
\end{proposition}
\begin{proof} The square in question is a homotopy pullback if and only if for all $t \in \cT$, the square
\[ \begin{tikzcd}
(\cC^{/(p,\cT)})_t \simeq ( \cC_{\underline{t}}^{/(p_{\underline{t}},\cT^{/t})} )_{\id_t} \ar{r} \ar{d} & (\underline{\PShv}_\cT(\cC)^{/(j_\cT p,\cT)})_t \simeq ( \PShv_{\cT^{/t}}(\cC_{\underline{t}})^{/(p_{\underline{t}},\cT^{/t})} )_{\id_t} \ar{d} \\
\cC_t \simeq (\cC_{\underline{t}})_{\id_t} \ar{r} & \underline{\PShv}_\cT(\cC)_t \simeq \PShv_{\cT^{/t}}(\cC_{\underline{t}})
\end{tikzcd} \]
is a homotopy pullback of $\infty$-categories. Therefore, after replacing $\cT$ by $\cT^{/t}$, we may suppose that $\cT$ has a final object $\ast \in \cT$, and it suffices to check that the square of $\infty$-categories
\[ \begin{tikzcd}
(\cC^{/(p,\cT)})_{\ast} \ar{r} \ar{d} & (\underline{\PShv}_\cT(\cC)^{/(j_\cT p,\cT)})_{\ast} \ar{d} \\
\cC_{\ast} \ar{r} & \PShv_\cT(\cC)
\end{tikzcd} \]
is a homotopy pullback. Let $\pi$ generically denote all pullbacks of the structure map $\cK \to \cT^\op$. By the universal property of $\underline{\Spc}_\cT$ as a $\cT$-$\infty$-category of $\cT$-objects in $\Spc$ \cite[Prop.~3.10]{Exp2}, we have an identification of the constant $\cT$-diagram functor
$$\PShv_\cT(\cC) = \Fun_\cT(\cC^{\vop}, \underline{\Spc}_\cT) \to \Fun_\cT(\cK, \underline{\PShv}_\cT(\cC))$$
with the functor
$$\pi^{\ast}: \Fun(\cC^{\vop}, \Spc) \to \Fun(\cC^{\vop} \times_{\cT^{\op}} \cK, \Spc)$$
given by restriction along $\pi$. Abusing notation, let $j_\cT p$ also denote the corresponding functor $\cC^{\vop} \times_{\cT^{\op}} \cK \to \Spc$ under this equivalence. Then by \cref{lem:ParameterizedSlicesAsOrdinarySlices}, we have a homotopy pullback square
\[ \begin{tikzcd}
(\underline{\PShv}_\cT(\cC)^{/(j_\cT p,\cT)})_{\ast} \ar{r} \ar{d} & \Fun(\cC^{\vop} \times_{\cT^{\op}} \cK, \Spc)^{/j_\cT p} \ar{d} \\
\PShv_\cT(\cC) \simeq \PShv(\cC^{\rm v}) \ar{r}{\pi^{\ast}} & \Fun(\cC^{\vop} \times_{\cT^{\op}} \cK, \Spc).
\end{tikzcd} \]
By the $S = \Delta^0$ case of \cite[Lem.~8.8]{Exp2} applied to the adjunction $\pi^\ast \dashv \pi_\ast$, we deduce an equivalence $$(\underline{\PShv}_\cT(\cC)^{/(j_\cT p,\cT)})_{\ast} \simeq \PShv(\cC^{\rm v})^{/\pi_{\ast}(j_\cT p)}.$$
Next, consider the functor
$$P' = P \times \id_\cK: \cC_{\ast}^{\op} \times \cK \cong (\cC_{\ast}^{\op} \times \cT^{\op}) \times_{\cT^{\op}} \cK \to \cC^{\vop} \times_{\cT^{\op}} \cK,$$
defined to be the product of the unique $\cT$-functor $P: \cC_{\ast}^{\op} \times \cT^{\op} \to \cC^{\vop}$ extending the inclusion $\iota: \cC_{\ast}^{\op} \subset \cC^{\vop}$ on the first factor and the identity on $\cK$ on the second factor; informally, $P'(c,k) = (\chi_t^{\ast} (c), k)$ for $k \in \cK_t$ and $\chi_t: t \to \ast$ the unique map. Using the canonical unit transformation $\iota \Rightarrow P$, $P'$ fits into a lax commutative diagram
\[ \begin{tikzcd}
\cC_{\ast}^{\op} \times \cK \ar{r}{P'} \ar{d}{\pr} \ar[phantom]{dr}[rotate=45]{\Rightarrow} & \cC^{\vop} \times_{\cT^{\op}} \cK \ar{d}{\pi} \\
\cC_{\ast}^{\op} \ar{r}{\iota} & \cC^{\vop}.
\end{tikzcd} \]
We claim that the induced map $\theta: \iota^{\ast} \pi_{\ast} (j_\cT p) \to \pr_{\ast} P'^{\ast} (j_\cT p)$ is an equivalence, which we may check objectwise at each $c \in \cC_{\ast}$. Let $\underline{c}: \cT^{\op} \to \cC^{\vop}$ denote the unique $\cT$-functor such that $\underline{c}(\ast) = c$. Since $\cT$-limits in $\cT$-functor categories are computed pointwise (by the dual of \cite[Prop.~9.17]{Exp2}), under the equivalences of \cite[Prop.~3.10]{Exp2} we have a commutative square
\[ \begin{tikzcd}
\Fun(\cC^{\vop} \times_{\cT^{\op}} \cK, \Spc) \ar{r}{\underline{c}^{\ast}} \ar{d}{\pi_{\ast}} & \Fun(\cK, \Spc) \ar{d}{\pi_{\ast}} \\
\Fun(\cC^{\vop}, \Spc) \ar{r}{\underline{c}^{\ast}} & \Fun(\cT^{\op}, \Spc)
\end{tikzcd} \]
where the vertical functors are given by right Kan extension and the horizontal functors by restriction, and after evaluation at $\ast \in \cT^{\op}$ this equivalence identifies with $\theta(c)$, which proves the claim.

Abusing notation, let $j p$ also denote its adjoint $\cK \times \cC^{\op} \to \Spc$, and observe that $(P')^{\ast}(j_t p) \simeq (j p)|_{\cK \times \cC^{\op}_{\ast}}$, so that the equivalence $\theta$ yields an equivalence
\[ \PShv(\cC_\ast)^{/(\pi_\ast (j_\cT p))|_{\cC^\op_\ast}} \simeq \PShv(\cC_\ast)^{/\pr_\ast ((j p)|_{\cK \times \cC^{\op}_{\ast}})}. \]
Using that $\pr$ is a cartesian fibration, we have an equivalence $\pr_\ast ((j p)|_{\cK \times \cC^{\op}_{\ast}}) \simeq (\pr_{\ast}(j p))|_{\cC^{\op}_{\ast}}$. Note that the limit of $j p: \cK \to \PShv(\cC)$ is computed by $\pr_{\ast} (j p)$, so $\PShv(\cC)^{/jp} \simeq \PShv(\cC)^{/\pr_{\ast}(j p)}$. Using that $\cC_{\ast} \to \PShv_\cT(\cC) \simeq \PShv(\cC^{\rm v})$ factors through $\PShv(\cC_{\ast})$ and invoking \cite[Lem.~8.8]{Exp2} with respect to the adjunctions $\adjunctb{\PShv(\cC_\ast)}{\PShv(\cC^{\rm v})}$ and $\adjunctb{\PShv(\cC_\ast)}{\PShv(\cC)}$, we then reduce the claim to checking that the outer square
\[ \begin{tikzcd}
(\cC^{/(p,\cT)})_{\ast} \ar{r} \ar{d} & \PShv(\cC_{\ast})^{/(j p)|_{\cC_{\ast}^{\op}}} \ar{r} \ar{d} & \PShv(\cC)^{/jp} \ar{d} \\
\cC_{\ast} \ar{r} & \PShv(\cC_{\ast}) \ar{r} & \PShv(\cC).
\end{tikzcd} \]
is a homotopy pullback square. But this follows from \cref{lem:ParametrizedOvercategoryComparison} and \cref{lem:YonedaPreservesLimitsForSliceCategories}.

The last statement then follows from \cite[Prop.~5.24]{Exp2}, \cref{lem:ParamRepresentableFibrations}(3), and \cref{rem:YonedaLemmaExtension}.
\end{proof}

\begin{lemma} \label{lem:ParametrizedOvercategoryComparison} Suppose $\cT$ has a final object $\ast$ and let $p: \cK \to \cC$ be a $\cT$-functor. Then we have a homotopy pullback square
\[ \begin{tikzcd}
(\cC^{/(p,\cT)})_{\ast} \ar{r} \ar{d} & \cC^{/p} \ar{d} \\
\cC_{\ast} \ar{r} & \cC.
\end{tikzcd} \]
\end{lemma}
\begin{proof} Note that the inclusion of the initial object $\Delta^0 \to (\cT^{\op})^{\sharp}$ is a cocartesian equivalence in $s\Set^+_{/\cT^{\op}}$. By \cite[Thm.~4.16]{Exp2}, $i: \Delta^0 \star_{\cT^{\op}} \leftnat{\cK} \to \leftnat{\cK^{\underline{\lhd}}}$ is a cocartesian equivalence. Since $\Fun_\cT(-,-)$ is a Quillen bifunctor, precomposition by $i$ then yields a homotopy pullback square
\[ \begin{tikzcd}
\Fun_\cT(\cK^{\underline{\lhd}}, \cC) \ar{r} \ar{d} & \Fun(\cK^{\lhd}, \cC) \ar{d} \\
\cC_{\ast} \times \Fun_\cT(\cK,\cC) \ar{r} & \cC \times \Fun(\cK,\cC)
\end{tikzcd} \]
where the bottom horizontal functor is the evident inclusion. Taking the pullback over $\{p \}$ then produces the desired homotopy pullback square.
\end{proof}

We have the evident parametrized analogue of a (co)representable fibration.

\begin{definition} \label{def:CorepresentableFibration} Let $f: \cD \to \cC$ be a $\cT$-left resp. $\cT$-right fibration. If $\cD$ admits a $\cT$-initial object resp. $\cT$-final object, then we say that $f$ is \emph{$\cT$-corepresentable} resp. \emph{$\cT$-representable}.
\end{definition}

We record some basic facts about $\cT$-corepresentable $\cT$-left fibrations and the $\cT$-Yoneda embedding.

\begin{lemma} \label{lem:ParamRepresentableFibrations} \begin{enumerate} \item Let $\cD$ be a $\cT$-$\infty$-category and let $\sigma: \cT^\op \to \cD$ be a cocartesian section. Then $\sigma$ is a $\cT$-initial object of $\cD$ if and only if $\cD^{(\sigma,\cT)/} \to \cD$ is a trivial fibration. 
\item Let $f: \cD \to \cC$ be a $\cT$-functor. If $f$ is a $\cT$-corepresentable $\cT$-left fibration with $\cT$-initial object $\sigma$, then we have a canonical equivalence $\cD \simeq \cC^{(f \sigma,\cT)/}$ of $\cT$-$\infty$-categories over $\cC$.
\item Let $\sigma: \cT^\op \to \cC$ be any cocartesian section. The $\cT$-left fibration $\cC^{(\sigma,\cT)/} \to \cC$, as a left fibration, is classified by the functor
\[ \underline{\Map}_\cC(\sigma,-): \cC \xrightarrow{(\sigma^\vop,\id)} \cC^{\vop} \times_{\cT^\op} \cC \xrightarrow{\underline{\Map}_\cC(-,-)} \Spc \]
that sends $c \in \cC_t$ to $\Map_{\cC_t}(\sigma(t),c)$.
\item The functor obtained by taking cocartesian sections of the $\cT$-Yoneda embedding $j_\cT$ $$\widehat{j_\cT} : \Fun_\cT(\cT^\op,\cC) \to \PShv_\cT(\cC) \simeq \PShv(\cC^{\rm v}), \quad \widehat{j_\cT}(\sigma) \mapsto \underline{\Map}_\cC(-,\sigma) \simeq \underline{\Map}_{\cC^\vop}(\sigma^\vop,-)$$
is fully faithful. Under the straightening equivalence $\PShv_\cT(\cC) \simeq \LFib(\cC^{\vop})$, $\widehat{j_\cT}$ has essential image spanned by the $\cT$-corepresentable $\cT$-left fibrations over $\cC^{\vop}$.
\item The composite $\cT$-functor $\underline{\colim}^\cT j_\cT: \cC \to \underline{\Spc}_\cT$ is constant with value the $\cT$-final object of $\underline{\Spc}_\cT$.
\end{enumerate}
\end{lemma}
\begin{proof} First note that we have a homotopy pullback square
\[ \begin{tikzcd}
\Ar_\cT(\cD) = \cT^{\op} \times_{\Ar(\cT^\op)} \Ar(\cD) \ar{r}{\simeq} \ar{d}{\ev_0} & \underline{\Fun}_\cT(\cT \times \Delta^1, \cD) \ar{d}{\ev_0} \\
\cD \ar{r}{\simeq} & \underline{\Fun}_\cT(\cT, \cD)
\end{tikzcd} \]
in which the horizontal maps are equivalences of $\cT$-$\infty$-categories. Therefore, for all $t \in \cT$, we have an equivalence $\cD^{(\sigma,\cT)/} \times_\cD \cD_t \simeq \cD_t^{\sigma(t)/}$, and (1) follows by checking fiberwise.

For (2), suppose $f$ is a $\cT$-left fibration and $\sigma$ is a $\cT$-initial object of $\cD$. Using (1), let $\tau: \cD \to \cD^{(\sigma,\cT)/}$ be a choice of section. We then claim that the composite map $$\chi: \cD \xrightarrow{\tau} \cD^{(\sigma,\cT)/} \xrightarrow{f} \cC^{(f \sigma, \cT)/}$$ is an equivalence. Since $\chi$ is a $\cT$-functor between $\cT$-left fibrations over $\cC$, it suffices to check that $\chi_t$ is an equivalence of left fibrations over $\cC_t$ for all $t \in \cT$. But $\chi_t$ is a functor of corepresentable left fibrations that preserves the initial object, and is thus an equivalence.

For (3), recall that $\underline{\Map}_\cC(-,-)$ was defined as the straightening of the $\cT$-left fibration $\Tw_\cT(\cC) \to \cC^{\vop} \times_{\cT^\op} \cC$ given by the $\cT$-fiberwise twisted arrow category  (cf. the discussion right after \cite[Thm.~11.5]{Exp2}). By (2), it then suffices to show that the pullback $\cD = \cC \times_{(\cC^{\vop} \times_{\cT^\op} C)} \Tw_\cT(\cC)$ has a $\cT$-initial object that projects to $\sigma$. But this again reduces to the fiberwise assertion about $\cC_t \times_{(\cC_t^\op \times \cC_t)} \Tw(\cC_t) \simeq (\cC_t)^{\sigma(x)/}$; note that the assumption that $\sigma$ is a \emph{cocartesian} section ensures that the collection of fiberwise initial objects in $\cD$ is stable under cocartesian pushforward, so it indeed promotes to a $\cT$-initial object.

For (4), since $j_\cT$ is fiberwise fully faithful and the formation of cocartesian sections computes the limit of the corresponding functor into $\Cat$, it follows that $\widehat{j_\cT}$ is fully faithful. The assertion now follows from (3).

For (5), we may check the assertion fiberwise, so suppose that $\cT$ has a final object $\ast$. Then we need to show that for any $x \in \cC_\ast$, $\colim^\cT j_\cT(x) \simeq 1 \in \Spc_\cT$. But note that for all $t \in \cT$, if we let $\alpha_t: t \to \ast$ denote the unique morphism then  $j_\cT(x)|_{\cC_t^\op} \simeq \Map_{\cC_t}(-,\alpha_t^\ast x)$, hence in view of \cite[Prop.~5.5]{Exp2} the assertion follows from its non-parametrized analogue $\colim_{\cC_t^\op} j(\alpha_t^\ast x) \simeq 1$.
\end{proof}

\section{Finite, filtered, and sifted diagrams} \label{sec:FiniteFilteredSifted}

In this section, we develop the theory of $\cT$-$\kappa$-small, $\cT$-filtered, and $\cT$-sifted $\cT$-$\infty$-categories (\cref{def:Tfinite}, \cref{def:ParamFiltered}, and \cref{def:ParamSifted}). To prepare for our discussion, we begin with the following proposition, which recovers and extends the homotopy colimit decomposition result of \cite[Cor.~4.2.3.10]{HTT}. Its statement involves the ``lower $\cT$-slice'' construction of \cite[Def.~4.17]{Exp2}.

\begin{thm} \label{prop:ParamColimitDecomposition} Let $\cC$ be a $\cT$-$\infty$-category.
\begin{enumerate}
    \item The assignments $[K \xrightarrow{p} \leftnat{\cC}] \mapsto [\leftnat{\cC}_{(p,\cT)/} \to \cC]$ and $[L \xrightarrow{q} \leftnat{\cC}] \mapsto [\leftnat{\cC}_{/(p,\cT)} \to \leftnat{\cC}]$ of marked simplicial sets assemble to a Quillen adjunction
    \[ \adjunct{\cC_{(-,\cT)/}}{s\Set^+_{/\leftnat{\cC}}}{(s\Set^+_{/\leftnat{\cC}})^\op}{\cC_{/(-,\cT)}} \] 
    with respect to the slice model structure (and its opposite) induced from the cocartesian model structure on $s\Set^+_{/\cT^\op}$. Consequently, we obtain an adjunction of $\infty$-categories
    \[ \adjunct{\cC_{(-,\cT)/}}{(\Cat_\cT)_{/ \cC}}{(\Cat_\cT)^\op_{/ \cC}}{\cC_{/(-,\cT)}}. \]
    \item Let $p_{\bullet}: \cI \to (\Cat_\cT)_{/ \cC}$ be a functor with colimit $p: \cK \to \cC$ and suppose that for every $i \in \cI$, the $\cT$-functor $p_i: \cK_i \to \cC$ admits a $\cT$-colimit $\sigma_i$. Then the $\sigma_i$ assemble to a $\cT$-functor $\sigma_{\bullet}: \cI \times \cT^\op \to \cC$ such that if $\sigma_{\bullet}$ admits a $\cT$-colimit $\sigma$, then $p$ admits a $\cT$-colimit given by $\sigma$.
\end{enumerate}
\end{thm}
\begin{proof} (1): Let us suppress the markings on $\cC$ and its relatives for clarity. The two displayed functors participate in an adjunction in view of the definitional isomorphisms of hom-sets
\[ \Hom_{/ \cC}(L,\cC_{(p,\cT)/}) \cong \Hom_{K\sqcup L/ /\cT^\op}(K \star_{\cT^\op} (L \times_{\cT^\op} \Ar(\cT^\op)^\sharp), C) \cong \Hom_{/ \cC}(K,\cC_{/(q,\cT)}). \]
The left adjoint preserves cofibrations and weak equivalences by \cite[Prop.~4.18]{Exp2}, \cite[Prop.~4.19]{Exp2}, and the discussion immediately proceeding it, so the adjunction is Quillen. Finally, even though this adjunction is not generally simplicial, we may descend to an adjunction on the underlying $\infty$-categories by \cite[Cor.~1.3.4.26]{HA}.

(2): Let $\overline{p_{\bullet}}: \cI^\rhd \to (\Cat_\cT)_{/ \cC}$ be a colimit diagram extending $p_{\bullet}$ and let $\overline{\varphi}:(\cI^\op)^{\lhd} \to (\Cat_\cT)_{/ \cC}$ be the opposite of the postcomposition of $\overline{p_{\bullet}}$ with $\cC_{(-,\cT)/}$. By (1), $\overline{\varphi}$ is a limit diagram. We wish to show that the value $\cC_{(p,\cT)/} \to \cC$ on its cone point is a $\cT$-corepresentable $\cT$-left fibration with $\cT$-initial object as indicated.

By assumption, $\varphi$ factors through the subcategory of $\cT$-corepresentable $\cT$-left fibrations over $\cC$, hence by \cref{lem:ParamRepresentableFibrations}(4) applied to $\cC^\vop$ and using that $(-)^\op$ preserves limits, we may factor $\varphi^\op$ as
$$\psi: \cI \to \Fun_\cT(\cT^\op,\cC^\vop)^\op \simeq \Fun_\cT(\cT^\op,\cC).$$
Let $\sigma_{\bullet}: \cI \times \cT^\op \to \cC$ be the adjoint $\cT$-functor and suppose $\sigma_{\bullet}$ extends to a $\cT$-colimit diagram
$$\overline{\sigma_{\bullet}}: (\cI \times \cT^\op) \star_{\cT^\op} \cT^\op \simeq \cI^\rhd \times \cT^\op \to \cC$$
(necessarily then adjoint to a colimit diagram $\overline{\psi}: \cI^\rhd \to \Fun_\cT(\cT^\op,\cC)$ extending $\psi$). Since the $\cT$-Yoneda embedding $j_\cT$ strongly preserves $\cT$-limits \cite[Cor.~11.10]{Exp2}, we obtain a $\cT$-limit diagram $j_\cT \circ \overline{\sigma_{\bullet}}{}^\vop$ and thus a limit diagram
$$\widehat{j_\cT} \circ \overline{\psi}{}^\op: (\cI^\op)^\lhd \to \Fun_\cT(\cT^\op,\cC^\vop) \to \PShv_\cT(\cC^\vop) \simeq \LFib(\cC).$$
Since $\LFib(\cC)$ is a subcategory of $(\Cat_\cT)_{/ \cC}$ stable under limits and $\overline{\varphi}$ factors through this subcategory, we deduce that $\overline{\varphi} \simeq \widehat{j_\cT} \circ \overline{\psi}{}^{\op}$ as limit diagrams extending $\varphi$, which proves the claim.
\end{proof}


\begin{definition} \label{def:Tfinite} Let $\Delta_\cT \subset \Cat_\cT \simeq \Fun(\cT^\op,\Cat)$ be the full subcategory spanned by the objects $\{ \Delta^n \times \Map_\cT(-,t) : t \in \cT, n \geq 0 \}$. Then for every regular cardinal $\kappa$, define the full subcategory $\Cat_\cT^{\ksmall} \subset \Cat_\cT$ of \emph{$\cT$-$\kappa$-small $\cT$-$\infty$-categories} to be the smallest full subcategory that contains $\Delta_\cT$ and is closed under all colimits indexed by $\kappa$-small simplicial sets.

If $\kappa = \omega$, then we will also use the terminology \emph{$\cT$-finite} in place of $\cT$-$\omega$-small.
\end{definition}

\begin{remark}\label{rem:CatYonedaEmbedding} Let $i: \Delta \subset \Cat$ be the usual inclusion of the simplex category, which extends to the adjunction $[\adjunct{L}{\PShv(\Delta)}{\Cat}{R}]$ whose fully faithful right adjoint exhibits $\Cat$ as the full subcategory of complete Segal spaces in $\PShv(\Delta)$. Then applying $\Fun(\cT^\op,-)$ to $L \dashv R$ yields an adjunction $[\adjunct{L_\cT}{\PShv(\Delta \times \cT)}{\Cat_\cT}{R_\cT}]$, where $L_\cT$ is the unique functor extending $i_\cT: \Delta \times \cT \to \Cat_\cT, ([n],t) \mapsto \Delta^n \times \Map_\cT(-,t)$ and $R_\cT$ is fully faithful. In particular, $i_\cT$ is fully faithful with essential image $\Delta_\cT$. Furthermore, using that $\id_{\PShv(\Delta \times \cT)}$ is the left Kan extension of the Yoneda embedding along itself, we deduce that $\id_{\Cat_\cT}$ is the left Kan extension of $i_\cT$ along itself; indeed, for any $\cT$-$\infty$-category $\cC$, applying the colimit-preserving functor $L_\cT$ to the equivalence $$R_\cT(\cC) \simeq \colim [(\Delta \times \cT) \times_{\PShv(\Delta \times \cT)} \PShv(\Delta \times \cT)^{/R_\cT(\cC)} \to \PShv(\Delta \times \cT)]$$
and using that $(\Delta \times \cT) \times_{\PShv(\Delta \times \cT)} \PShv(\Delta \times \cT)^{/R_\cT(\cC)} \simeq (\Delta \times \cT) \times_{\Cat_\cT} (\Cat_\cT)^{/ \cC}$ proves the claim.

As a corollary, if $\kappa_0$ is the strongly inaccessible cardinal that fixes our definition of small simplicial set, then $\Cat_\cT^{\operatorname{\kappa_0-small}} = \Cat_\cT$. Our definition of small $\cT$-$\infty$-category is thus unambiguous.
\end{remark}

\begin{remark}\label{rem:essentiallysmall} In \cref{def:Tfinite}, if $\cT = \Delta^0$, then the notion of a $\cT$-$\kappa$-small $\cT$-$\infty$-category $\cC$ coincides with that of an \emph{essentially} $\kappa$-small $\infty$-category \cite[Def.~5.4.1.3]{HTT}: that is, there exists a $\kappa$-small simplicial set $K$ and a categorical equivalence $K \to \cC$. To see this, note that if $\cC$ is essentially $\kappa$-small then by \cite[Var.~4.2.3.15]{HTT} if $\kappa > \omega$ or \cite[Var.~4.2.3.16]{HTT} if $\kappa = \omega$, there exists a $\kappa$-small simplicial set $L$ and a map $\phi: L \to \Delta$ such that $\cC \simeq \colim_L \phi$. Conversely, the full subcategory of essentially $\kappa$-small $\infty$-categories is closed under $\kappa$-small colimits by a direct argument if $\kappa = \omega$ and by the identification with $\kappa$-compact objects in $\Cat$ if $\kappa > \omega$ \cite[Prop.~5.4.1.2]{HTT}.
\end{remark}

\begin{remark}\label{rem:ExplicitKSmallPresentation} We may describe the $\infty$-category $\Cat^{\ksmall}_\cT$ more explicitly in the following manner. Define an increasing union of full subcategories $(\Cat^{\ksmall}_\cT)_{\alpha} \subset \Cat^{\ksmall}_\cT$ for all ordinals $\alpha \leq \kappa$ inductively as follows:
\begin{enumerate}[start=0]
    \item $(\Cat^{\ksmall}_\cT)_0 = \Delta_\cT$.
    \item For every successor ordinal $\beta = \alpha+1$, let $(\Cat^{\ksmall}_\cT)_{\beta}$ consist of all colimits of diagrams
    $$K \to (\Cat^{\ksmall}_\cT)_{\alpha}$$
    indexed by $\kappa$-small simplicial sets $K$.
    \item For every limit ordinal $\lambda \leq \kappa$, let $(\Cat^{\ksmall}_\cT)_{\lambda} = \bigcup_{\alpha<\lambda} (\Cat^{\ksmall}_\cT)_{\alpha}$.
\end{enumerate}
Then since the ordinal $\kappa$ is itself $\kappa$-filtered, we conclude that $\Cat^{\ksmall}_\cT = (\Cat^{\ksmall}_\cT)_{\kappa}$.
\end{remark}

We now apply \cref{prop:ParamColimitDecomposition} to prove a $\kappa$-$\cT$-cocompleteness result that generalizes \cite[Cor.~12.15]{Exp2}, with a different method of proof.

\begin{thm} \label{prop:FiberwiseAndCoproductsGiveAll} Let $\cC$ be a $\cT$-$\infty$-category and $\kappa$ a regular cardinal. Then $\cC$ strongly admits all $\cT$-$\kappa$-small $\cT$-colimits if and only if
\begin{enumerate}
    \item For every $t \in \cT$, $\cC_{\underline{t}}$ admits $\cT^{/t}$-colimits indexed by corepresentable left fibrations.
    \item For every $t \in \cT$, $\cC_t$ admits $\kappa$-small colimits, and for every $\alpha: s \to t$, the restriction functor $\alpha^\ast: \cC_t \to \cC_s$ preserves $\kappa$-small colimits.
\end{enumerate}
Furthermore, if $\cC$ and $\cD$ are $\cT$-$\infty$-categories that strongly admit $\cT$-$\kappa$-small $\cT$-colimits and $F: \cC \to \cD$ is a $\cT$-functor, then $F$ strongly preserves all $\cT$-$\kappa$-small $\cT$-colimits if and only if
\begin{enumerate}[a.]
    \item For every $\alpha: s \to t$, the mate $\alpha_! F_s \Rightarrow F_t \alpha_!$ is an equivalence.
    \item For every $t \in \cT$, $F_t$ preserves all $\kappa$-small colimits.
\end{enumerate}
\end{thm}
\begin{proof} For the `only if' direction, suppose that $\cC$ strongly admits $\cT$-colimits. Then (1) holds since $\Delta_{\cT^{/t}} \subset \Cat^{\ksmall}_{\cT^{/t}}$ by definition, and (2) holds since if $\cK$ is an essentially $\kappa$-small $\infty$-category, then as noted in \cref{rem:essentiallysmall} there exists a $\kappa$-small simplicial set $L$ and a functor $\phi: L \to \Delta \times \{\id_t \} \subset \Delta \times \cT^{/t} \xrightarrow{\simeq} \Delta_{\cT^{/t}}$ such that $\cK \times (\cT^{/t})^\op \simeq \colim_L \phi$ in $\Cat_\cT$, hence $\cK \times (\cT^{/t})^\op$ is a $\cT$-$\kappa$-small $\cT$-$\infty$-category.

Conversely, suppose $\cC$ satisfies (1) and (2). Let $\CMcal{K}_t$ denote the full subcategory of $\Cat_{\cT^{/t}}$ consisting of all $\cT^{/t}$-$\infty$-categories $\cK$ such that all $\cK$-indexed $\cT^{/t}$-diagrams in $\cC_{\underline{t}}$ admit $\cT^{/t}$-colimits. We wish to show that $\CMcal{K}_t$ contains $\Cat_{\cT^{/t}}^{\ksmall}$. To ease notation, let us replace $\cT^{/t}$ by $\cT$ and suppose that $\cT$ has a final object $\ast$. By our first assumption, $\CMcal{K}_{\ast}$ contains $\Delta_{T}$. It thus suffices to show that $\CMcal{K}_{\ast}$ is closed under $\kappa$-small colimits in $\Cat_{T}$. So suppose we have a diagram $f: I \to \CMcal{K}_{\ast}$ with $I$ $\kappa$-small such that $f$ has colimit $\cK$ in $\Cat_\cT$, and let $p: \cK \to \cC$ be a $\cT$-functor. Then since colimits in $(\Cat_\cT)_{/ \cC}$ are created by the forgetful functor to $\Cat_\cT$, we obtain a colimit diagram $p_{\bullet}: I^{\rhd} \to (\Cat_\cT)_{/ \cC}$ such that for each $i \in I$, the $\cT$-functor $p_i$ admits a $\cT$-colimit $x_i \in \cC_\ast$. By our second assumption combined with \cite[Cor.~5.9]{Exp2} and \cref{prop:ParamColimitDecomposition}(2), we deduce that $p$ admits a $\cT$-colimit, which proves the claim. Repeating the same type of argument establishes the assertion about the $\cT$-functor $F$.
\end{proof}

\begin{remark}\label{rem:OrbitalAdditionToFiberwiseAndCoproductsGiveAll} If we suppose that $\cT$ is orbital in \cref{prop:FiberwiseAndCoproductsGiveAll}, then by \cite[Prop.~5.12]{Exp2} we may replace (1) with the assumption that for all $\alpha: s \to t$, $\alpha^\ast$ admits a left adjoint $\alpha_!$, and for all pullback squares in $\FF_\cT$
\[ \begin{tikzcd}
U' \ar{r}{\beta'} \ar{d}{\alpha'} & U \ar{d}{\alpha} \\
V' \ar{r}{\beta} & V,
\end{tikzcd} \]
the mate $\alpha'_! \beta'^{\ast} \Rightarrow \beta^\ast \alpha_!$ is an equivalence.
\end{remark}


\begin{definition} \label{def:ParamFiltered} Let $\cJ$ be a $\cT$-$\infty$-category and let $\kappa$ be a regular cardinal. Then $\cJ$ is \emph{$\cT$-$\kappa$-filtered} if for all $t \in \cT$ and $\cT^{/t}$-$\kappa$-small $\cK$, every $\cT^{/t}$-functor $p: \cK \to \cJ_{\underline{t}}$ admits an extension to a $\cT^{/t}$-functor $\overline{p}: \cK^{\underline{\rhd}} \to \cJ_{\underline{t}}$.
\end{definition}

\begin{lemma} \label{lem:JoinPreservesSmallness} Suppose that $\cT$ is orbital and has a final object $\ast$. Let $\kappa$ be a regular cardinal and suppose that $\cK$ and $\cL$ are $\cT$-$\kappa$-small $\cT$-$\infty$-categories. Then $\cK \star_{\cT^\op} \cL$ is $\cT$-$\kappa$-small.
\end{lemma}
\begin{proof} Let $(\Cat^{\ksmall}_\cT)_{\alpha}$ be as in \cref{rem:ExplicitKSmallPresentation}. First observe that if $\cK = \Delta^n \times (\cT^{/t})^{\op}$ and $\cL = \Delta^m \times (\cT^{/s})^{\op}$, then $\cK \times \cL \in (\Cat^{\operatorname{\omega-small}}_\cT)_{1}$ using that $\cT$ is orbital. Because
$$\cK \star_{\cT^\op} \cL \simeq (\cK \times_{\cT^\op} \cL) \times \Delta^1 \coprod_{(\cK \times_{\cT^\op} \cL) \times  \partial \Delta^1} \cK \sqcup \cL \: ,$$
we deduce that the $\cT$-join restricts to a functor $\Delta_\cT \times \Delta_\cT \to (\Cat^{\operatorname{\omega-small}}_\cT)_{1} \subset (\Cat^{\ksmall}_\cT)_1$. We then formulate the following inductive hypothesis:
\begin{enumerate}
    \item If $\beta = \alpha+1 < \kappa$ is a successor ordinal, then for all $\cK,\cL \in (\Cat^{\ksmall}_\cT)_\alpha$, $\cK \star_{\cT^\op} \cL \in (\Cat^{\ksmall}_\cT)_{\beta}$.
    \item If $\lambda \leq \kappa$ is a limit ordinal, then for all $\cK,\cL \in (\Cat^{\ksmall}_\cT)_{\lambda}$, $\cK \star_{\cT^\op} \cL \in (\Cat^{\ksmall}_\cT)_{\lambda}$. 
\end{enumerate}
Because the $\cT$-join preserves colimits separately in each variable (as may be checked fiberwise) and $\kappa$ is a regular cardinal, we may proceed by induction to confirm the inductive hypothesis for all $\lambda \leq \kappa$; this proves the claim since $\Cat^{\ksmall}_\cT = (\Cat^{\ksmall}_\cT)_{\kappa}$.
\end{proof}


\begin{lemma} \label{lem:FilteredPreservedBySlices} Suppose that $\cT$ is orbital and $\cJ$ is $\cT$-$\kappa$-filtered. Then for any $\cT^{/t}$-$\kappa$-small $\cK$ and $\cT^{/t}$-functor $p: \cK \to \cJ_{\underline{t}}$, $(\cJ_{\underline{t}})^{(p,\cT^{/t})/}$ is $\cT^{/t}$-$\kappa$-filtered.
\end{lemma}
\begin{proof} Replace $\cT^{/t}$ by $\cT$ and suppose $\cK$ is $\cT$-$\kappa$-small and $p: \cK \to \cJ$ is a $\cT$-functor. To check that $\cJ^{(p,\cT)/}$ is $\cT$-filtered, after replacing $\cT^{/t}$ by $\cT$ once more it suffices to show that for any $\cT$-functor $\phi: \cL \to \cJ^{(p,\cT)/}$ with $\cT$-$\kappa$-small $\cL$, $\phi$ extends over $\cL^{\underline{\rhd}}$. Using that $\cJ^{(p,\cT)/} \simeq \cJ_{(p,\cT)/}$, by adjunction it suffices to extend the $\cT$-functor $\psi: \cK \star_{\cT^\op} \cL \to \cJ$ (under $p \sqcup q$) over $(\cK \star_{\cT^\op} \cL)^{\underline{\rhd}}$. By our assumption that $\cJ$ is $\cT$-filtered, this is possible since $\cK \star_{\cT^\op} \cL$ is $\cT$-$\kappa$-small by \cref{lem:JoinPreservesSmallness}.
\end{proof}

\begin{theorem} \label{thm:EquivalentConditionsFiltered} Suppose that $\cT$ is orbital. Let $\cJ$ be a $\cT$-$\infty$-category and let $\kappa$ be a regular cardinal. The following conditions are equivalent:
\begin{enumerate}
    \item $\cJ$ is $\cT$-$\kappa$-filtered.
    \item For all $t \in \cT$, $\cJ_t$ is $\kappa$-filtered, and $\cJ$ is cofinal-constant (\cref{def:cofinalconstant}).
    \item The $\cT$-colimit $\cT$-functor $\underline{\colim}^\cT_\cJ: \underline{\Fun}_\cT(\cJ, \underline{\Spc}_\cT) \to \underline{\Spc}_\cT$ strongly preserves $\cT$-$\kappa$-small $\cT$-limits.
\end{enumerate}
\end{theorem}
\begin{proof} First suppose (1). Then for any essentially $\kappa$-small $\infty$-category $\cK$ and $t \in \cT$, our assumption ensures that every $\cT^{/t}$-functor $\cK \times (\cT^{/t})^\op \to \cJ_{\underline{t}}$ extends over $\cK^\rhd \times (\cT^{/t})^\op$, which shows that $\cJ_t$ is $\kappa$-filtered. Now suppose $\alpha: s \to t$ is a morphism in $\cT$. We want to show that $\alpha^{\ast}: \cJ_t \to \cJ_s$ is cofinal. Let $x \in \cJ_s$ and $\sigma: (\cT^{/s})^\op \to \cJ_{\underline{t}}$ be the unique $\cT^{/t}$-functor that selects $x$, and note that by \cref{lem:ParameterizedSlicesAsOrdinarySlices}
\[ \cJ_t \times_{\cJ_s} (\cJ_s)^{x/} \simeq ((\cJ_{\underline{t}})^{(\sigma,\cT^{/t})/})_{t}. \]
By \cref{lem:FilteredPreservedBySlices}, $(\cJ_{\underline{t}})^{(\sigma,\cT^{/t})/}$ is $\cT^{/t}$-$\kappa$-filtered, so $((\cJ_{\underline{t}})^{(\sigma,\cT^{/t})/})_{t}$ is $\kappa$-filtered by what was previously shown, and hence weakly contractible.\footnote{In fact, we don't need to invoke \cref{lem:JoinPreservesSmallness} as in the proof of \cref{lem:FilteredPreservedBySlices} because we are only interested in the extension property for constant $\cT^{/t}$-diagrams; in particular, the assumption that $\cT$ is orbital there is not necessary in this instance.} The claim now follows by Joyal's cofinality theorem. We conclude that $\cJ$ is cofinal-constant, so $\cJ$ satisfies (2).

Next suppose (2). To show (3), by the dual of \cref{prop:FiberwiseAndCoproductsGiveAll} it suffices to show that $\colim^\cT_\cJ$ preserves $\kappa$-small limits fiberwise and intertwines with the `coinduction' right adjoints. By \cite[Prop.~5.5]{Exp2}, under the equivalence $\Fun_{\cT^{/t}}(\cJ_{\underline{t}}, \underline{\Spc}_{\cT^{/t}}) \simeq \Fun(\cJ_{\underline{t}}, \Spc)$, the functor $\colim^{\cT^{/t}}_{\cJ_{\underline{t}}}$ identifies with left Kan extension along the structure map $\cJ_{\underline{t}} \to (\cT^{/t})^\op$. In view of our assumption that the fibers of $\cJ$ are $\kappa$-filtered, by \cite[Prop.~5.3.3.3]{HTT} and \cite[Prop.~4.3.3.10]{HTT} we conclude that $\colim^{\cT^{/t}}_{\cJ_{\underline{t}}}$ preserves $\kappa$-small limits. Next, let $\alpha: s \to t$ be a morphism in $\cT$ and consider the resulting pullback square
\[ \begin{tikzcd}
\cJ_{\underline{s}} \ar{r}{\phi} \ar{d}{\pi} & \cJ_{\underline{t}} \ar{d}{\pi} \\
(\cT^{/s})^{\op} \ar{r}{\phi} & (\cT^{/t})^{\op}
\end{tikzcd} \quad \text{yielding} \quad
\begin{tikzcd}
\Fun(\cJ_{\underline{t}}, \Spc) \ar{r}{\pi_!} \ar[phantom]{rd}[rotate=-45]{\Rightarrow} & \Fun((\cT^{/t})^\op, \Spc) \\
\Fun(\cJ_{\underline{s}}, \Spc) \ar{u}{\phi_\ast} \ar{r}{\pi_!} & \Fun((\cT^{/s})^\op, \Spc). \ar{u}{\phi_\ast}
\end{tikzcd} \]
We need to show that the mate $\chi: \pi_! \phi_\ast \Rightarrow \phi_\ast \pi_!$ is an equivalence. To ease notation, let us replace $\cT^{/t}$ by $\cT$ and $t$ by $\ast$. Let $p: \cJ_{\ast} \times \cT^\op \to \cJ$ be the unique $\cT$-functor extending the inclusion $\cJ_{\ast} \subset \cJ$ and note that our assumption that $p$ is $\cT$-cofinal yields a factorization of $\underline{\colim}^\cT_\cJ$ as
\[ \underline{\Fun}_\cT(\cJ, \underline{\Spc}_\cT) \xrightarrow{p^{\ast}} \underline{\Fun}_\cT(\cJ_\ast \times \cT^\op, \underline{\Spc}_\cT) \xrightarrow{\underline{\colim}^\cT_{\cJ_{\ast} \times \cT^\op}} \underline{\Spc}_\cT. \]
Since the $\cT$-functor $p^\ast$ admits a $\cT$-left adjoint given by $\cT$-left Kan extension, $p^\ast$ commutes with all $\cT$-limits \cite[Cor.~8.9]{Exp2}. We may thus replace $\cJ$ by the constant diagram $\cJ_{\ast} \times \cT^\op$ in the proof.

Now let $F: \cJ_{\ast} \times (\cT^{/s})^\op \to \Spc$ be a functor and $u \in \cT$; we will show that $\chi_F(u)$ is an equivalence. For every $\gamma: v \to s$, let $F_{\gamma} = F|_{\cJ_{\ast} \times \{ \gamma \}}$. Invoking our assumption that $\cT$ is orbital, let $\{ v_i \in \cT: i \in I \}$ be a finite collection such that $\cT^{/s} \times_\cT \cT^{/u} \simeq \coprod_{i \in I} \cT^{/v_i}$ and let $\gamma_i: v_i \to s$ be the structure maps. Then
\[ (\phi_\ast \pi_! F)(u) \simeq \lim\left( \bigsqcup_{i \in I} (\cT^\op)^{v_i/} \to (\cT^\op)^{s/} \xrightarrow{\pi_! F} \Spc \right) \simeq \prod_{i \in I} (\pi_! F)(\gamma_i) \simeq \prod_{i \in I} \colim_{\cJ_{\ast}} F_{\gamma_i} \]
On the other hand, using that the upper $\phi$ is a map of cartesian fibrations over $\cJ_{\ast}$, we get that $\phi_{\ast} F$ is computed fiberwise over $\cJ_{\ast}$ by \cite[Prop.~4.3.3.10]{HTT}. Thus for all $x \in \cJ_{\ast}$,
\[ (\phi_{\ast} F)(x,u) \simeq \lim \left( \bigsqcup_{i \in I} (\cT^\op)^{v_i/} \to \{x \} \times (\cT^\op)^{s/} \xrightarrow{F} \Spc \right) \simeq \prod_{i \in I} F(x,\gamma_i), \]
and using that $\kappa$-filtered colimits commute with finite products in $\Spc$, we deduce that $\chi_F(u)$ is an equivalence.

Finally, suppose (3). To show (1), after the usual reduction it suffices to prove that if $\cT$ admits a final object $\ast$ and $q: \cK \to \cJ$ is a $\cT$-functor with $\cK$ $\cT$-$\kappa$-small, then $(\cC^{(q,\cT)/})_{\ast}$ is nonempty. Let $\varphi = \lim^\cT_{\cK^{\vop}} ( j_\cT q^{\vop} ) \in \PShv_\cT(\cJ^{\vop}) \simeq \Fun(\cJ, \Spc)$. Then
\[ \colim^\cT_\cJ ( \varphi) \simeq \lim^\cT_{\cK^{\vop}} (\underline{\colim}^\cT_\cJ j_\cT q^{\vop}) \simeq \lim^\cT_{\cK^{\vop}} (1_\cT) \simeq 1_\cT, \]
using that $\underline{\colim}^\cT j_\cT$ factors through the $\cT$-final object $1_\cT$ of $\underline{\Spc}_\cT$ by \cref{lem:ParamRepresentableFibrations}(5) and any $\cT$-limit of $\cT$-final objects is again $\cT$-final. By \cref{prop:ParamYonedaPreservesSlices} applied to $p = q^{\vop}$ and \cref{lem:ParamClassifyingSpaceLeftFibration}, we deduce that $(\cC^{(q,\cT)/})_{\ast}$ is weakly contractible, so in particular nonempty.
\end{proof}

\begin{lemma} \label{lem:ParamClassifyingSpaceLeftFibration} Let $\pi: \cC \to \cT^\op$ be a $\cT$-$\infty$-category, $p: \cD \to \cC$ a $\cT$-left fibration, and $F: \cC \to \underline{\Spc}_\cT$ a $\cT$-functor that classifies $p$. Let $|\cD|_\cT$ denote the $\cT$-space obtained by inverting all morphisms in the fibers $\cD_t$ for all $t \in \cT$. Then $|\cD|_\cT \simeq \colim^\cT F$. 
\end{lemma}
\begin{proof} Recall again that under the equivalence $\Fun_\cT(\cC, \underline{\Spc}_\cT) \simeq \Fun(\cC, \Spc)$ and the identification of $\cT$-left fibrations with left fibrations, $p$ is classified as a left fibration by $F^\dagger: \cC \to \Spc$ and $\colim^\cT F \simeq \pi_! F^{\dagger}$. Denote the classifying space adjunction by $\adjunct{|-|}{\Cat}{\Spc}{\iota}$. In view of the functoriality of the straightening equivalence \cite[Prop.~3.2.1.4]{HTT}, we have that the functor $\LFib(\cC) \to \LFib(\cT^\op)$ defined by $\cD \mapsto |\cD|_\cT$ identifies with the composite
\[ L: \Fun(\cC, \Spc) \xrightarrow{\iota} \Fun(\cC, \Cat) \xrightarrow{\pi_!} \Fun(\cT^\op,\Cat) \xrightarrow{|-|} \Fun(\cT^\op, \Spc). \]
Since $\pi^{\ast} \iota \simeq \iota \pi^{\ast}$ the right adjoint of $L$ identifies with $\pi^{\ast}$, so $L$ is canonically equivalent to $\pi_!$.
\end{proof}

\begin{thm} \label{prop:FilteredCofinalityCriterion} Suppose that $\cT$ is orbital. Let $\cJ$ be a $\cT$-$\infty$-category and $\kappa$ a regular cardinal. Then $\cJ$ is $\cT$-$\kappa$-filtered if and only if for all $t \in \cT$ and $\cT^{/t}$-$\kappa$-small $\cK$, the diagonal $\cT^{/t}$-functor $$\delta: \cJ_{\underline{t}} \to \underline{\Fun}_{\cT^{/t}}(\cK, \cJ_{\underline{t}})$$ is $\cT^{/t}$-cofinal.
\end{thm}
\begin{proof} In the proof, let us replace $\cT^{/t}$ with $\cT$ and suppose that $\cT$ has a final object $\ast$. For the `if' direction, after replacing $\cT^{/t}$ by $\cT$ once more it suffices to show that for any $\cT$-$\kappa$-small $\cK$ and $\cT$-functor $p: \cK \to \cJ$, $p$ extends over $\cK^{\underline{\rhd}}$, i.e., $(\cJ^{(p,\cT)/})_{\ast}$ is nonempty. But since $(\cJ^{(p,\cT)/})_{\ast} \simeq \cJ_{\ast} \times_{\Fun_\cT(\cK,\cJ)} \Fun_\cT(\cK,\cJ)^{\{p\}/}$ by the dual of \cref{lem:ParameterizedSlicesAsOrdinarySlices}, this follows by our assumption that $\cJ_\ast \to \Fun_\cT(\cK,\cJ)$ is cofinal. Conversely, for the `only if' direction suppose that $\cJ$ is $\cT$-$\kappa$-filtered and let $\cK$ be $\cT$-$\kappa$-small. After replacing $\cT^{/t}$ by $\cT$, it suffices to show that $\delta_\ast$ is cofinal. For this, by Joyal's cofinality theorem and \cref{lem:ParameterizedSlicesAsOrdinarySlices} again, it suffices to show that $(\cJ^{(p,\cT)/})_{\ast}$ is weakly contractible for all $p: \cK \to \cJ$. But this follows by \cref{lem:FilteredPreservedBySlices} and the (1)$\Rightarrow$(2) implication of \cref{thm:EquivalentConditionsFiltered}.
\end{proof}

To formulate the notion of a $\cT$-sifted $\cT$-$\infty$-category, we adopt the viewpoint of the alternative characterization of \cref{prop:FilteredCofinalityCriterion}, but over a more restrictive class of diagrams.

\begin{definition} \label{def:ParamSifted} Let $\cJ$ be a $\cT$-$\infty$-category. Then $\cJ$ is \emph{$\cT$-sifted} if for all $t \in \cT$ and finite $\cT^{/t}$-sets $U$, the diagonal $\cT^{/t}$-functor $\delta: \cJ_{\underline{t}} \to \underline{\Fun}_{\cT^{/t}}(\underline{U}, \cJ_{\underline{t}})$ is $\cT^{/t}$-cofinal.
\end{definition}

\begin{theorem} \label{thm:EquivalentConditionsSifted} Suppose that $\cT$ is orbital and let $\cJ$ be a $\cT$-$\infty$-category. The following conditions are equivalent:
\begin{enumerate}
    \item $\cJ$ is $\cT$-sifted.
    \item For all $t \in \cT$, $\cJ_t$ is sifted, and $\cJ$ is cofinal-constant (\cref{def:cofinalconstant}).
    \item The $\cT$-colimit $\cT$-functor $\underline{\colim}^\cT_\cJ: \underline{\Fun}_\cT(\cJ, \underline{\Spc}_\cT) \to \underline{\Spc}_\cT$ preserves finite $\cT$-products. 
\end{enumerate}
\end{theorem}
\begin{proof} First suppose (1). Then for $t \in \cT$ and $U = \id_t \sqcup \id_t$, the $\cT^{/t}$-cofinality of $\delta: \cJ_{\underline{t}} \to \underline{\Fun}_\cT(\underline{U},\cJ_{\underline{t}})$ ensures that $\cJ_t$ is sifted, whereas if we let $U = [\alpha: s \to t]$, then because $\delta_{\id_t} \simeq \alpha^{\ast}: \cJ_t \to \cJ_s$ we deduce that $\alpha^{\ast}$ is cofinal, hence $\cJ$ is cofinal-constant. This shows that $\cJ$ satisfies (2).

The implication (2)$\Rightarrow$(3) follows by the same proof as (2)$\Rightarrow$(3) in \cref{thm:EquivalentConditionsFiltered}. Finally, suppose (3). By Joyal's cofinality theorem and \cref{lem:ParameterizedSlicesAsOrdinarySlices}, it suffices to show that for all $t \in \cT$, finite $\cT^{/t}$-set $U$, and $\cT^{/t}$-functor $p: \underline{U} \to \cJ_{\underline{t}}$, the $\infty$-category $((\cJ_{\underline{t}})^{(p,\cT^{/t})/})_{\id_t}$ is weakly contractible. But this follows by the same proof as (3)$\Rightarrow$(1) in \cref{thm:EquivalentConditionsFiltered}.
\end{proof}

We end this section by explaining a parametrized generalization of the following fact: if $F: \cC \times \cC \to \cD$ is a functor that commutes with colimits separately in each variable, then $F$ preserves sifted colimits. To do this, we need to recall the appropriate parametrized notion of distributive functor, whose definition is due to Nardin. We first fix some local notation.

\begin{dfn}
Let $U$ be a finite $\cT$-set. A \emph{$\underline{U}$-$\infty$-category} is a cocartesian fibration over $\underline{U}$.\footnote{Perhaps confusingly, this is reversing our convention that a $\cT$-$\infty$-category is a cocartesian fibration over $\cT^{\op}$. However, we don't want to write ``$\underline{U}^{\op}$-$\infty$-category''.}
\end{dfn}

\begin{ntn}
Let $f: U \to V$ be a map of finite $\cT$-sets. Then we have the adjunction
\[ \adjunct{f^*}{\Cat_{\underline{V}}}{\Cat_{\underline{U}}}{f_*} \]
where $f^*$ is pullback along $\underline{U} \to \underline{V}$. Note also that some authors also prefer to write $f_*$ as $\prod_f$ to emphasize its interpretation as an indexed product.
\end{ntn}

To understand the following definition, the reader should convince themselves that it reduces to ``preserving colimits separately in each variable'' when $\cT = \ast$. 

\begin{dfn}[{\cite[Def.~3.15]{nardin}}] \label{dfn:DistributiveFunctor}
Suppose that $\cT$ is orbital, let $f: U \to V$ be a map of finite $\cT$-sets, let $\cC$ be a $\underline{U}$-$\infty$-category, and let $\cD$ be a $\underline{V}$-$\infty$-category. Let $F: f_* \cC \to \cD$ be a $\underline{V}$-functor. Then we say that $F$ is \emph{$\underline{V}$-distributive} if for every pullback square
\[ \begin{tikzcd}
U' \ar{r}{f'} \ar{d}{g'} & V' \ar{d}{g} \\
U \ar{r}{f} & V
\end{tikzcd} \]
of finite $\cT$-sets and $\underline{U'}$-colimit diagram $\overline{p}: \cK^{\underline{\rhd}} \to g'^* \cC$, the $\underline{V'}$-functor
\[ (f'_* \cK)^{\underline{\rhd}} \xto{\can} f'_* (\cK^{\underline{\rhd}}) \xto{f'_* \overline{p}} f'_* g'^* \cC \simeq g^* f_* \cC \xto{g^* F} g^* \cD \]
is a $\underline{V'}$-colimit diagram.\footnote{Using the compatibility of the parametrized join with restriction \cite[Lem.~4.4]{Exp2}, the canonical map $(f'_* \cK)^{\underline{\rhd}} \xto{\can} f'_* (\cK^{\underline{\rhd}})$ is defined to be the adjoint to $\epsilon^{\underline{\rhd}}: (f'^* f'_* \cK)^{\underline{\rhd}} \to \cK^{\underline{\rhd}}$ for $\epsilon$ the counit of the adjunction $f'^* \dashv f'_*$.}
\end{dfn}

\begin{proposition} \label{prop:DistributiveFunctorsPreserveSifted} Suppose that $\cT$ is orbital, let $f: U \to V$ be a map of finite $\cT$-sets, and let $\cC$ resp. $\cD$ be a $\underline{U}$ resp. $\underline{V}$-$\infty$-category. Suppose that $F: f_* \cC \to \cD$ is a $\underline{V}$-distributive $\underline{V}$-functor. Then $F$ strongly preserves $\underline{V}$-sifted $\underline{V}$-colimits.
\end{proposition}
\begin{proof} Since the property of parametrized distributivity is stable under base-change, it suffices to show that $F$ preserves $\underline{V}$-sifted $\underline{V}$-colimits. Without loss of generality, we may also suppose that $V$ is an orbit. Let $\cK$ be a $\underline{V}$-sifted $\underline{V}$-$\infty$-category and suppose that $\overline{p}: \cK^{\underline{\rhd}} \to f_\ast \cC$ is a $\underline{V}$-colimit diagram. Then because the counit map $f^\ast f_\ast \cC \simeq \underline{\Fun}_{\underline{U}}(\underline{U} \times_{\underline{V}} \underline{U},\cC) \to \cC$ is given by restriction along the diagonal $\underline{U} \to \underline{U} \times_{\underline{V}} \underline{U}$, the adjoint map $(\cK_{\underline{U}})^{\underline{\rhd}} \to \cC$ is a $\underline{U}$-colimit diagram. Since $F$ is $\underline{V}$-distributive, the $\underline{V}$-functor
\[ \overline{\psi}: (f_\ast f^\ast \cK)^{\underline{\rhd}} \to f_\ast (f^\ast \cK^{\underline{\rhd}}) \to f_\ast \cC \xrightarrow{F} \cD  \]
is a $\underline{V}$-colimit diagram. Since $\cK$ is $\underline{V}$-sifted, the unit $\underline{V}$-functor $\delta: \cK \to f_\ast f^\ast \cK \simeq \underline{\Fun}_{\underline{V}}(\underline{U}, \cK)$ is $\underline{V}$-cofinal, so $\overline{\psi \circ \delta}$ is also a $\underline{V}$-colimit diagram. But this composite is homotopic to $F \circ \overline{p}$ via the triangle identity for $f^\ast \dashv f_\ast$, proving the claim.
\end{proof}

We will use \cref{prop:DistributiveFunctorsPreserveSifted} together with \cref{thm:ParametrizedLimitsAndColimitsInSectionCategories} in \cite{paramalg} to show e.g. that given a $\cT$-distributive $\cT$-symmetric monoidal $\cT$-$\infty$-category $\cC$ (such as the $G$-$\infty$-category of $G$-spectra equipped with the Hill--Hopkins--Ravenel norms when $\cT = \OO_G$), the forgetful $\cT$-functor from the $\cT$-$\infty$-category of $\cT$-commutative algebras in $\cC$ to $\cC$ creates all $\cT$-sifted $\cT$-colimits.

\section{Universal constructions}

In this section, we introduce a few more universal constructions in addition to that of $\cT$-presheaves that formally adjoin smaller classes of $\cT$-colimits. We begin with the following lemma.

\begin{lemma} \label{lem:PresheafFunctorialityCofinalConstant} Let $f: \cC \to \cD$ be a functor of small $\infty$-categories and let $F: \PShv(\cC) \to \PShv(\cD)$ be the unique colimit-preserving functor that extends $j \circ f$. Then for every presheaf $\varphi \in \PShv(\cC)$, the induced functor
\[ \cC \times_{\PShv(\cC)} \PShv(\cC)^{/\varphi} \to \cD \times_{\PShv(\cD)} \PShv(\cD)^{/F(\varphi)} \]
is cofinal.
\end{lemma}
\begin{proof}
We verify the hypotheses of Joyal's cofinality theorem. Let $\tau = [d,\gamma: j(d) \to F(\varphi)]$ be an object of $\cD \times_{\PShv(\cD)} \PShv(\cD)^{/F(\varphi)}$. We want to show that the $\infty$-category
\[ \cE \coloneqq (\cC \times_{\PShv(\cC)} \PShv(\cC)^{/\varphi}) \times_{(\cD \times_{\PShv(\cD)} \PShv(\cD)^{/F(\varphi)})} (\cD \times_{\PShv(\cD)} \PShv(\cD)^{/F(\varphi)})_{\tau/} \]
is weakly contractible. Consider the commutative diagram
\[ \begin{tikzcd}
\cE \ar{r}{G'} \ar{d}{\pi'} & (\cD \times_{\PShv(\cD)} \PShv(\cD)^{/F(\varphi)})_{\tau/} \ar{d}{\pi} \\
\cD_{d/} \times_\cD (\cC \times_{\PShv(\cC)} \PShv(\cC)^{/\varphi}) \ar{r}{G} \ar{d} & \cD_{d/} \times_\cD (\cD \times_{\PShv(\cD)} \PShv(\cD)^{/F(\varphi)}) \ar{r} \ar{d} & \cD_{d/} \ar{d} \\
\cC \times_{\PShv(\cC)} \PShv(\cC)^{/\varphi} \ar{r} & \cD \times_{\PShv(\cD)} \PShv(\cD)^{/F(\varphi)} \ar{r} & \cD.
\end{tikzcd} \]
Observe that since $F(\varphi) \simeq \colim (\cD \times_{\PShv(\cD)} \PShv(\cD)^{/F(\varphi)} \to \PShv(\cD))$ and $\Map_{\PShv(\cD))}(j(d),-): \PShv(\cD) \to \Spc$ preserves colimits, we have that $\cD_{d/} \times_\cD (\cD \times_{\PShv(\cD)} \PShv(\cD)^{/F(\varphi)})$ is weakly homotopy equivalent to $F(\varphi)(d)$. Likewise, since $F$ preserves colimits, the functor $G$ is a weak homotopy equivalence. By \cref{lem:KanFibration}, the functor $\pi$ is a Kan fibration. By right properness of the Quillen model structure on simplicial sets, we deduce that $G'$ is a weak homotopy equivalence, hence $W$ is weakly contractible.
\end{proof}

\begin{lemma} \label{lem:KanFibration} Let $\psi: X \to Y$ be a right fibration and let $p: K \to X$ be a functor. Then the induced functor
\[ X_{p/} \to Y_{\psi p/} \times_Y X \]
is a Kan fibration.
\end{lemma}
\begin{proof}
Let $n>0$ and $\iota: A = \Lambda^n_i \to B = \Delta^n$ be a horn inclusion. We need to solve the lifting problem
\[ \begin{tikzcd}
A \ar{r} \ar{d}{\iota} & X_{p/} \ar{d} \\
B \ar{r} \ar[dotted]{ru} & Y_{\psi p/} \times_Y X
\end{tikzcd} \text{ which transposes to }
\begin{tikzcd}
K \star A \bigcup_{A} B \ar{r} \ar{d}{\iota'} & X \ar{d}{\psi} \\
K \star B \ar{r} \ar[dotted]{ru} & Y.
\end{tikzcd} \]
If $i<n$ so that $\iota$ is left anodyne, then by \cite[Lem.~2.1.2.3]{HTT}, $\iota'$ is inner anodyne, and if $i>0$ so that $\iota$ is right anodyne, then by the opposite of \cite[Lem.~2.1.2.4]{HTT}, $\iota'$ is right anodyne. Therefore, the dotted lift exists.
\end{proof}

\begin{definition} \label{dfn:fiberwisePresheaves} Let $\cC$ be a $\cT$-$\infty$-category. We define the \emph{fiberwise $\cT$-$\infty$-category of presheaves of $\cC$} to be the full subcategory
$$\underline{\PShv}_\cT^{\fb}(\cC)  \subset \underline{\PShv}_\cT(\cC)$$
whose fiber over each object $t \in \cT$ is the full subcategory $\PShv(\cC_t)$ of $\underline{\PShv}_\cT(\cC)_t \simeq \PShv(\cC_{\underline{t}}^{\textrm{v}})$, embedded via left Kan extension along the fully faithful inclusion $\cC_t^{\op} \subset \cC_{\underline{t}}^{\vop}$.
\end{definition}

\begin{remark} In \cref{dfn:fiberwisePresheaves}, we note that $\underline{\PShv}_\cT^{\fb}(\cC)$ is a full $\cT$-subcategory of $\underline{\PShv}_\cT(\cC)$, i.e., it is a sub-cocartesian fibration over $\cT^{\op}$. Indeed, the existence of the $\cT$-Yoneda embedding $j_\cT$ as a $\cT$-functor implies that for any morphism $\alpha: s \to t$ in $\cT$, the diagram
\[ \begin{tikzcd}
\cC_t \ar{r} \ar{d}{\alpha^{\ast}} & \cC_{\underline{t}}^{\textrm{v}} \ar{r}{j} & \underline{\PShv}_\cT(\cC)_t \simeq \PShv(\cC_{\underline{t}}^{\textrm{v}}) \ar{d}{\overline{\alpha}^{\ast}} \\
\cC_s \ar{r} & \cC_{\underline{s}}^{\textrm{v}} \ar{r}{j} & \underline{\PShv}_\cT(\cC)_s \simeq \PShv(\cC_{\underline{s}}^{\textrm{v}})
\end{tikzcd} \]
commutes, where $\overline{\alpha}^{\ast}$ is given by restriction along $\cC_{\underline{s}}^{\vop} \to \cC_{\underline{t}}^{\vop}$. Since the inclusions $\PShv(\cC_t) \subset \PShv(\cC_{\underline{t}}^{\textrm{v}})$ and $\PShv(\cC_s) \subset \PShv(\cC_{\underline{s}}^{\textrm{v}})$ along with $\overline{\alpha}^{\ast}$ are colimit-preserving, we have a factorization of the outer rectangle as
\[ \begin{tikzcd}
\cC_t \ar{r}{j} \ar{d}{\alpha^{\ast}} & \PShv(\cC_t) \ar{r} \ar{d}{\overline{\alpha}^{\ast}|_{\PShv(\cC_t)}} & \underline{\PShv}_\cT(\cC)_t \simeq \PShv(\cC_{\underline{t}}^{\textrm{v}}) \ar{d}{\overline{\alpha}^{\ast}} \\
\cC_s \ar{r}{j} & \PShv(\cC_s) \ar{r} & \underline{\PShv}_\cT(\cC)_s \simeq \PShv(\cC_{\underline{s}}^{\textrm{v}}),
\end{tikzcd} \]
which both establishes the claim and also identifies $\overline{\alpha}^{\ast}|_{\PShv(\cC_t)}$ with the prolongation of $j \circ \alpha^{\ast}$ obtained via the universal property of $\PShv(\cC_t)$. The $\cT$-Yoneda embedding then restricts to a $\cT$-functor $j^{\fb}_\cT: \cC \to \underline{\PShv}_\cT^{\fb}(\cC)$.
\end{remark}

\begin{definition} \label{def:cofinalconstant} Let $\cK$ be a $\cT$-$\infty$-category. Then $\cK$ is \emph{cofinal-constant} (\emph{cc}) if for all morphisms $\alpha: s \to t$ in $\cT$, the restriction functor $\alpha^{\ast}: \cT_t \to \cT_s$ is cofinal.

We say that a $\cT$-$\infty$-category $\cC$ is \emph{cc $\cT$-cocomplete} if $\cC$ strongly admits all cc $\cT$-colimits. If $\cC$ and $\cD$ are cc $\cT$-cocomplete, we will let $\underline{\Fun}^{cc}_\cT(\cC,\cD)$ denote the full $\cT$-subcategory of $\underline{\Fun}_\cT(\cC,\cD)$ whose fiber over each $t \in \cT$ is spanned by those $\cT^{/t}$-functors that strongly preserve all cc $\cT^{/t}$-colimits.

More generally, if $\CMcal{K}$ is a collection of small simplicial sets, then we have analogous definitions of $\CMcal{K}$-cc $\cT$-$\infty$-categories, $\CMcal{K}$-cc $\cT$-cocompleteness and $\underline{\Fun}^{\operatorname{\CMcal{K}-cc}}_\cT(\cC,\cD)$, where we suppose that the collection of $\cT$-diagrams in question are cofinal-constant and have fibers in $\CMcal{K}$.
\end{definition}

\begin{proposition} \label{prop:ccEqualsConstant} Let $\cC$ be a $\cT$-$\infty$-category. Then $\cC$ is cc $\cT$-cocomplete if and only if $\cC$ strongly admits all constant $\cT$-colimits. Similarly, if $\cC$ and $\cD$ are cc $\cT$-cocomplete, then a $\cT$-functor $F: \cC \to \cD$ strongly preserves all cc $\cT$-colimits if and only if $F_t$ preserves all colimits for all $t \in \cT$.
\end{proposition}
\begin{proof} We prove the first assertion about $\cC$; the second assertion about $F$ will then follow immediately. The `only if' implication is obvious. Conversely, suppose $\cC$ strongly admits all constant $\cT$-colimits. Let $t \in \cT$ and let $\cK$ be a cc $\cT^{/t}$-$\infty$-category. We have the $\cT^{/t}$-functor
$$\psi: \cK_t \times (\cT^{/t})^{\op} \to \cK$$
given as the cocartesian extension of the inclusion of the fiber $\cK_t \subset \cK$ over the initial object $\id_t \in (\cT^{/t})^{\op}$. By assumption, for all morphisms $\alpha: s \to t$ in $\cT$, the functor $\psi_{\alpha} \simeq \alpha^{\ast}: \cK_t \to \cK_s$ is cofinal, so by \cite[Thm.~6.7]{Exp2}, for each $\cT^{/t}$-functor $f: \cK \to \cC_{\underline{t}}$, the induced $\cT$-functor $\psi^{\ast}: \cC^{(f,\cT^{/t})/} \to \cC^{(f \psi,\cT^{/t})/}$ is an equivalence. In particular, $\cC^{(f,\cT^{/t})/}$ admits a $\cT^{/t}$-initial object if and only if $\cC^{(f \psi,\cT^{/t})/}$ does. Therefore, $f$ extends to a $\cT^{/t}$-colimit diagram if and only if $f \psi$ does, which completes the proof. 
\end{proof}

Recall that $\cC$ strongly admits all constant $\cT$-colimits if and only if its fibers admit all colimits and its pushforward functors preserve all colimits. For example, $\underline{\PShv}_\cT^{\fb}(\cC)$ strongly admits all constant $\cT$-colimits, so by \cref{prop:ccEqualsConstant}, $\underline{\PShv}_\cT^{\fb}(\cC)$ is cc $\cT$-cocomplete.

\begin{proposition} \label{prop:FiberwisePresheavesUniversalProperty} Let $\cC$ be a small $\cT$-$\infty$-category and let $\cD$ be cc $\cT$-cocomplete. Then for any $\cT$-functor $f: \cC \to \cD$, the $\cT$-left Kan extension $F$ of $f$ along $j_\cT^{\fb}$ exists. Moreover, restriction along $j_\cT^{\fb}$
\[ (j_\cT^{\fb})^\ast: \underline{\Fun}^{cc}_\cT(\underline{\PShv}_\cT^{\fb}(\cC),\cD) \to \underline{\Fun}_\cT(\cC,\cD) \]
implements an equivalence of $\cT$-$\infty$-categories, with inverse given by $\cT$-left Kan extension.
\end{proposition}
\begin{proof} For any $\varphi \in \underline{\PShv}_\cT^{\fb}(\cC)_t \simeq \PShv(\cC_t)$, note that by \cref{lem:PresheafFunctorialityCofinalConstant}
$$\cC \times_{\underline{\PShv}^{\fb}_\cT(\cC)} \underline{\PShv}^{\fb}_\cT(\cC)^{/\underline{\varphi}} = \cC \times_{\underline{\PShv}^{\fb}_\cT(\cC)} \Ar_\cT(\underline{\PShv}^{\fb}_\cT(\cC)) \times_{\underline{\PShv}^{\fb}_\cT(\cC)} \underline{\varphi} \to \underline{\varphi} \stackrel{\simeq}{\twoheadrightarrow} (\cT^{/t})^\op$$
is a cofinal-constant $\cT^{/t}$-$\infty$-category. Therefore, by the pointwise formula for $\cT$-left Kan extensions \cite[Thm.~10.3]{Exp2}, $F = (j_\cT^{\fb})_! f$ exists and is computed by $F_t \simeq j_! f_t$, so $F_t$ preserves all colimits. Furthermore, given a $\cT$-functor $G: \underline{\PShv}_\cT^{\fb}(\cC) \to \cD$ such that $G_t$ preserves colimits for all $t \in \cT$, since $j_! j^{\ast} G_t \xrightarrow{\simeq} G_t$ it follows that $(j^{\fb}_\cT)_! (j^{\fb}_\cT)^{\ast} G \xrightarrow{\simeq} G$ from the pointwise formula. By the same logic as \cite[Cor.~10.7]{Exp2}, we thus obtain a $\cT$-adjunction
\[ \adjunct{(j_\cT^{\fb})_!}{\underline{\Fun}_\cT(\cC,\cD)}{\underline{\Fun}_\cT(\underline{\PShv}_\cT^{\fb}(\cC),\cD)}{(j_\cT^{\fb})^\ast} \]
in which $(j_\cT^{\fb})_!$ is $\cT$-fully faithful with essential image $\underline{\Fun}^{cc}_\cT(\underline{\PShv}_\cT^{\fb}(\cC),\cD)$.
\end{proof}

\begin{variant} \label{var:FiberwisePresheaves} Let $\cC$ be a small $\cT$-$\infty$-category. For a collection $\CMcal{K}$ of small simplicial sets, let $\underline{\PShv}^{\fbK}_\cT(\cC) \subset \underline{\PShv}^{\fb}_\cT(\cC)$ be the full subcategory whose fiber over each $t \in \cT$ is given by $\PShv^{\cK}(\cC_t) \subset \PShv(\cC_t)$ \cite[Prop.~5.3.6.2]{HTT}, and let $j^{\fbK}_\cT$ denote the factorization of the $\cT$-Yoneda embedding through $\underline{\PShv}^{\fbK}_\cT(\cC)$. Then by the universal property of $\PShv^{\cK}(-)$, $\underline{\PShv}^{\fbK}_\cT(\cC)$ is a sub-cocartesian fibration of $\underline{\PShv}^{\fb}_\cT(\cC)$ and hence a $\cT$-$\infty$-category. Note that the proof of \cite[Prop.~5.3.6.2]{HTT} shows that for a $\cK$-cocomplete $\infty$-category $\cD$, the equivalence $\Fun^{\cK}(\PShv^{\cK}(\cC_t),\cD) \xrightarrow{\simeq} \Fun(\cC_t,\cD)$ implemented by restriction has inverse given by left Kan extension. Thus, by the same proof as in \cref{prop:FiberwisePresheavesUniversalProperty}, we see that if $\cD$ is a $\cT$-$\infty$-category that is $\cK$-cc $\cT$-cocomplete, we have an $\cT$-adjunction
\[ \adjunct{(j^{\fbK}_\cT)_!}{\underline{\Fun}_\cT(\cC,\cD)}{\underline{\Fun}_\cT(\underline{\PShv}^{\fbK}_\cT(\cC),\cD)}{(j^{\fbK}_\cT)^{\ast}}\]
in which $(j^{\fbK}_\cT)_!$ is $\cT$-fully faithful with essential image $\underline{\Fun}^{\operatorname{\CMcal{K}-cc}}_\cT(\underline{\PShv}^{\fbK}_\cT(\cC),\cD)$.
\end{variant}

\begin{dfn} \label{dfn:Ind} If $\CMcal{K}$ is the collection of sifted resp. $\kappa$-filtered simplicial sets, we will write $\underline{\PShv}^{\Sigma}_\cT(\cC)$ and $j^{\Sigma}_\cT$ resp. $\underline{\Ind}^{\kappa}_\cT(\cC)$ and $j^{\kappa}_\cT$ for $\underline{\PShv}^{\fbK}_\cT(\cC)$ and $j^{\fbK}_\cT$. If $\kappa = \omega$, we will also write $\underline{\Ind}_\cT(\cC)$.
\end{dfn}

For the following lemma, note that if $\cT$ is orbital and $\cE$ is an $\infty$-category that admits finite products, then $\underline{\cE}_{T}$ admits finite $\cT$-products in view of \cite[Prop.~5.6]{Exp2} and the pointwise formula for right Kan extension.

\begin{lemma} \label{lem:ParamProductsAndCategoryOfElements} Suppose that $\cT$ is an orbital $\infty$-category. Let $\cC$ be a $\cT$-$\infty$-category that admits finite $\cT$-products and let $\cE$ be an $\infty$-category that admits finite products. Then under the equivalence $$(-)^{\dagger}: \Fun_\cT(\cC, \underline{\cE}_\cT) \xrightarrow{\simeq} \Fun(\cC,\cE)$$ of \cite[Prop.~3.10]{Exp2}, a $\cT$-functor $F: \cC \to \underline{\cE}_\cT$ preserves finite $\cT$-products if and only if $F^{\dagger}: \cC \to \cE$ sends cartesian edges to equivalences and $F^{\dagger}|_{\cC_t}$ preserves finite products for all $t \in \cT$. Moreover, if $\cT$ admits a final object $\ast$, then $(-)^{\dagger}$ restricts to an equivalence $\Fun^{\times}_\cT(\cC, \underline{\cE}_\cT) \xrightarrow{\simeq} \Fun^{\times}(\cC_{\ast},\cE)$.
\end{lemma}
\begin{proof} For the first statement, first note that $F: \cC \to \underline{\cE}_\cT$ preserves finite products fiberwise if and only if for all $\alpha: s \to t$ in $\cT$, $\ev_{\alpha} F_t: \cC_t \to \cE_{\cT^{/t}} \to \cE$ preserves finite products. But since $\ev_{\alpha} F_t \simeq \ev_{\id_s} F_s \alpha^{\ast}$, this occurs if and only if $F^{\dagger}|_{\cC_t}$ preserves finite products for all $t \in \cT$. Furthermore, by definition $F^{\dagger}$ inverts cartesian edges if and only if for all $\alpha: s \to t$ in $\cT$ and $x \in \cC_s$, the natural map $(F_t \alpha_{\ast} x)(\id_t) \xrightarrow{\simeq} (\alpha_{\ast} F_s x)(\id_t)$ is an equivalence. This shows the `only if' implication. Now let $\beta: u \to t$ be any morphism and write $s \times_t u \simeq \bigsqcup_{i \in I} o_i$ for $o_i \in \cT$ and a finite set $I$. For each $o_i$, let $\alpha_i: o_i \to u$ and $\beta_i: o_i \to s$ denote the implicit maps, so that $\beta^\ast \alpha_\ast x \simeq \prod_{i \in I} {\alpha_i}_{\ast} {\beta_i}^{\ast} x$ by our assumption that $\cC$ admits finite $\cT$-products. If we suppose that $F$ preserves finite products fiberwise and $F^\dagger$ inverts cartesian edges, we then have
\begin{align*}
(F_t \alpha_\ast x)(\beta) & \simeq (F_u \beta^{\ast} \alpha_\ast x)(\id_u) \simeq \prod_{i \in I} (F_u {\alpha_i}_\ast \beta_i^\ast x)(\id_u) \simeq \prod_{i \in I} ({\alpha_i}_\ast \beta_i^\ast F_s)(\id_u) \\
& \simeq (\beta^{\ast} \alpha_{\ast} F_s x)(\id_u) \simeq (\alpha_\ast F_s x)(\beta),
\end{align*}
which shows the `if' implication.

To prove the second statement, suppose now that $\cT$ has a final object $\ast$. First note that if we let $W$ denote the set of cartesian edges in $\cC$, then the composite $\cC_\ast \to \cC \to \cC[W^{-1}]$ is an equivalence of $\infty$-categories in view of \cite[Cor.~3.3.4.3]{HTT}. Now let $G: \cC \to \cE$ be a functor that inverts $W$ and suppose $G|_{\cC_{\ast}}$ preserves finite products. For any $t \in \cT$, let $\alpha_t: t \to \ast$ denote the unique morphism. If $\prod_{i \in I} x_i$ is a finite product in $\cC_t$, then we have a cartesian edge $\prod_{i \in I} {\alpha_t}_\ast x_i \to \prod_{i \in I} x_i$ in $\cC$ lifting $\alpha_t$ since ${\alpha_t}_\ast$ is a right adjoint, hence $G(\prod_{i \in I} x_i) \simeq \prod_{i \in I} G(x_i)$ and the claim is proven.
\end{proof}

\begin{theorem} \label{thm:SiftedCocompletionAsParamProductPreservingPresheaves} Suppose that $\cT$ is an orbital $\infty$-category and let $\cC$ be a $\cT$-$\infty$-category. Suppose that $\cC$ admits finite $\cT$-coproducts. Then the following statements obtain:
\begin{enumerate}
    \item We have an equality $$\underline{\PShv}^{\Sigma}_\cT(\cC) = \underline{\Fun}^{\times}_\cT(\cC^{\vop}, \underline{\Spc}_\cT)$$ as full $\cT$-subcategories of $\underline{\PShv}_\cT(\cC)$.
    \item The inclusion $\underline{\PShv}^{\Sigma}_\cT(\cC) \subset \underline{\PShv}_\cT(\cC)$ strongly preserves $\cT$-sifted $\cT$-colimits and admits a $\cT$-left adjoint $L$ such that $j^{\Sigma}_\cT \simeq L \circ j_\cT$.
    \item $\underline{\PShv}^{\Sigma}_\cT(\cC)$ is $\cT$-cocomplete, $j_\cT^{\Sigma}$ preserves finite $\cT$-coproducts, and if $\cD$ is any $\cT$-cocomplete $\cT$-$\infty$-category, then restriction along $j_\cT^{\Sigma}$ implements an equivalence
    \[ \underline{\Fun}_\cT^L(\underline{\PShv}^{\Sigma}_\cT(\cC),\cD) \xrightarrow{\simeq} \underline{\Fun}_\cT^{\sqcup}(\cC,\cD) \]
    with inverse given by $\cT$-left Kan extension.
\end{enumerate}

Similarly, if $\cC$ strongly admits $\cT$-$\kappa$-small $\cT$-colimits, then:
\begin{enumerate}
    \item $\underline{\Ind}_\cT^\kappa(\cC)$ equals the full $\cT$-subcategory $\underline{\Fun}^{\operatorname{\kappa-lex}}_\cT(\cC^{\vop}, \underline{\Spc}_\cT)$ of $\underline{\PShv}_\cT(\cC)$ whose fiber over $t \in \cT$ is spanned by those $\cT^{/t}$-presheaves that strongly preserve $\cT^{/t}$-$\kappa$-small $\cT^{/t}$-limits.
    \item The inclusion $\underline{\Ind}_\cT^\kappa(\cC) \subset \underline{\PShv}_\cT(\cC)$ strongly preserves $\cT$-$\kappa$-filtered $\cT$-colimits and admits a $\cT$-left adjoint $L$ such that $j^{\kappa}_\cT \simeq L \circ j_\cT$.
    \item $\underline{\Ind}_\cT^\kappa(\cC)$ is $\cT$-cocomplete, $j_\cT^{\kappa}$ strongly preserves $\cT$-$\kappa$-small $\cT$-colimits, and if $\cD$ is any $\cT$-cocomplete $\cT$-$\infty$-category, then restriction along $j_\cT^{\kappa}$ implements an equivalence
    \[ \underline{\Fun}_\cT^L(\underline{\Ind}_\cT^\kappa(\cC),\cD) \xrightarrow{\simeq} \underline{\Fun}_\cT^{\operatorname{\kappa-rex}}(\cC,\cD) \]
    with inverse given by $\cT$-left Kan extension.
\end{enumerate}
\end{theorem}
\begin{proof} In both cases, (1) is an immediate consequence of \cref{lem:ParamProductsAndCategoryOfElements} (together with the dual of \cref{prop:FiberwiseAndCoproductsGiveAll} in the second instance). Given \cref{thm:EquivalentConditionsSifted}, \cref{thm:EquivalentConditionsFiltered}, and \cref{var:FiberwisePresheaves}, the rest of the statements then follow formally as in the proof of \cite[Prop.~5.5.8.10]{HTT} and \cite[Prop.~5.5.8.15]{HTT}.
\end{proof}

\section{Appendix: Promonoidal Day convolution}

We record the following important technical lemma on flat fibrations and apply it to construct \emph{$\cO$-promonoidal Day convolution} with respect to a base $\infty$-operad $\cO$.

\begin{lemma} \label{lem:FlatCategoricalFibrationArrowLemma}
Let $\cB$ be an $\infty$-category with a factorization system $(\sL,\sR)$ and let $p: \cX \to \cB$ be a categorical fibration. Let $\Ar^L(\cB)$ denote the full subcategory of $\Ar(\cB)$ on those arrows in $\sL$ and consider the functor
\[ \pi = \ev_0 \circ \pr_1: \Ar^L(\cB) \times_{\cB} \cX \to \cB. \]
Suppose that
\begin{enumerate}
\item For every edge $e: a \to b$ in $\sL$ and $x \in \cX$ such that $p(x) = a$, there exists a $p$-cocartesian edge $x \to y$ covering $e$.
\item The pullback $\cX_{R} = \cX \times_{\cB} \cB_{R} \to \cB_{R}$ is a flat categorical fibration, where $\cB_{R} \subset \cB$ denotes the wide subcategory on those morphisms in $\sR$.
\end{enumerate}
Then $\pi$ is a flat categorical fibration.
\end{lemma}
\begin{proof}
We apply the criterion of \cite[Prop.~B.3.2]{HA} to show flatness. In other words, if we let $\sigma_0 = [a_0 \ra b_0 \ra c_0]$ be a $2$-simplex in $\cB$ and let
\[ \left( \begin{tikzcd}
a_0 \ar{r} \ar{d}{\alpha} & c_0 \ar{d}{\gamma} \\
a_1 \ar{r} & c_1
\end{tikzcd}, \: x \to z \right) \]
be an edge in $\Ar^L(\cB) \times_{\cB} \cX$ covering $a_0 \to c_0$ via $\pi$, then we need to show that
\begin{align*}
(\Ar^L(\cB) \times_{\cB} \cX)^{(\alpha,x)/ /(\gamma,z)}_{b_0} & \coloneq \{\sigma_0\} \times_{\cB^{a_0 / /c_0}} (\Ar^L(\cB) \times_{\cB} \cX)^{(\alpha,x)/ /(\gamma,z)} \\
& \simeq \Ar^L(\cB)^{\alpha/ /\gamma}_{b_0} \times_{\cB^{a_1//c_1}} \cX^{x//z}
\end{align*}
is weakly contractible.

As we noted in \cref{prop:InertCartesianFibration}(1), the functor $\ev_0: \Ar^L(\cB) \to \cB$ is a  cartesian fibration, with $\ev_0$-cartesian edges given by morphisms $f \to g$ such that the edge $f(1) \to g(1)$ is in $\sR$. Therefore, we may identify the full subcategory of $\Ar^L(\cB)^{\alpha/ /\gamma}_{b_0}$ spanned by the final objects with that spanned by objects of the form
\[ \begin{tikzcd}
a_0 \ar{r} \ar{d}{\alpha} & b_0 \ar{d}{\beta} \ar{r} & c_0 \ar{d}{\gamma} \\
a_1 \ar{r} & b_1 \ar{r} & c_1
\end{tikzcd} \]
in which $b_1 \to c_1$ is in $\sR$. Fix such a choice of final object $\sigma_{\bullet}$, and let
$$\theta: \Delta^1 \times \Ar^L(\cB)^{\alpha/ /\gamma}_{b_0} \to \Ar^L(\cB)^{\alpha/ /\gamma}_{b_0}$$
be the natural transformation recording the essentially unique homotopy of the identity functor to the constant functor at $\sigma_{\bullet}$ (i.e., the unit transformation of the associated localization functor). Also let
$$ \theta':  \Ar^L(\cB)^{\alpha/ /\gamma}_{b_0} \to \Fun(\Delta^1, \Ar^L(\cB)^{\alpha/ /\gamma}_{b_0}) \to \Fun'(\Delta^1, \cB^{a_1//c_1})$$
be the composite of the adjoint to $\theta$ and evaluation at the target. Here, $\Fun'$ denotes the full subcategory on objects $\tau = [a_1 \ra b_1' \ra b_1 \ra c_1]$ with $d_1 \tau = \sigma_1$ and such that $b'_1 \to b_1$ is in $\sL$; in other words, $d_0 \tau$ is the essentially unique factorization of $b'_1 \to c_1$ furnished by $(\sL,\sR)$. We then define a natural transformation
$$ \eta: \Delta^1 \times \Ar^L(\cB)^{\alpha/ /\gamma}_{b_0} \times_{\cB^{a_1//c_1}} \cX^{x//z} \to \Ar^L(\cB)^{\alpha/ /\gamma}_{b_0} \times_{\cB^{a_1//c_1}} \cX^{x//z}$$
as $\theta$ on the first factor and the adjoint to
$$\Ar^L(\cB)^{\alpha/ /\gamma}_{b_0} \times_{\cB^{a_1//c_1}} \cX^{x//z} \xto{(\theta',\id)} \Fun'(\Delta^1, \cB^{a_1//c_1}) \times_{\ev_0,\cB^{a_1//c_1}} \cX^{x//z} \xto{P} \Fun(\Delta^1, \cX^{x//z})$$
on the second factor, where $P$ is the cocartesian pushforward functor that on objects is given by
\[ \goesto{([a_1 \ra b'_1 \xrightarrow{e} b_1 \ra c_1], \: x \ra y \ra z)}{[x \ra e_! y \ra z]} \]
and rigorously defined as in \cite[Lem.~2.23]{Exp2}.

Let $L = \eta_1$ and observe that the essential image of $L$ is $\cX^{x// z}_{b_1} \coloneq \{\sigma_1\} \times_{\cB^{a_1//c_1}} \cX^{x//z}$. It is then straightforward to show that $\eta$ satisfies condition (3) of \cite[Prop~5.2.7.4]{HTT} so that $L$ is a localization functor. In particular, it suffices to show that $\cX^{x// z}_{b_1}$ is weakly contractible. Moreover, after choosing a $(\sL,\sR)$-factorization $[a_1 \xrightarrow{e} a_1' \ra b_1]$, we have a $p$-cocartesian lift $\overline{e}: x \to x'$ of $e$ by assumption (1), and by the universal property of $\overline{e}$ we have an equivalence
$$\cX^{x// z}_{b_1} \simeq (\cX_R)^{x'//z}_{a_1'}.$$
But hypothesis (2) then ensures that $(\cX_R)^{x'//z}_{a_1'}$ is weakly contractible, which completes the proof.
\end{proof}

Now let $\cO^{\otimes}$ be an $\infty$-operad and consider the factorization system given by the inert and active edges (cf. \cite[Def.~2.1.2.3]{HA} and \cite[Prop.~2.1.2.4]{HA}). Let $\cO^{\otimes}_{\act} \subset \cO^{\otimes}$ be the wide subcategory on the active edges.

\begin{dfn} \label{dfn:promonoidal}
Let $p: \cC^{\otimes} \to \cO^{\otimes}$ be a fibration of $\infty$-operads. We say that \emph{$p$ exhibits $\cC^{\otimes}$ as a $\cO$-promonoidal $\infty$-category} if the restricted functor $p_{\act} :\cC^{\otimes}_{\act} \to \cO^{\otimes}_{\act}$ is flat.
\end{dfn}

\begin{exm}
Suppose that $\cC^{\otimes}$ is a $\cO$-monoidal $\infty$-category, so that its structure map $p$ is a cocartesian fibration. $\cC^{\otimes}$ is then $\cO$-promonoidal since cocartesian fibrations are flat \cite[Exm.~B.3.4]{HA}.
\end{exm}

The following example was pointed out to us by Harpaz and shows that our earlier definition of symmetric promonoidal given as \cite[Def.~1.4]{BarwickGlasmanShah} was too restrictive.

\begin{exm} \label{exm:promonoidalFlat}
There exists examples of $\cO$-promonoidal $\infty$-categories $(\cC^{\otimes},p)$ such that $p$ itself is not flat. For instance, consider the $\infty$-operad $\MCom^{\otimes}$ that parametrizes modules over commutative algebras (cf. \cite[4.3]{Harpaz2019}). Then $\MCom^{\otimes}$ is symmetric promonoidal, but $p: \MCom^{\otimes} \to \FF_{\ast}$ is not flat. Indeed, let $\angs{n} = \{ 1, ..., n, + \}$ and consider the composition of maps of pointed finite sets
\[ h: \angs{3} \xto{f} \angs{2} \xto{g} \angs{1} \]
where $f(1) = 1, f(2) = 2, f(3) = 2$ and $g(1) = 1, g(2) = +$. Let $m$ be the object of $\MCom$ representing the module factor and consider the inert edge $e: (m,m,m) \to m$ over $h$. Then $e$ doesn't factor over $h = g \circ f$, so $p$ is not flat.
\end{exm}

\begin{rem} Let $(\cC^{\otimes},p)$ be a $\cO$-promonoidal $\infty$-category. Then if $\cC^{\otimes}$ is moreover \emph{corepresentable} in the sense that $p$ is locally cocartesian, we claim that $p$ is cocartesian so that $\cC^{\otimes}$ is $\cO$-monoidal. Indeed, by \cite[Prop.~1.5]{BarwickGlasmanShah} we see that $p_{\act}$ is cocartesian, and using the inert-active factorization system on $\cC^{\otimes}$ together with the decomposition of mapping spaces in $\cC^{\otimes}$ ensured by the definition of an $\infty$-operad \cite[Def.~2.1.1.10(2)]{HA}, it is not difficult to check that $p$ itself is cocartesian.
\end{rem}

We now generalize Lurie's construction of Day convolution \cite[Thm.~2.2.6.2]{HA}, which assumed that $\cC^{\otimes}$ was $\cO$-monoidal. Let $\Ar^{\inert}(\cO^{\otimes})$ denote the full subcategory of $\Ar(\cO^{\otimes})$ on the inert edges.

\begin{conthm} \label{conthm:DayConvolution}
Let $(\cC^{\otimes},p)$ be a $\cO$-promonoidal $\infty$-category. Consider the span of marked simplicial sets
\[ \begin{tikzcd}
(\cO^{\otimes}, \Inert) & (\Ar^{\inert}(\cO^{\otimes}) \times_{\ev_1, \cO^{\otimes}, p} \cC^{\otimes}, \Inert) \ar{r}{\pr_{\cC^{\otimes}}} \ar{l}[swap]{\ev_0} & (\cC^{\otimes}, \Inert) 
\end{tikzcd} \]
where the middle marking consists of those edges in $\Ar^{\inert}(\cO^{\otimes}) \times_{\cO^{\otimes}} \cC^{\otimes}$ whose source in $\cO^{\otimes}$ is inert and whose projection to $\cC^{\otimes}$ is inert. Then the functor
\[ (\ev_0)_* \circ (\pr_{\cC^{\otimes}})^* : \sSet^+_{/(\cC^{\otimes}, \Inert)} \to \sSet^+_{/(\cO^{\otimes}, \Inert)} \]
is right Quillen with respect to the operadic model structures of \cite[Prop.~2.1.4.6]{HA}. For a fibration $\cD^{\otimes} \to \cC^{\otimes}$ of $\infty$-operads, we then define the \emph{$p$-operadic coinduction} of $\cD^{\otimes}$ to be
\[ (\Nm_p \cD)^{\otimes} \coloneq (\ev_0)_* (\pr_{\cC^{\otimes}})^*(\cD^{\otimes},\Inert). \]
For a fibration $\cD^{\otimes} \to \cO^{\otimes}$ of $\infty$-operads, we define the \emph{Day convolution} (of $\cC^{\otimes}$ with $\cD^{\otimes}$ over $\cO^{\otimes}$) to be
\[ \widetilde{\Fun}_{\cO}(\cC, \cD)^{\otimes} \coloneq (\Nm_p p^* \cD)^{\otimes}. \]
\end{conthm}
\begin{proof}
It suffices to verify the hypotheses of \cite[Thm.~B.4.2]{HA}. For (1), $\ev_0$ is flat by \cref{lem:FlatCategoricalFibrationArrowLemma}. The remainder of the proof is now identical to that of \cite[Prop.~2.2.6.20(a)]{HA}; the only additional point to note is that the verification of (5) only uses that $\cC^{\otimes} \to \cO^{\otimes}$ admits $p$-cocartesian lifts over \emph{inert} edges in the base.
\end{proof}

\begin{rem}
It follows readily from the definition that the underlying $\infty$-category of the Day convolution $\widetilde{\Fun}_{\cO}(\cC, \cE)^{\otimes}$ is equivalent to the pairing construction $\widetilde{\Fun}_{\cO}(\cC, \cE)$ (\cref{conthm:PairingConstructionRecalled} with $\cT = \ast$).
\end{rem}

Given \cref{conthm:DayConvolution}, all the usual properties of Day convolution with this extra generality in the source variable then hold; we will give a comprehensive treatment of the parametrized theory in \cite{paramalg}. In particular, we have that for any fibration $\cD^{\otimes} \to \cO^{\otimes}$ of $\infty$-operads, the identity section
\[ \iota: \into{(\cD^{\otimes} \times_{\cO^{\otimes}} \cC^{\otimes}, \Inert)}{(\cD^{\otimes} \times_{\cO^{\otimes}} \Ar^{\inert}(\cO^{\otimes}) \times_{\cO^{\otimes}} \cC^{\otimes}, \Inert)} \]
is a homotopy equivalence in $\sSet^+_{/(\cC^{\otimes}, \Inert)}$, so for all fibrations $\cE^{\otimes} \to \cC^{\otimes}$ of $\infty$-operads, restriction along $\iota$ induces an equivalence of $\infty$-categories
\[ \Alg_{\cD/\cO}(\Nm_p \cE) \xto{\simeq} \Alg_{\cD \times_{\cO} \cC/\cC}(\cE). \]

Thus, we may think of the class of $\cO$-promonoidal $\infty$-categories as singling out the \emph{exponentiable} fibrations of $\infty$-operads over $\cO^{\otimes}$.

\bibliographystyle{amsalpha}
\bibliography{Gcats}

\end{document}